\documentclass[a4paper,10pt]{amsart}
\usepackage{amsmath}
\usepackage{amsthm}
\usepackage{amssymb}
\usepackage{amsfonts}
\usepackage{mathtools}
\usepackage{graphicx}
\usepackage{tikz-cd}
\usepackage{bbm}
\usepackage[margin=1.25in]{geometry}

\usepackage{graphicx}
\usepackage[labelformat=simple]{subcaption}

\usepackage[mathscr]{euscript}

\usepackage[backend=biber]{biblatex}
\renewbibmacro{in:}{}
\newcommand{\ZZ}{\mathbb{Z}}
\newcommand{\CC}{\mathbb{C}}
\newcommand{\GG}{\mathbb{G}}
\newcommand{\DD}{\mathbb{D}}
\newcommand{\NN}{\mathbb{N}}
\newcommand{\WW}{\mathbb{W}}
\newcommand{\FF}{\mathbb{F}}

\newcommand{\HH}{Hochschild\ }
\newcommand{\hur}{Hurewicz\ }
\newcommand{\Grpd}{\mbox{\textbf{Grpd}}}
\newcommand{\pshv}{\mathscr{PS}hv}
\newcommand{\stk}{\mathscr{S}tk}

\newcommand{\rrarrow}{  \hspace{.05cm}\mbox{\,\put(0,-2){$\rightarrow$}\put(0,2){$\rightarrow$}\hspace{.45cm}}}
\newcommand{\rrrarrow}{ \hspace{.05cm}\mbox{\,\put(0,-3){$\rightarrow$}\put(0,1){$\rightarrow$}\put(0,5){$\rightarrow$}\hspace{.45cm}}}
\newcommand{\Sing}{\mathscr{S}ing}
\newcommand{\sSet}{s\mathscr{S}et}
\newcommand{\bsSet}{bs\mathscr{S}et}
\newcommand{\Map}{\mbox{Maps}}
\newcommand{\calX}{\mathcal{X}}
\newcommand{\calY}{\mathcal{Y}}

\newcommand{\Fin}{\mathcal{F}in}
\newcommand{\oprd}{\mathbf{Prpd_{red}}}
\newcommand{\oprdinfty}{\mathbf{Prpd_{red}^\otimes}}
\newcommand{\oprdten}{\mathbf{Prpd}^\otimes}
\newcommand{\oprdO}{\mathcal{O}}
\newcommand{\oprdOten}{\mathcal{O}^\otimes}

\newcommand{\Mmod}{\mathcal{M}}
\newcommand{\Mmodbar}{\overline{\mathcal{M}}}
\newcommand{\Mmodhat}{\widehat{\mathcal{M}}}
\newcommand{\Mfr}{{\partial_+}\mathfrak{M}}
\newcommand{\Mbar}{\overline{\mathfrak{M}}}
\newcommand{\Mhat}{\widehat{\mathfrak{M}}}
\newcommand{\Mfrfull}{\mathfrak{M}}

\newcommand{\Mfrunst}{\partial_+\mathfrak{M}^{unst}}
\newcommand{\Mhatunst}{\widehat{\mathfrak{M}}^{unst}}
\newcommand{\Mfrfullunst}{\mathfrak{M}^{unst}}
\newcommand{\Mfrfullnop}{\mathfrak{M}^{nop}}
\newcommand{\Mbarunst}{\overline{\mathfrak{M}}^{unst}}
\newcommand{\Mfrunstth}{{\partial_+ \mathbb{M}}^{unst}}
\newcommand{\Mhatunstth}{\widehat{\mathbb{M}}^{unst}}

\newcommand{\mfr}{\partial_+\mathfrak{M}}

\newcommand{\mfrunst}{\partial_+ \mathfrak{M}^{nop}}

\newcommand{\mfrfull}{\mathscr{M}}
\newcommand{\mfrfullmod}{\mathscr{M}^{fr}}

\newcommand{\mfrfullnopair}{\mathscr{M}^{nop}}

\newcommand{\Mbarclsd}{\overline{\mathcal{M}}}
\newcommand{\mfrtilde}{\widetilde{\mathscr{M}}}

\DeclareMathAlphabet{\mathdutchcal}{U}{dutchcal}{m}{n}
\SetMathAlphabet{\mathdutchcal}{bold}{U}{dutchcal}{b}{n}
\DeclareMathAlphabet{\mathdutchbcal}{U}{dutchcal}{b}{n}

\newcommand{\mdfrfull}{\mathdutchcal{M}}

\newcommand{\mdfrfullunst}{\mathdutchcal{M}^{unst}}

\newcommand{\mdbar}{\overline{\mathdutchcal{M}}}
\newcommand{\mdbarunst}{\overline{\mathdutchcal{M}}^{unst}}

\newcommand{\mlfrfullmod}{f\mathcal{M}^{\mathbb{R}}}
\newcommand{\mlbar}{\overline{\mathcal{M}}}

\newcommand{\smallmfrfull}{M^{fr}}
\newcommand{\smallmhat}{\widehat{M}^{fr}}
\newcommand{\smallmbar}{\overline{M}}

\newcommand{\mmfrfull}{\partial_+\mathbb{M}}

\newcommand{\spot}{marked point\ }
\newcommand{\spots}{marked points\ }
\newcommand{\splitsurface}{spotted surface}
\newcommand{\splitsurfaces}{spotted surfaces}

\newcommand{\cfam}{\mathcal{C}}
\newcommand{\Nu}{\mathcal{V}}
\newcommand{\un}[1]{\underline{#1}}

\newcommand{\catc}{\mathcal{C}}
\newcommand{\catb}{\mathcal{B}}

\newcommand{\Top}{\mathbf{Top}}

\newcommand{\redTopSeq}{\mathbf{TopSeq}^{\mathbf{io}}}
\newcommand{\Galg}{\GG\mathbf{Alg}}
\newcommand{\reduced}{io-}
\newcommand{\rda}{ioda}
\newcommand{\symplectic}{symplectic\ }
\newcommand{\CFT}{TFT}
\newcommand{\delCFT}{$\partial_+$TFT}

\newtheorem{theorem}{Theorem}[section]
\newtheorem{lemma}[theorem]{Lemma}
\newtheorem{corollary}[theorem]{Corollary}
\newtheorem{proposition}[theorem]{Proposition}
\numberwithin{equation}{section}

\theoremstyle{remark}
\newtheorem{remark}[theorem]{Remark}
\newtheorem{notation}[theorem]{Notation}

\theoremstyle{definition}
\newtheorem{definition}[theorem]{Definition}

\title{A homotopical description of Deligne-Mumford compactifications}
\author{Yash Deshmukh}

\begin{document}

\begin{abstract}
We give a description of the Deligne-Mumford properad expressing it as the result of homotopically trivializing S1 families of annuli (with appropriate compatibility conditions) in the properad of smooth Riemann surfaces with parameterized boundaries. This gives an analog of the results of Drummond-Cole and Oancea--Vaintrob in the setting of properads. We also discuss a variation of this trivialization which gives rise to a new partial compactification of Riemann surfaces relevant to the study of operations on symplectic cohomology.
\end{abstract}

\maketitle

\setcounter{tocdepth}{1}	
\tableofcontents

\section{Introduction}
\subsection{Context and Motivation}\label{subsec:Context and Motivation}\label{subsec: context motivtn}
In this paper we study homotopical aspects of the problem of extending operations in the closed sector of a two-dimensional Topological Field Theory (\CFT) (in the sense of \cite[Section 2]{costello2009partition}) from the uncompactified moduli spaces of Riemann surfaces of all genera to their compactifications. 

The motivation for considering such extensions comes from a proposal of M. Kontsevich for the construction of the so-called \emph{categorical enumerative invariants} associated with a Calabi-Yau $A_\infty$-category, under suitable hypotheses (see \cite[\nopp 2.23]{kontsevich2008xi}, \cite[Section 11]{kontsevich2008notes}). The \HH chain complex of a proper (respectively, smooth) Calabi-Yau $A_\infty$-category carries a chain-level two-dimensional right (respectively, left) positive-boundary Topological Field Theory ({\CFT}) structure (see \cite{kontsevich2006notes}, \cite{costello2007topological} for the proper case and  \cite{kontsevich1smooth} for the smooth case). Recall that chain-level two-dimensional right (respectively, left) positive-boundary {\CFT}s are field theories with operations given by chains on the moduli spaces of Riemann surfaces in which each connected component has at least one input (respectively, output). The {\CFT} structure in particular induces a chain-level $S^1$-action on the \HH chains coming from the operations given by moduli spaces of annuli. Kontsevich proposed that under the assumption that this $S^1$-action is homotopically trivial, the {\CFT} structure should extend to include operations coming from chains on the Deligne-Mumford compactification of the moduli spaces of Riemann surfaces. The desired categorical enumerative invariants can then be defined using the induced action of chains on the Deligne-Mumford compactifications.

If the Calabi-Yau $A_\infty$-category is Morita equivalent to the Fukaya category of a symplectic manifold or the category of coherent sheaves on a Calabi-Yau manifold then, under nice circumstances (for example when the open-closed map associated with the Fukaya category is an isomorphism), the categorical enumerative invariants are expected to recover, respectively, the Gromov-Witten (GW) invariants and the Bershadsky-Cecotti-Ooguri-Vafa (BCOV) invariants of the underlying manifolds. In such cases \emph{homological mirror symmetry}, which asserts an equivalence between the Fukaya category and a $dg$-enhancement of the derived category of coherent sheaves of a Calabi-Yau mirror pair, can be used to deduce \emph{enumerative mirror symmetry}, which relates the GW and BCOV invariants of the mirror pair.

\begin{remark}
\begin{enumerate}
\item In what follows, we will mostly work with left positive-boundary {\CFT}s. This is because of our interest in examples coming from symplectic topology which, as explained below in Section \ref{subsec: intro SH QH}, admit only a left positive-boundary {\CFT} structure in-general. For this reason, below whenever we refer to a `positive-boundary {\CFT}' we will mean a \emph{left} positive-boundary {\CFT}, unless specified otherwise.
\item An alternate approach to defining the categorical enumerative invariants was carried out in \cite{costello2009partition} and further developed in \cite{caldararu2020categorical}. This approach bypasses the use of Deligne-Mumford spaces by working with a Batalin-Vilkovisky algebra constructed using the uncompactified moduli spaces.
\end{enumerate} 
\end{remark}

The genus $0$, $n$-to-$1$ part of the extension described above was formulated by Kontsevich as the following conjecture: The structure of an algebra over the framed little disk operad along with a homotopy trivialization of the $S^1$-action is equivalent to the structure of an algebra over the Deligne-Mumford operad. Recall that an algebra over the framed little disk operad is equivalent to the data of the genus $0$, $n$-to-$1$ part of a {\CFT}. Various versions of this conjecture were discussed in \cite{khoroshkin2013hypercommutative} and \cite{drummond2013minimal}, in the category of chain complexes over a characteristic $0$ field. The statement in the category of topological spaces was proved in \cite{drummond2014homotopically}. An extension to higher genus, $n$-to-$1$ operations was proved in \cite{oancea2020deligne} where the operad of framed little disks was replaced by the operad of Riemann surfaces with parametrized boundaries.

In this paper we consider an extension of the above conjecture from the genus 0, $n$-to-$1$ part of a {\CFT} to the entire connected, left positive-boundary part of the {\CFT}. In particular,  we shall consider operations coming from Riemann surfaces with possibly higher genus and with multiple inputs and multiple outputs. We do this using the language of `properads'.

\subsection{Properads} Properads are a generalization of operads. Recall that operads encode algebraic structures involving operations with multiple inputs and a single output. Generalizing this, as we shall discuss in Section \ref{sec: overview} below, properads encode algebraic structures involving operations that have multiple inputs and outputs. Moreover, these operations can be composed along multiple inputs and outputs. 

\begin{remark}
	In this paper, we will consider a modification of the notion of properads which we call `input-output' properads (\emph{\reduced properads}). These are analogous to usual properads except that there are no operations having $0$ inputs and $0$ outputs. Consequently, the compositions in properads resulting in operations with no inputs and no outputs are also omitted from the structure. The restriction to \reduced properads has been done for the purpose of technical simplification of some parts of the proof.\\
In the properads which concern us, this omission results in the loss of certain structures coming from operations parametrized by the moduli spaces of stable Riemann surfaces with no intputs and outputs. However, we do not expect the restriction to \reduced properads to be an essential one and expect our results to hold at the level of usual properads.
\end{remark}

We now briefly describe properads of Riemann surfaces that are relevant to us. See Section \ref{sec: moduli spaces} below for more details on these properads and the moduli spaces involved.\\
The data of operations in a {\CFT} coming from connected Riemann surfaces naturally presents itself as representation of a certain \mbox{properad} $\Mfrfull$. The spaces of operations in $\Mfrfull$ are given by the moduli spaces of Riemann surfaces having input and output boundary components, each of which is equipped with an analytic $S^1$-parameterization. We refer to $\Mfrfull$ as the \emph{{\CFT}-properad}.
There is a subproperad $\Mfr$ of the {\CFT}-properad which encodes the structure given by the connected part of a positive-boundary {\CFT}. This properad is obtained by considering only  the operations indexed by Riemann surfaces which have at least one output. We refer to this as the \emph{\delCFT-properad}.\\
Moreover, to encode the operations coming from the compactified Riemann surfaces we use the properad $\Mbar$, which we call the \emph{Deligne-Mumford properad}. The spaces of operations in $\Mbar$ are given by the moduli spaces of stable nodal Riemann surfaces with input and output parametrized boundaries.\\
Theorem \ref{hatthm} will be concerned with a variant of $\Mbar$, denoted $\Mhat$, based on a new partial compactification $\Mmodhat_{g,n_-,n_+}$ of the moduli spaces of Riemann surfaces with parametrized boundaries. This compactification is obtained by considering nodal Riemann surfaces with boundaries, which satisfy the condition that every irreducible component has at least one output boundary. We call it the \emph{\symplectic properad}. See Section \ref{subsec: intro SH QH} for more on this properad.\\
$\Mfrfullunst$, $\Mfrunst$, $\Mhatunst$ , and $\Mbarunst$ will be the subproperads of $\Mfrfull$, $\Mfr$, $\Mhat$, and $\Mbar$ consisting of (possibly nodal) Riemann surfaces in the unstable range 
\[\{ (g,n_-,n_+) | 2g-2+n_++n_-< 0 \}=\{ (0,0,1), (0,1,0), (0,1,1), (0,0,2), (0,2,0)\}.\]
These properads have no operations in aritites other than those described by the above set.  ($\Mfrunst$ and $\Mhatunst$ have no operations in arity $(1,0)$ and $(2,0)$ either).
Finally, $\Mfrfullnop$ will be the subproperad of $\Mfrfullunst$ obtained by excluding from $\Mhatunst$ the annuli with two input boundaries. (The superscript `nop' here stands for \emph{`no pairing'}. This is because we think of the operations corresponding to annuli with two inputs of as pairings.) 

The main results of the paper are statements about certain homotopy colimits of \reduced properads. In order to talk about the homotopy theory of these structures, we outline the construction of a model category structure on the category of topological \reduced properads. This is carried out in Section \ref{sec: model category str and cofib resolutions} below. 

\begin{remark}
The structure of a {\CFT} has the additional data of compositions coming from taking disjoint unions of Riemann surfaces. In order to capture this data, it is necessary to use a generalization of properads, known as \emph{PROPs}. In addition to the compositions similar to those in properads, known as the \emph{vertical compositions}, PROPs also have \emph{horizontal compositions}. These latter compositions can be used to encode the operation of taking disjoint unions of Riemann surfaces. 
We however choose to work with properads instead of PROPs, since the homotopy theory of properads is easier to handle than that of  PROPs.
\end{remark}

\subsection{Statements of the Theorems}
Our main results are as follows:
\begin{theorem}\label{frtofullthm}
\begin{enumerate}
	\item $\Mfrfull$ is the homotopy colimit of the following diagram in the category of \reduced properads:
	\begin{equation}\label{dig: frtofullthm1}
	\begin{tikzcd}
		\Mfrunst \ar{r}\ar{d} & \Mfrfullnop\\
		\Mfr
	\end{tikzcd}
	\end{equation}
	\item  $\Mbar$ is the homotopy colimit of the following diagram in the category of \reduced properads:
	\begin{equation}\label{dig: frtofullthm2}
	\begin{tikzcd}
		\Mfrfullunst \ar{r}\ar{d} & \Mbarunst \\
		\Mfrfull
	\end{tikzcd}
	\end{equation}
\end{enumerate}
\end{theorem}
The first part describes the extension from a positive-boundary {\CFT} to a {\CFT}. This is achieved by the introduction of a \emph{trace} operation (the point in $\Mfrfullnop(1,0)$) dualizing the \emph{cotrace} operation (the point in $\Mfrfullnop(0,1)$).

The second part describes the extension of a {\CFT} to an algebra over the Deligne-Mumford properad. An action of properad $\Mfrfull$ in particular induces an action of the subproperad of disks and annuli. The second statement above can be interpreted as saying that, homotopically, an extension of this action to the properad $\Mbarunst$ disks and possibly nodal annuli determines an extension of the {\CFT} structure to an algebra over the Deligne-Mumford properad.

Note that the spaces of operations in $\Mfrfullunst$ in arities $(1,1), (0,1),$ and $(1,0)$ are given by moduli spaces of annuli and are all homotopy equivalent to $S^1$. A homotopically essential family can be obtained by fixing an annulus and varying the parametrization of one of the ends by rotations. On the other hand, the corresponding space of operations in $\Mbarunst$, given by moduli spaces of possibly nodal annuli, are all contractible. Operations in arity $(0,1)$ and $(1,0)$ in both $\Mfrfullunst$ and $\Mbarunst$ are given by points.\\ 
An extension of $\Mfrfullunst$-action to an $\Mbarunst$-action thus corresponds to a homotopy trivialization of operations corresponding to the $S^1$-families in arities $(1,1), (0,1)$, and $(1,0)$, such that the trivializations are compatible with properad compositions with other $S^1$-families, as well as with the disks in arities $(1,0)$ and $(0,1)$. 

We also prove Theorem \ref{hatthm} which describes an extension of the positive-boundary {\CFT}-properad to the {\symplectic} properad. The motivation for studying such extensions comes from looking at the examples coming from symplectic cohomology (see \ref{subsec: intro SH QH} below).
\begin{theorem}\label{hatthm}
	$\Mhat$ is the homotopy colimit of the following diagram in the category of \reduced properads
	\begin{equation}
	\begin{tikzcd}
		\Mfrunst \ar{r}\ar{d} & \Mhatunst \\
		\Mfr
	\end{tikzcd}
	\end{equation}\label{dig: hatthm2}
\end{theorem}
Theorem \ref{hatthm} can be interpreted as saying that the data needed for the extension of an $\Mfr$-action to an $\Mhat$-action is precisely the data of extension of the induced action of the subproperad $\Mfrunst$ to $\Mhatunst$. Note that in this case the space of arity $(0,2)$ operations of $\Mhatunst$ is again contractible, however, the space of operations in arity $(1,1)$ is homotopy equivalent to $S^1$. The extension from $\Mfrunst$ to $\Mhatunst$ thus correponds to homotopy trivializing $S^1$-family of coparings coming from operations in airty $(0,2)$, leaving the operations in arity $(1,1)$ unchanged.\\
In the work under preparation \cite{rezchikov2020generalizations}, it is shown that such trivializations are related to conjectures of Kontsevich regarding generalized versions of the categorical Hodge-de Rham degeneration for smooth and for proper dg-categories. These conjectures were proved in \cite{rezchikov2020generalizations} for certain classes of Fukaya categories, but are known to be false in general \cite{efimov2020categorical}. 

\subsection{Symplectic and Quantum cohomology}\label{subsec: intro SH QH}
Symplectic cohomology and Quantum cohomology provide prototypical examples of the extensions described in the above theorems. 

Symplectic cohomology $SH^*(X)$ is an invariant associated with a symplectic manifold $(X,\omega)$ with bounded geometry. $SH^*(X)$ is known to admit a topological quantum field theory (TQFT) structure. This was introduced in \cite{seidel2006biased}, with a detailed construction carried out in \cite{ritter2013topological}. This structure is expected to lift to a chain-level action of the moduli spaces of Riemann surfaces with parametrized boundaries.\\
This structure is defined in terms of counts of maps $u$ from a Riemann surface $(\Sigma, j)$ to $X$ which satisfy a perturbed Cauchy-Riemann equation:
\begin{equation}\label{jholeq} (du- Z \otimes \beta)^{0,1}=0 \end{equation}
Here $Z$ and $\beta$ are, respectively, auxiliary choices of a  Hamiltonian vector field on $X$ and a $1$-form on $\Sigma$, both having a certain standard form near the boundaries of $\Sigma$. When $X$ is not necessarily compact, we must in addition require that $d\beta \leq 0$. This is necessary to prevent a sequence of such maps from escaping to infinity, which in turn is needed to ensure compactness of the moduli space of these maps. Note that this condition imposes the restriction that $n_+ \geq 0$. Thus, in general, such operations can only be defined for surfaces which have at least one output and therefore we only expect a chain-level $\Mfr$-action on $SH^*(X)$. In particular, trace operations coming from disks with an input boundary are not defined on $SH^*(X)$ in general.

The definition of these operations can be extended to include maps from nodal curves which carry a $1$-form $\beta$ on its non-singular locus satisfying $d\beta \leq 0$ and vanishing near the nodes. Curves appearing in the compactification $\Mmodhat_{g,n_-,n_+}$ satisfy this condition. The chain-level action of $\Mfr$ on the symplectic cohomology is thus expected to extend to an action of the properad $\Mhat$. Theorem \ref{hatthm} provides the extension statement at the level of underlying properads.\\
The action of the properad $\Mfr$ includes the data a of ccopairings induced by the action of cylinders with both ends outputs. 
According to  Theorem \ref{hatthm}, the homotopy trivialization of this family of ccopairings, given by the construction of operations corresponding to nodal annuli with two outputs, is the homotopically essential data providing the extension from the $\Mfr$-action to the $\Mhat$-action.

In the case when $X$ is compact, the condition $d\beta \leq 0$ may be dropped. As a consequence, operations can be defined using Riemann surfaces without any restriction on the number of outputs and thus the chain level $\Mfr$-action is expected to extend to an $\Mfrfull$-action. The corresponding extension statement at the level of underlying properads is given by Theorem \ref{frtofullthm}. The compactness of the ambient manifold $X$ is reflected in the properad-level statement as the existence of the trace operations indexed by disks with an input boundary which dualize the cotrace operations mentioned in the previous paragraph.\\
Moreover, in this case, similar to the extension from $\Mfr$ to $\Mhat$ described above, the action of $\Mfrfull$ is expected to extend to an action of \emph{all} stable curves in $\Mbar$. As a result, one would obtain a chain-level action of the properad $\Mbar$ on $SH^*(X)$. This corresponds to the extension from $\Mfrfull$ to $\Mbar$ described in second part of Theorem \ref{frtofullthm}. 

Recall that for a closed symplectic manifold $X$, the symplectic cohomology of $X$ is isomorphic to its quantum cohomology $QH^*(X)$. The quantum cohomology itself is expected to admit a chain-level action of a properad $\Mbarclsd$: the spaces of operations in $\Mbarclsd$ are given by the usual Deligne-Mumford moduli spaces of stable boundary-less nodal Riemann surfaces with input-output marked points. The compositions are given by simply concatenating stable curves along suitable marked points. The $\Mbarclsd$ action is defined by counting pseudo-holomorphic maps from stable (boundary-less) Riemann surfaces with marked points into $X$, in a manner similar to the $\Mfr$-action on $SH^*(X)$.\\
This action is equivalent to the data of a chain-level lift of the \emph{Cohomological Field Theory} (CohFT) structure on $H^*(X)$ (\cite{kontsevich1994gromov}). The properad $\Mbarclsd$ is weakly homotopy equivalent to $\Mbar$ (see Section \ref{subsec: fin dim htpy types}) and the isomorphism between $SH^*(X)$ and $QH^*(X)$ should intertwine these actions. Using Theorem \ref{frtofullthm}(2) it follows that the data of the CohFT structure on quantum cohomology of $X$ is homotopically equivalent to the data of the $\Mfrfull$-action along with a (suitably compatible) homotopy trivializations of the operations coming from the $S^1$-families of annuli in $\Mfrfullunst$. 

\subsection{Properads in Stacks} 
It is well know that the moduli spaces of stable curves admit natural lifts to topological stacks. As a consequence, to discuss the above statements, we are required to deal with properads valued in topological stacks. This introduces an additional layer of complexity: since topological stacks form a 2-category, it is natural to consider properads in which compositions are associative only up to certain natural equivalences. To discuss the homotopy theory of such properads, instead of dealing with the homotopy theory of $2$-properads, we find it convenient to use the language of $\infty$-properads. In, Section \ref{sec: infty properad} we outline this formalism and describe how $\Mbar$ can be viewed as an $\infty$-properad.\\
However, for the most part we carry out our discussion without encountering this issue. Specifically, throughout Sections \ref{sec: overview} to \ref{fulltcftsec}, we only work with properads in topological spaces. In Section \ref{mbarsec}, which essentially is the only place where we deal with a properad in stacks, our computation of homotopy colimit yields a topological properad with the property that for any $(n_-,n_+)$ the space of operations with $n_-$ inputs and $n_+$ outputs has the same weak homotopy type as the stack $\Mmodbar_{n_-,n_+}$, in the sense described in Section \ref{sec: infty properad}. The results outlined in Section \ref{sec: infty properad} can then be used to argue that this homotopy colimit computation carried out in the category of topological properads can be upgraded to a homotopy colimit \emph{of $\infty$-properads}, with the resulting colimit homotopy equivalent (\emph{as an $\infty$-properad}) to $\Mbar$.

\subsection{Outline of the paper}
We start by giving a brief overview of properads in Section \ref{sec: overview}. Here we introduce \reduced properads as algebras over a monad $\GG$. 
We also include a discussion of pushouts in the category of topological \reduced properads.\\
In Section \ref{sec: model category str and cofib resolutions} we describe a model structure on the category of topological \reduced properads using the description of properads as algebras over the monad $\GG$. We then use the bar construction associated with monad $\GG$ to construct cofibrant resolutions of properads, and provide a description of homotopy pushouts of properads. Most of the technical details required to prove the statements appearing here are deferred to Appendix \ref{apndx: Cofibrant Resolutions via The Bar Construction}.\\
In Section \ref{sec: moduli spaces} we give the precise definitions of the properads which appear in the main theorems and of the moduli spaces on which these properads are based.\\
We then carry out the proof of Theorem \ref{hatthm} in Section \ref{sec: hatthem pf pt1}. 
The technical heart of the proof is the constructions in Section \ref{prfcondfiltrtop}, where local weak equivalences are proved by exhibiting explicit simplicial homotopies. These homotopies are the simplicial analogues of the geometric homotopies constructed in \cite{oancea2020deligne} for operads. Additional work is required in our case as we work with properads which have compositions indexed by directed graphs more general than trees and also because we are required to work with cofibrant resolutions which are of a simplicial nature as opposed to the geometric ones used in \cite{oancea2020deligne}\\
In Section \ref{fulltcftsec} we prove the first part of Theorem \ref{frtofullthm}.  The pushout is computed with $\Mfrfullunst$ and $\Mfrfull$ replaced by homotopy equivalent properads of \splitsurfaces\  (see \ref{subsec: splitsurf}). The theorem is then proved by the constructions similar to those in the proof of Theorem \ref{hatthm}, with the weighted \spot playing a role analogous to nodes there.\\ 
In Section \ref{mbarsec} the proof of the second part of Theorem \ref{frtofullthm} is provided. The proof in this case is obtained by combining the constructions used in the proofs of Theorem \ref{hatthm} and the first part of Theorem \ref{frtofullthm}. After a few modifications (see Remarks \ref{rmk: seq of nodes} and \ref{rmk: mbar no output}), the proof follows a scheme very similar to that in the proof of Theorems \ref{hatthm} and \ref{frtofullthm}(1).\\
Finally, in Section \ref{sec: infty properad} we outline a formalism for $\infty$-\reduced properads and show that the homotopy colimits computed in earlier sections in the category of topological properads remain valid in the context of $\infty$-properads.

\subsection*{Acknowledgement}
I thank my advisor Mohammed Abouzaid for his constant support and encouragement. I would also like to thank Andrew Blumberg, Sheel Ganatra, and Ezra Getzler for helpful conversations. Finally, I am grateful to Lea Kenigsberg, Juan Muñoz-Echániz,  Alex Pieloch, and Mohan Swaminathan for useful discussions at various points during the preparation of this article.

\section{An Overview of Properads}\label{sec: overview}

\begin{notation}
	Let $\Top$ denote the category of compactly generated Hausdorff topological spaces. Let $\redTopSeq$ denote the category $\Top^{\NN \times \NN \setminus (0,0)}$ of bi-sequences of topological spaces indexed by $\NN \times \NN \setminus (0,0)$. Here the superscript \emph{`io'}, short for \emph{`input-output'}, refers to the fact that these sequences do not contain a $(0,0)$ component. Below, we will often refer to these sequences as \emph{topological \reduced sequences}.
\end{notation}

In this section we present a brief overview of input-output properads in the category of topological spaces. For a detailed treatment, in the context of \emph{usual} properads, see \cite{vallette2007koszul}. We shall define properads as algebras over a certain monad $\GG$. This point of view will be convenient when we discuss the cofibrant resolutions of properads below.

\subsection{Monads and algebras over them}
\begin{definition}\label{tripledef}
A \textbf{monad} (also called a \textbf{triple}) $T$ on a category $\catc$, is a functor $T \colon \catc \to \catc$ along with natural transformations $\mu \colon T\circ T \to T$ and $\eta \colon \mathbbm{1}_{\catc} \to T$ such that the following diagrams commute:

\[
\begin{tikzcd}
T \circ T \circ T \ar{r}{\mu \circ 1} \ar{d}{1 \circ \mu} &T \circ T \ar{d}{\mu} \\
T \circ T \ar{r}{\mu} & T\\
\end{tikzcd} \mbox{ \	 \	}
\begin{tikzcd}
\mathbbm{1}_{\catc} \circ T\ar{r}{\eta \circ 1} \ar[equals]{rd} &T \circ T \ar{d}{\mu} &T \circ \mathbbm{1}_{\catc} \ar{l}{1 \circ \eta} \ar[equals]{ld}\\
 &T & \\
\end{tikzcd}
\]
The natural transformations $\mu$ and $\eta$ are referred to as the composition and the unit of $T$, respectively. The first commutative diagram describes the associativity of $\mu$, whereas the second one describes the unitality of $\eta$. 

An \textbf{algebra over }$T$ is an object $X \in \catc$ equipped with a map $m \colon TX \to X \in \catc$ such that the following diagrams commute: 
\[
\begin{tikzcd}
T T X \ar{r}{\mu (X)} \ar{d}{Tm} &TX\ar{d}{m} \\
TX  \ar{r}{m} & T\\
\end{tikzcd} \mbox{ \	 \	}
\begin{tikzcd}
X \ar{r}{\eta (X)}\ar[equals]{rd} &T X \ar{d}{m}\\
 &X & \\
\end{tikzcd}
\]
\end{definition}
We denote the category of $T$-algebras in $\catc$ by \textbf{T-Alg}. For any $X \in \catc$, $TX$ is canonically a $T$-algebra. This defines the free $T$-algebra functor $T \colon \catc \to$ \textbf{T-Alg} which is left adjoint to the forgetful functor $Forget \colon $\textbf{T-Alg} $\to \catc$.


\subsection{Graphs}\label{sec: graphs}
To define the monad $\GG$, we need to fix some conventions regarding graphs. 

By a \emph{directed graph} $G$ we mean the data of
\begin{itemize}
\item A finite set of vertices $V(G)$
\item A finite set of directed edges $E(G)$, equipped with source and target incidence maps $s, t \colon E(G) \to V(G)$. 
\item A set of input legs $in(G)$ and output legs $out(G)$ (also called external edges/tails), along with incidence maps $t_l \colon in(G) \to V(G)$ and $s_l \colon out(G) \to V(G)$.
\item An ordering on the sets $in(G), out(G)$ and on $in(v), out(v)$ for every $v\in V(G)$, where $in(v)$ (respectively, $out(v)$) denotes the set of edges and legs directed into (respectively, out of)  $v$.
\end{itemize}

We say that $G$ is connected if the geometric realization of $G$ is connected. Here by the geometric realization of $G$ we mean the topological space obtained by considering a union of points and intervals indexed by vertices and edges/legs of $G$ respectively, with the endpoints of the intervals identified with the vertices according to the incidence maps.

By an \emph{isomorphism of directed graphs} $f \colon G \to H$, we mean a collection of bijections
\[ f_v \colon V(G) \to V(H), f_e \colon E(G) \to E(H),  f_{l_{in}} \colon in(G) \to in(H), \mbox{ and } f_{l_{out}} \colon out(G) \to out(H)\]
which preserve the incidence maps and the orderings.\\
By a \emph{directed cycle} in $G$ we mean a sequence of edges $e_1,...,e_n \in E(G)$ such that $t(e_i)= s(e_{i+1})$ for $1 \leq i \leq n-1$ and $s(e_1)=t(e_n)$.\\
We say that $G$ is \emph{input-output} if either $in(G)$ or $out(G)$ is non-empty i.e. if $G$ has at least one input leg or output leg.

\begin{definition}
By an \emph{input-output directed acyclic graph (\rda-graph)} we mean an input-output, connected directed graph which contains no directed cycles.
\end{definition}

\subsection{The monad $\GG$}\label{GGtripsubsec}
For an \rda-graph $G$ and $P \in \redTopSeq$, let 
\[ P(G) := \prod_{v \in V(G)} P(|in(v)|,|out(v)|) \]
denote the space of $P$-labelings of $G$. Denote by $\mathcal{G}_{n_-,n_+}$ the collection of graphs one from each isomorphism class of \rda-graphs having $n_-$ input and $n_+$ output leaves. Set 
\[ \GG_{n_-,n_+}(P) := \coprod_{G \in \mathcal{G}_{n_-,n_+}}P(G).\]
Denote by $\GG(P) \in \redTopSeq$ the bi-sequence given by $\GG(P) (n_-,n_+) = \GG_{n_-,n_+}(P)$. 

We now describe how to make $\GG$ into a monad. To do this we need to define operations $\mu \colon \GG \GG \to \GG$ and $\eta \colon \mathbbm{1} \to \GG$, satisfying the condition mentioned in Definition \ref{tripledef}.\\
For $P \in \redTopSeq$, $\GG\GG(P)$ is given by
\[ \GG\GG(P)(n_-,n_+) = \coprod_{G\in \mathcal{G}_{n_-,n_+}} \GG(P) (G) .\]
Thus an element in $\GG\GG(P)$ is represented by a $\GG(P)$-labeling of an \rda-graph. Such a decorated graph can be identified with a `nested' \rda-graph. Forgetting the nesting gives an \rda-graph with $P$-labelings. This defines the natural transformation $\mu \colon \GG\GG \to \GG$ (see \cite{markl2008operads} for details).\\
Let us now turn to the unit. Given $P \in \redTopSeq$, consider the maps $P(n_-,n_+) \to \GG_{n_-,n_+}(P)$ which identify $P(n_-,n_+)$ with $P$-labelings of the $(n_-,n_+)$-corolla, the unique graph in $\mathcal{G}_{n_-,n_+}$ having a single vertex and no internal edges. This defines the natural transformation $\eta \colon \mathbbm{1} \to \GG$. \\
It is not difficult to verify that $\mu$ and $\eta$ satisfy the associativity and unitality conditions in Definition \ref{tripledef}. We record here the conclusion for future reference:
\begin{lemma}\label{lem: GG is monad} $(\GG,\mu,\eta)$ is a monad. \\ \null \hfill \qedsymbol\end{lemma}

\subsection{\reduced Properads as monad algebras}
\begin{definition}
	An \emph{\reduced properad} in topological spaces is an algebra over the monad $(\GG, \mu, \eta)$ in the category $\redTopSeq$.
\end{definition}

We will denote the category of topological \reduced properads by $\Galg$.


\begin{figure} 
	\centering
	\begin{subfigure}[b]{0.3\textwidth}
		\includegraphics[width=\textwidth]{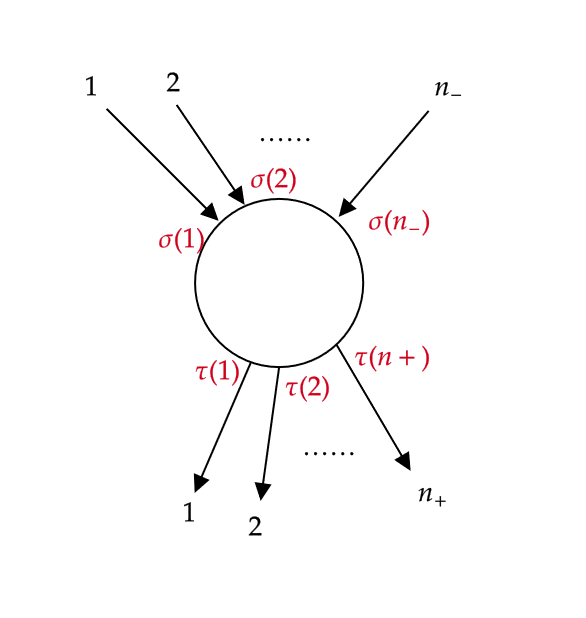}
		\caption{a permuted corolla}
		\label{perm_corr1}
	\end{subfigure}
	\begin{subfigure}[b]{0.35\textwidth}
		\includegraphics[width=\textwidth]{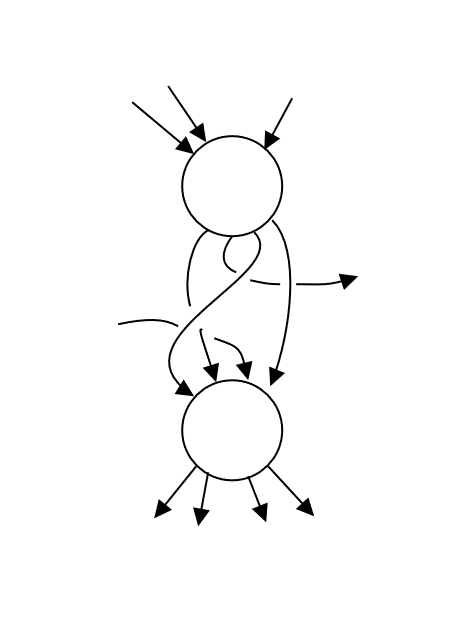}
		\caption{a partially grafted corolla}
		\label{perm_corr2}
	\end{subfigure}
	\caption{}
\end{figure}


\begin{remark}
	In particular if $P$ is an \reduced properad, for any \rda-graph $G$ we have a map
	\[ P(G) \to P(|in(G)|,|out(G)|)\]
	We will refer to it as the \emph{properad composition indexed by $G$} and denote it by $\mu_G$. In particular, we have the following structure maps:
\begin{enumerate}
\item \label{proprmk1} Using the graphs as in Figure \ref{perm_corr1}, we obtain a left $\Sigma_{n_-}$ and a right $\Sigma_{n_+}$ action on $P(n_-,n_+)$.
\item \label{proprmk2} Using the graphs as in Figure \ref{perm_corr2}, for any $n_-,n_+,m_-,m_+ \geq 0$ and \emph{non-empty} subsets $s_-$ and $s_+$ of $[n_+]$ and $[m_-]$ respectively, with an identification $\phi \colon s_- \to s_+$, we get a composition map
 \[ \circ_{\phi} \colon P(n_-,n_+) \times P(m_-,m_+) \to P(n_-+m_--|s_+|,n_++m_+-|s_-|)\]
provided $n_-,m_- - |s_+|, n_+- |s_-|$, and $m_+$ are not all simultaneously $0$. This slightly unnatural constraint is imposed by the fact that we work with \reduced properads which have no $(0,0)$ component.\\
These compositions are associative and are equivariant with respect to the $\Sigma \times \Sigma^{op}$-actions, described in point \eqref{proprmk1}, in a suitable sense.
\end{enumerate}
Conversely, it can be shown that $\Sigma \times \Sigma^{op}$-actions as in point \eqref{proprmk1} and (associative, $\Sigma \times \Sigma^{op}$-equivariant) composition maps as in point \eqref{proprmk2} define the structure of a $(\GG,\mu,\eta)$-algebra over $P$. We will not use this viewpoint and hence omit the details.
\end{remark}

\subsection{Pushouts of topological \reduced properads}
The main results in this paper are statements about certain homotopy pushouts in the category of topological \reduced properads. An explicit description of pushouts in the category of topological \reduced properads is used in Sections \ref{sec: hatthem pf pt1}, \ref{fulltcftsec}, and \ref{mbarsec} below to give a presentation of these homotopy pushouts. We now outline this description.

\begin{notation} 
To distinguish colimits in the category of topological \reduced-sequences from the colimits in topological \reduced properads, we will indicate the latter with a superscript $\Galg$. Below, whenever we use $coeq,\sqcup,$ and $\coprod$ without a superscript $\Galg$, it is understood that they refer to the respective operations in $\redTopSeq$.
\end{notation}

\begin{lemma}\label{pushoutlem}
	Let $P \leftarrow R \to Q$ be a diagram of topological \reduced properads. Then the pushout $P \sqcup_{R}^{\Galg} Q$ can be obtained as the following coequalizer in $\redTopSeq$
	\[ coeq\Big[ \GG(\GG P \coprod_{\GG R} \GG Q) \rightrightarrows \GG(P \coprod_{R} Q) \Big] \]
	with the \reduced properad structure induced from that of $\GG(P \sqcup_{R} Q)$. Here,
	\begin{enumerate}
	\item the first arrow is given by applying functor $\GG$ to the map $ \GG P \sqcup_{\GG R} \GG Q \to P \sqcup_{R} Q$ which itself is induced from the $\GG$-algebra structure maps $\GG P \to P, \GG Q \to Q$, and $\GG R \to R$
	\item the second arrow is induced via the universal property of free $\GG$-algebras using the map $ \GG P \sqcup_{\GG R} \GG Q \to \GG(P \sqcup_{R} Q) $ which in turn is the pushout of the maps given by applying functor $\GG$ to $P \to P \sqcup_R Q$, $Q \to P \sqcup_R Q$, and $R \to P \sqcup_R Q$
	\null \hfill \qedsymbol
	\end{enumerate}
\end{lemma}
The above lemma can be proved by directly verifying that the coequalizer with the given $\GG$-algebra structure satisfies the universal property of the pushout.
\begin{notation} 
	We say that a subgraph $H$ of an \rda-graph $G$ is \emph{collapsible} if the graph $G/H$, obtained by replacing $H$ with a single vertex, contains no directed cycles.	More precisely: consider the topological \reduced-sequence $\mathcal{G}= \{\mathcal{G}_{n_-,n_+}\}$ (with the discrete topology), where, as in Section \ref{GGtripsubsec}, $\mathcal{G}$ is the collection containing one representative from each isomorphism class of \rda-graphs . It has a natural $\GG$-algebra structure. We say that $H \subset G$ is collapsible if $\mu_K(\prod_{v \in V(K)} H_v ) = G$ for some \rda-graphs $K$ and $H_{v}, v \in V(K)$, such that $H_v$, considered as a subgraph of $G$, coincides with $H$ for some $v$.
\end{notation}

Unraveling the coequalizer in Lemma \ref{pushoutlem}, we get that the topological \reduced-sequence underlying $P \sqcup_{R}^{\Galg} Q$ can be described as the quotient
\[ \frac{\GG(P \sqcup Q)}{\sim} \]
where $\sim$ is the equivalence relation generated by the following identifications:
\begin{enumerate}
	\item Let $G$ be a $(P \sqcup Q)$-labeled graph which has a collapsible subgraph $H \subset G$ with all the vertices labeled by $P$ (respectively, $Q$). Then $G$ is identified with the graph obtained by collapsing $H$ to a single vertex labeled by the element of $P$(respectively, $Q$) obtained by applying properad composition (indexed by $H$) to the labels of $H$. 
	\item Suppose that the maps $R\to P$ and $R \to Q$ are denoted by $i$ and $j$ respectively. Let $G$ be a $(P \sqcup Q)$-labeled graph with a vertex $v \in G$ labeled by an element $i(r) \in i(R) \subset P$ for some $r \in R$. Then $G$ is identified with the graph obtained by switching the label of $v$ to $j(r) \in j(R) \subset Q$.
\end{enumerate}
This is analogous to the description of pushouts of operads presented in \cite[Section 2.3]{oancea2020deligne}.

\subsection{Properads as algebras over a colored operad}\label{subsubsec: proprd oprd}
We end this section by describing an alternate characterization of properads as algebras over a colored operad $\oprd$. We will use this alternate description in Section \ref{sec: model category str and cofib resolutions} below to construct a model structure  on topological \reduced properads and also in Section \ref{sec: infty properad} to discuss the notion of an $\infty$-properad.\\
In this subsection we assume basic familiarity with colored operads and algebras over them. We refer the reader to \cite[Section 1]{berger2007resolution} for an exposition to these notions.

\begin{definition}\label{definition: vertex order}
	By a \emph{vertex-ordered} \rda-graph we mean a graph as in Section \ref{sec: graphs}, with the additional data of an ordering on the set of vertices. Isomorphisms of such graphs are required to preserve this ordering. We refer to isomorphisms of underlying \rda-graphs as \emph{unordered-isomorphisms}.\\
\end{definition}

\subsubsection{Colored Operad $\oprd$}
Let us start by describing the set-valued colored operad $\oprd$ which encodes input-output properads. The colors of $\oprd$ are given by $\NN \times \NN \setminus (0,0)$. Let $\un{n}_-^1, \un{n}_-^2,...,\un{n}_-^k,$ and $\un{n}_+$ be elements of $\NN \times \NN \setminus (0,0)$. The operations corresponding to these colors,
\[ \oprd(\un{n}_-^1, \un{n}_-^2,...,\un{n}_-^k; \un{n}_+), \]
 are given by the set of isomorphism classes of vertex-ordered \rda-graphs with input-output profile $\un{n}$ and where the input-output profiles of the vertices, in the vertex-ordering, are given by the sequence $\un{n}_-^1, \un{n}_-^2,...,\un{n}_-^k$. The $\Sigma$-action on the operad is given by permuting the ordering of the vertices and the operad composition is given by graph substitution.
 
The following statement is not difficult to see:
\begin{lemma}\label{lem: proprd as oprd alg}
A $\oprd$-algebra in $\Top$ is precisely a topological \reduced properad.
\null \hfill \qedsymbol
\end{lemma}

We note here a fact which will be important later in Section \ref{sec: infty properad}:
\begin{lemma}\label{lem: sigma cofibrancy}
$\oprd$ is $\Sigma$-cofibrant, in other words the $\Sigma$-action on $\oprd$ is free.
\end{lemma}
\begin{proof}
This follows from the fact that vertex-ordered \rda-graphs have no non-trivial unordered automorphism (since \rda-graphs have no non-trivial automorphisms).
\end{proof}
\begin{remark}
This is false for a \emph{non-io} vertex-ordered graph. Figure \ref{non_red_graph_fig} shows an example.
\end{remark}


\begin{figure}
	\centering
		\includegraphics[width=\textwidth]{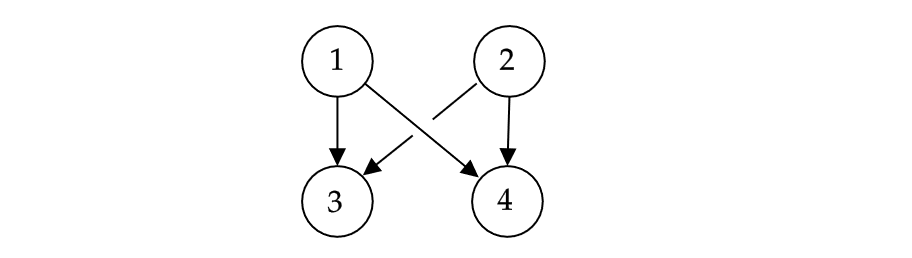}
		\caption{The \emph{non-io} vertex-ordered graph shown above has a non-trivial unordered-automorphism which swaps vertices 1, 2 and 3,4 (at each vertex the order of edges is assumed to be the one induced form the planar embedding)}
		\label{non_red_graph_fig}
	\end{figure}


\section{Model Category Structure and cofibrant resolutions}\label{sec: model category str and cofib resolutions}
\subsection{Model Category structure}
Model categories provide a natural setting for discussing homotopy theory in an abstract categorical setting. A model structure on a category consists of three distinguished classes of morphisms: fibrations, cofibrations, and weak equivalences, satisfying certain conditions. A model structure in particular allows us to define homotopy limits/colimits in a category. For more details on model categories and homotopy limits/colimits see \cite{may2011more}.

In order to make sense of the homotopy colimits in the category of topological \reduced properads which appear in Theorems \ref{frtofullthm} and \ref{hatthm}, we equip it with a model category structure.

\begin{notation}
Unless specified otherwise, whenever we refer to the model structure on $\Top$ we mean the standard model structure with the weak equivalences given by weak homotopy equivalences and fibrations given by Serre fibrations \cite[Section 17.2]{may2011more}. By the model structure on $\redTopSeq$ we will mean the induced model structure with weak equivalences and fibrations respectively given by component-wise weak equivalences and fibrations in $\Top$.
\end{notation}

Let $\oprdO$ be a colored operad in topological spaces with a set of colors $C$. Let $Alg_{\oprdO}(\Top)$ denote the category of algebras over $\oprdO$. Recall that any $\oprdO$-algebra $A$ has an underlying topological space $A(c)$, for every color $c \in C$.\\
Using \cite[Theorem 2.1]{berger2007resolution} for the monoidal model category $\Top$, we have the following:
\begin{proposition}
$Alg_{\oprdO}(\Top)$ admits a cofibrantly generated model structure such that a map $A \to B$ of $\oprdO$-algebras is a weak equivalence (respectively, fibration ) if and only if the underlying map of topological spaces $A(c) \to B(c)$ is a  weak homotopy equivalence (respectively, Serre fibration), for every $c \in C$.
\end{proposition}
\begin{remark}\label{rmk: cofib in model str} Using the general theory of model categories, it follows that the cofibrations in the above model structure are given by morphisms which have the left lifting property with respect to the class of acyclic fibrations (fibrations which are weak equivalences) of $\oprdO$-algebras.
\end{remark}

As a consequence, using Lemma \ref{lem: proprd as oprd alg}, we have:
\begin{corollary}\label{thm: model str stmt}
The category of topological \reduced properads has a cofibrantly generated model structure with weak equivalences (respectively, fibrations) given by the weak equivalences (respectively, fibrations) of the underlying topological \reduced-sequences.
\end{corollary}
\begin{remark}
As in Remark \ref{rmk: cofib in model str} above, cofibrations of \reduced properads are precisely the maps which have left lifting property with respect to acyclic fibrations of \reduced properads.
\end{remark}

\subsection{Cofibrant Resolutions via the Bar Construction}\label{subsec: Cofibrant Resolutions via The Bar Construction}
Homotopy colimits are defined in terms of the cofibrant resolutions of objects (and diagrams) in model categories. In the computations of homotopy colimits in Sections \ref{sec: hatthem pf pt1}, \ref{mbarsec}, and \ref{fulltcftsec}, we use an explicit description of the cofibrant resolutions to provide presentations of the homotopy colimits, which we now explain.

The cofibrant resolution we use is obtained by applying the bar construction to the monad $\GG$:
\subsubsection{Monadic Bar Construction} Let $(T,\mu,\eta)$ be a monad acting on a category $\catc$. The monadic bar construction of $T$ applied to a $T$-algebra is the analogue of the usual bar construction for algebras. Given a $T$-algebra $X$, its monadic bar construction $B_{\bullet}(T,T,X)$ is a simplicial $T$-algebra with the $n$-simplices given by 
\[ B_n(T,T,X) : = T^{n+1}(X)\]
The simplicial face maps are given by
\[ d_i = T^{i} \mu T^{n-i-1}X, 0 \leq i \leq n\]
and the degeneracy maps are given by
\[ s_i=T^{i+1}\eta T^{n-i}X, 0 \leq i \leq n.\]

The bar construction gives a simplicial resolution of the algebra $X$ in the following sense: Let $s\catc$ denote the category of simplicial objects in $\catc$ and let the constant simplicial object associated to $X$ be denoted by $X_{\bullet}$. Then,
\begin{proposition}[{\cite[Proposition 9.8]{may2006geometry}}]\label{prop: simplicial deformation retract}
$X_\bullet$ is a simplicial deformation retract of $B_{\bullet}(T,T,X)$ in $s\catc$.
\null \hfill \qedsymbol
\end{proposition}

Applying this to the monad $\GG$ and an \reduced properad $P$, we obtain a simplicial resolution of $P$. When $P$ is nice enough, it is possible to use this to obtain a cofibrant resolution of $P$. This is the cofibrant resolution we shall use below. The first step in this direction is to obtain an \reduced properad starting with the simplicial \reduced properad $B_{\bullet}(T,T,X)$. This is done by taking its \emph{geometric realization}, which we now outline.

\subsubsection{Geometric Realization}
The category of topological \reduced properads is tensored over topological spaces. Roughly this means that
\begin{itemize}
\item for any two \reduced properads $P_1$ and $P_2$ we can endow the set of \reduced properad morphisms $\Galg(P, Q)$ with a topology. Denote this topological space by $[P,Q]$, and
\item for any \reduced properad $P$ and a topological space $Z$, we can define their `tensor product' $P \odot Z$,
\end{itemize}
such that a version of hom-tensor adjunction holds:
\[ \Top(Z,[P, Q]) \simeq \Galg(P \odot Z, Q) .\]
See Appendix \ref{apndx: Cofibrant Resolutions via The Bar Construction} for more details on this.

Using this tensor product over topological spaces on the category of \reduced properads, we can define the geometric realization of a simplicial \reduced properad $\{P_\bullet\}_{\bullet \geq 0}$ by the usual formula:
\begin{equation}\label{geomrelform}
|P_\bullet|^{\Galg} = \int_{\Delta}^{\Galg} P_n \odot \Delta^n := coeq^{\Galg} \Big[ \coprod_{\phi \colon [n] \to [m] \in \Delta}^{\Galg} P_m \odot \Delta^n\rightrightarrows \coprod_{[n] \in \Delta}^{\Galg} P_n \odot \Delta^n \Big] .
\end{equation}
Here $\Delta$ denotes the simplex category, and $\Delta^n \in \Top$ is the standard $n$-simplex.	In the coequalizer, the first arrow is induced from the properad map given by the simplicial structure map $P(\phi) \colon P_m \to P_n$ corresponding to $\phi \colon [n] \to [m]$, and the second arrow is induced from the map $\Delta(\phi) \colon \Delta^n \to \Delta^m$ of topological spaces.

\subsubsection{Cofibrancy of the bar construction}
\begin{proposition}\label{prop: cofibrancy of bar construction}
Let $P$ be an \reduced properad which is cofibrant as a topological \reduced sequence. Then, the geometric realization $|B_{\bullet}(\GG,\GG,P)|^{\Galg}$ is a cofibrant \reduced properad.
\end{proposition}
The statement is proved by expressing the geometric realization as a sequence of pushouts in terms of the so-called \emph{latching spaces} and reformulating the cofibrancy condition in terms of these successive pushouts. We defer the details to Appendix \ref{apndx: Cofibrant Resolutions via The Bar Construction}.

We will also use a relative version of this statement:
\begin{proposition}\label{prop: rel cofibrancy of bar construction}
Let $P, Q$ be \reduced properads which are cofibrant as topological \reduced-sequences, and let $P \to Q$ be a map of \reduced properads such that the underlying map of topological \reduced sequences is a cofibration. Then, the induced map $|B_{\bullet}(\GG,\GG,P)|^{\Galg} \to |B_{\bullet}(\GG,\GG,Q)|^{\Galg}$ is a cofibration of \reduced properad.
\end{proposition}
The proof is provided in Appendix \ref{apndx: Cofibrant Resolutions via The Bar Construction}.

\subsubsection{Geometric realizations in topological \reduced sequences}
It is possible to define the geometric realization of simplicial topological \reduced sequences using the same formula \eqref{geomrelform}, with $\odot$ replaced by component-wise product with topological spaces and with the coproduct and the coequalizer taken in topological \reduced-sequences instead of \reduced properads. This coincides with taking the usual geometric realization of simplicial topological spaces component-wise.\\
Let $P_\bullet$ be a simplicial object in topological \reduced properads. Considering it as a  simplicial object in topological \reduced sequences, we can take its geometric realization $|P_\bullet|^{\redTopSeq}$ in the category of topological \reduced sequences.\\
It is well known that $|\ \ |^{\redTopSeq}$ is a monoidal functor (i.e. satisfies$|X_\bullet \times Y_\bullet| = |X_\bullet | \times |Y_\bullet |$), and hence it follows that $|P_\bullet|^{\redTopSeq}$ has a natural properad structure:\\
Given $G \in \mathcal{G}$, the composition along $G$, $\mu_G \colon G(|P_{\bullet}|) \to |P_{\bullet}|(in(G),out(G))$ is defined as the geometric realization of the corresponding map $\mu_G \colon G(P_{\bullet}) \to P_{\bullet}(in(G),out(G))$.\\

A priori it is not clear what the relation between the two geometric realizations is. However, we have the following (somewhat surprising) fact:
\begin{proposition}[{\cite[Theorem 7.5 (ii)]{mandell2019operads}}]\label{prop: geometric realizations iso}
Let $P_{\bullet}$ be a simplicial object in topological \reduced properads. Then, the topological \reduced properads obtained by taking the geometric realization of $P_{\bullet}$ in the category of topological \reduced properads and the category of topological \reduced sequences are isomorphic.
\end{proposition}
The proof, following \cite{mandell2019operads}, is given in Appendix \ref{apndx: Cofibrant Resolutions via The Bar Construction}.

\begin{notation}\label{notn: geom real}
In the sequel, we will use $|\ \ |$ to denote this common geometric realization. When it is necessary to emphasize the ambient category, we do so using superscripts $\redTopSeq$ and $\Galg$.
\end{notation}

\begin{corollary}\label{corr: cofibrancy of bar construction}
Let $P$ be a topological \reduced properad which is cofibrant as a topological \reduced-sequence, then $|B_{\bullet}(\GG,\GG,P)|^{\redTopSeq} \to P$ is a cofibrant resolution of $P$.
\end{corollary}
\begin{proof}
Using Proposition \ref{prop: simplicial deformation retract} and taking geometric realizations, we get that $|B_{\bullet}(\GG,\GG,P)|^{\redTopSeq} \to |P_{\bullet}|^{\redTopSeq} \simeq P$ is a homotopy equivalence. Also, using Proposition \ref{prop: cofibrancy of bar construction} and Proposition \ref{prop: geometric realizations iso}, it follows that $|B_{\bullet}(\GG,\GG,P)|^{\redTopSeq} \simeq |B_{\bullet}(\GG,\GG,P)|^{\Galg}$ is a cofibrant properad. This completes the proof that $|B_{\bullet}(\GG,\GG,P)|^{\redTopSeq} \to P$ is a cofibrant resolution of $P$.\\
\end{proof}

Similarly using Proposition \ref{prop: rel cofibrancy of bar construction} we have
\begin{corollary}\label{corr: rel cofibrancy of bar construction}
Let $P \to Q$ be a map of \reduced properad as in Proposition \ref{prop: rel cofibrancy of bar construction}, then $|B_{\bullet}(\GG,\GG,P)|^{\redTopSeq} \to |B_{\bullet}(\GG,\GG,Q)|^{\redTopSeq}$ is a cofibration of \reduced properads.\\
\null \hfill \qedsymbol
\end{corollary}

Using Corollaries \ref{corr: cofibrancy of bar construction} and \ref{corr: rel cofibrancy of bar construction}, we get the following 
\begin{lemma}\label{lem: bar pushout}
 Let $P \leftarrow R \to Q$ be a diagram of \reduced properads, such that 
\begin{itemize}
\item $P,Q, R$ are cofibrant as topological \reduced sequences, and
\item $R \to P$ and $R \to Q$ are cofibrations of topological \reduced sequences.
\end{itemize}
Then, the pushout of bar resolutions of $P,Q$, and $R$ 
\[ |B_{\bullet} P| \coprod_{|B_{\bullet} R|}|B_{\bullet} Q| \]
computes the homotopy pushout of $P \leftarrow R \to Q$.\\
\null \hfill \qedsymbol
\end{lemma}

In Section \ref{sec: hatthem pf pt1}, we will need to use this construction for a pushout diagram of properads where 
\begin{itemize}
\item the map $R \to Q$ is not a cofibration of topological \reduced sequences, and
\item the topological \reduced sequence underlying $P$ might not be cofibrant, but satisfies the property that $P(n_-,n_+)$ has the homotopy type of a CW-complex for every $(n_-,n_+)$.
\end{itemize}
The comparison statements in Proposition \ref{compareproposition} and \ref{compareproposition2} below, will be useful in this situation. Before stating the propositions we start with some terminology:\\
\begin{notation}\label{notn: hur weq cofib}
We will say that a map of topological \reduced-sequences is a \emph{\hur weak-equivalence (respectively, \hur cofibration)} if it is component-wise a homotopy equivalence (respectively, \hur cofibration) of topological spaces. We will say that a map of topological \reduced properads is a \hur weak-equivalence if the underlying map of topological \reduced-sequences is one.
\end{notation}
\begin{remark}\label{rmk: hur weq cofib}
\hur weak-equivalences and cofibrations of topological \reduced sequences correspond to weak-equivalences and cofibrations in a certain model structure on topological \reduced sequences, induced from the \hur (or Str\o{}m) model structure on topological spaces. (See Section \ref{subsec: app compareproposition} in Appendix \ref{apndx: Cofibrant Resolutions via The Bar Construction}).
\end{remark}

\begin{proposition}\label{compareproposition}
Let 
\begin{equation}\label{comparepropositiondig}
	\begin{tikzcd}
	P \ar{d} & R \ar{l}\ar{r}\ar{d} & Q \ar{d} \\
	P^{\prime}  & R^{\prime} \ar{l}\ar{r} & Q^{\prime} 
	\end{tikzcd}
\end{equation}
be a map of pushout diagrams of \reduced properads. If
	\begin{enumerate}
		\item each vertical arrow is a \hur weak-equivalence, and
		\item $R \to P$ , $R^{\prime} \to P^{\prime}$ are in fact \hur cofibrations of \reduced topological \reduced sequences
	\end{enumerate}
	Then 
	\[ |B_\bullet P| \coprod^{\Galg}_{|B_\bullet R|} |B_\bullet Q| \to |B_\bullet P^{\prime}| \coprod^{\Galg}_{|B_\bullet R^{\prime}|}|B_\bullet P^{\prime}|\]
	is a \hur weak-equivalence.
\end{proposition}
	Again we defer the proof of Proposition \ref{compareproposition} to Appendix \ref{apndx: Cofibrant Resolutions via The Bar Construction}.

\begin{proposition}\label{compareproposition2}
Let 
$P \leftarrow R \to Q$
be a pushout diagram of \reduced properads. If
\begin{enumerate}
\item $R \to P$ and $R \to Q$ are \hur cofibrations of underlying topological \reduced-sequences, and
\item for every $(n_-,n_+)$, $P(n_-,n_+), Q(n_-,n_+)$, and $R(n_-,n_+)$ have the homotopy type of a CW-complex,
\end{enumerate}
then
\[ |B_{\bullet} P| \coprod_{|B_{\bullet} Q|}|B_{\bullet} R| \]
computes the homotopy pushout of $P \leftarrow R \to Q$.
\end{proposition}
\begin{proof}
Recall that $\Top$ admits a functorial cofibrant replacement given by 
\[ | \Sing ( \_ ) | \colon \Top \to \Top,\]
where $\Sing(\_)$ is the singular simplicial set functor and $|\_|$ is the geometric realization functor. Applying it component-wise provides a functorial cofibrant replacement on $\redTopSeq$. Since $ | \Sing ( \_ ) |$ is symmetric monoidal on $\Top$, for any topological \reduced properad $X$, the topological \reduced-sequence $ | \Sing ( X) |$ admits a canonical properad structure such that the map
\[ | \Sing ( X) | \to X\]
is a morphism of \reduced properads. Applying this cofibrant replacement functor to pushout diagram $P \leftarrow R \to Q$, we obtain a diagram of \reduced properads
\begin{equation}\label{eq: comparepropositiondig2}
	\begin{tikzcd}
	{|\Sing(P)|} \ar{d} & {|\Sing(R)|} \ar{l}\ar{r}\ar{d} & {|\Sing(Q)|} \ar{d}\\
	P  & R\ar{l}\ar{r} & Q
	\end{tikzcd},
\end{equation}
	such that the vertical arrows are weak-equivalences. Using the fact that $P,Q$, and $R$ satisfy condition $(2)$ mentioned in the statement of the proposition, it follows that the vertical arrows are \hur weak-equivalences of properads. Moreover, note that maps $|\Sing(R)| \to |\Sing(P)|$ and $|\Sing(R)| \to |\Sing(Q)|$ are \hur cofibrations. Thus the diagram \eqref{eq: comparepropositiondig2} satisfies the hypothesis of Proposition \ref{compareproposition} and hence we have a weak-equivalence of \reduced properads
	\[ |B_\bullet {|\Sing(P)|}| \coprod^{\Galg}_{|B_\bullet {|\Sing(R)|}|} |B_\bullet {|\Sing(Q)|}| \to |B_\bullet P| \coprod^{\Galg}_{|B_\bullet R|}|B_\bullet P|\]
	The proposition now follows by observing that 
	\[ |B_\bullet {|\Sing(P)|}| \coprod^{\Galg}_{|B_\bullet {|\Sing(R)|}|} |B_\bullet {|\Sing(Q)|}|\]
	computes the homotopy pushout of $P \leftarrow R \to Q$.
\end{proof}

\subsubsection{Visualizing $B_{\bullet}(\GG,\GG,P)$}
Here we provide some examples of how simplices in the bar construction and their face and degeneracy maps are visualized in the figures which will appear in the later sections.

The $n$-simplices of $B_{\bullet}(\GG,\GG,P)$ are given by $\GG^{n+1}P$. We visualize points in this space as $n$-nested graphs, with each nesting indicated using a different color. For instance, elements in the space $\GG^2 P$ of $1$-simplices are denoted using $1$-nested \rda-graphs. Figure \ref{visualizing_d2_dig} shows an example of such a simplex. The $1$-nested graph there describes an element in $\GG^2P$ as follows: Considering each region in the outer nesting, indicated in blue, as a vertex, we obtain an \rda-graph. Moreover, each vertex of this \rda-graph has a labeling by a $P$-labeled graph given by the part of the nested graph lying inside the region.


\begin{figure} 
	\centering
	\begin{subfigure}[b]{0.3\textwidth}
		\includegraphics[width=\textwidth]{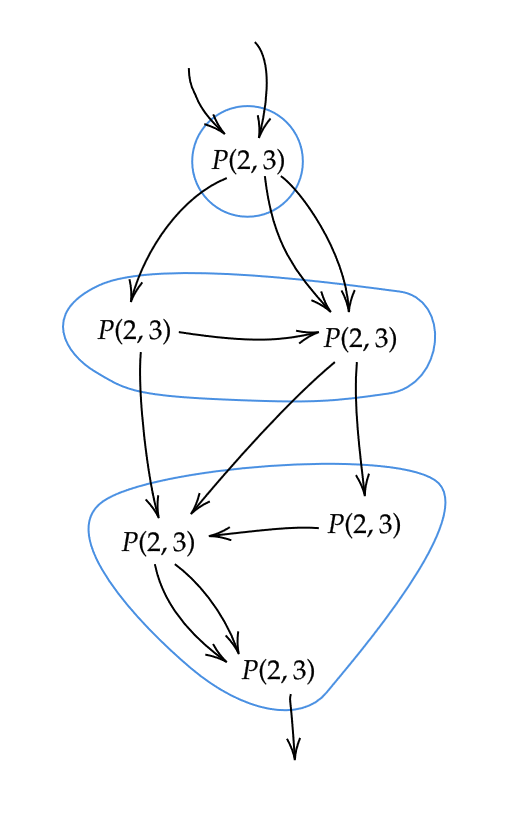}
		\caption{}
		\label{visualizing_d2_dig}
	\end{subfigure}
	\begin{subfigure}[b]{0.35\textwidth}
		\includegraphics[width=\textwidth]{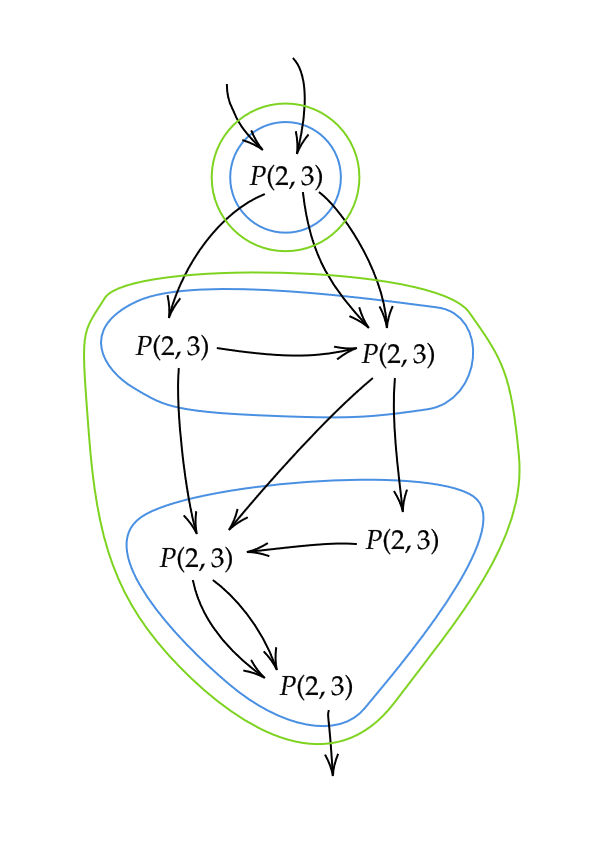}
		\caption{}
		\label{visualizing_base_dig}
	\end{subfigure}
	\caption{\protect\subref{visualizing_base_dig} shows a space of $2$-simplices. \protect\subref{visualizing_d1_dig} shows the result of applying $d_0$ to simplices in \protect\subref{visualizing_base_dig}}
\end{figure}


Similarly, the $2$-nested graph in Figure \ref{visualizing_base_dig} represents a $2$-simplex. Let us now describe the effect of applying the simplicial face maps to this simplex.\\
We start with the face map $d_0$. From its definition it follows that applying $d_0$ to this simplex corresponds to applying the properad composition to the $(\GG P)$-labels of the outermost nesting (along the graph described by the outermost nesting). This corresponds to simply forgetting the green nesting. The result is thus a $1$-simplex having a shape as indicated in Figure \ref{visualizing_d2_dig}. Similarly applying $d_1$ to a simplex in Figure \ref{visualizing_base_dig} corresponds to forgetting the blue nesting, thus giving a $1$-simplex as in Figure \ref{visualizing_d1_dig}.\\
On the other hand, applying $d_2$ corresponds to performing properad composition at the innermost level of the nesting. It is thus given by replacing each blue region by the properad compositions of the $P$-labels lying within it. The result of applying $d_2$ is thus a $1$-simplex as in Figure \ref{visualizing_d0_dig}. Note that unlike $d_0$ and $d_1$, $d_2$ makes use of the \reduced properad structure of $P$.

Lastly, we look at an example of a degeneracy map. Let us consider the case of applying degeneracy map $s_0$ to the simplices in Figure \ref{visualizing_base_dig}. This map is obtained by applying the unit $\eta$ of monad $\GG$ to each region in the blue nesting. It thus corresponds to adding an additional nesting around each blue region, resulting in a $3$-simplex having shape as shown in Figure \ref{visualizing_s1_dig}. This additional nesting is shown in red in Figure \ref{visualizing_s1_dig}.



\begin{figure} 
	\centering
	\begin{subfigure}[b]{0.3\textwidth}
		\includegraphics[width=\textwidth]{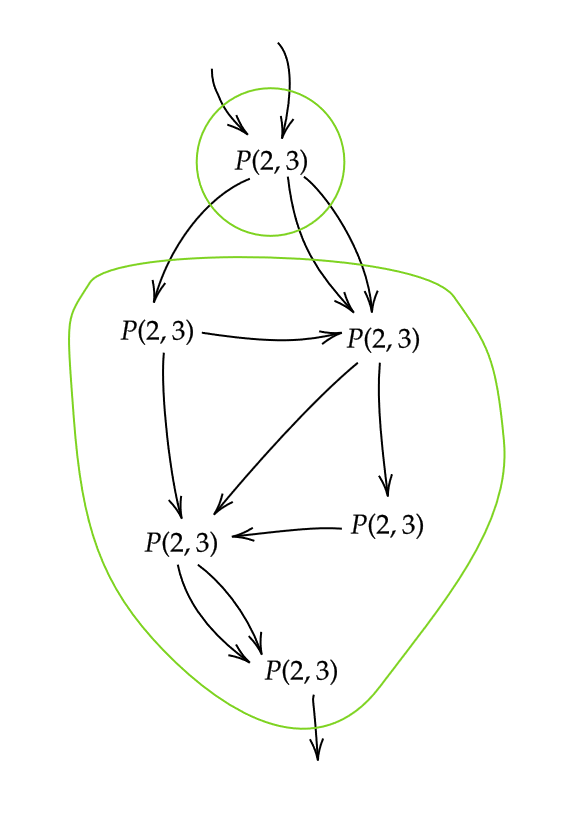}
		\caption{}
		\label{visualizing_d1_dig}
	\end{subfigure}
	\begin{subfigure}[b]{0.35\textwidth}
		\includegraphics[width=\textwidth]{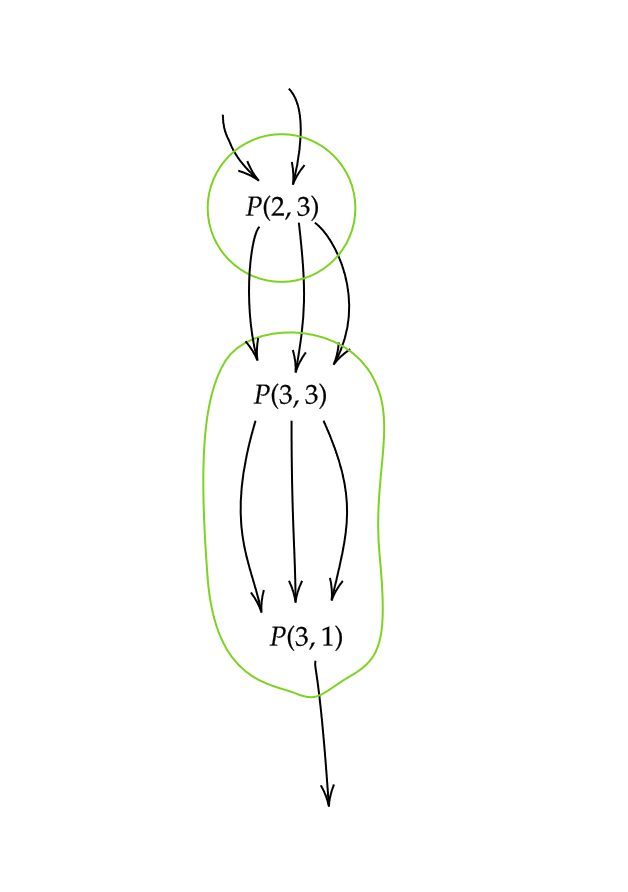}
		\caption{}
		\label{visualizing_d0_dig}
	\end{subfigure}
	\begin{subfigure}[b]{0.3\textwidth}
		\includegraphics[width=\textwidth]{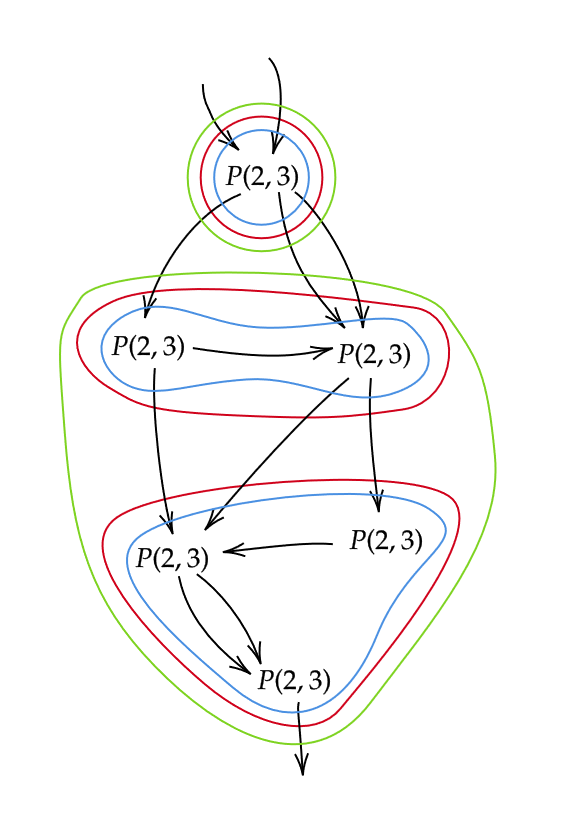}
		\caption{}
		\label{visualizing_s1_dig}
	\end{subfigure}
	\caption{\protect\subref{visualizing_d2_dig}, \protect\subref{visualizing_d0_dig}, \protect\subref{visualizing_s1_dig} show the result of applying, respectively, $d_2, d_0$, and $s_0$ to simplices in Figure \ref{visualizing_base_dig}}
\end{figure}


\section{Properads of Riemann surfaces}\label{sec: moduli spaces}
\subsection{Some moduli spaces of Riemann surfaces}\label{subsec: moduli spaces def}
For $g,n_-,n_+$ in the stable range $2-2g-n_--n_+<0$, let $\Mmod_{g,n_-,n_+}$ denote the moduli space of genus $g$ Riemann surfaces with $n_-$ input and $n_+$ output marked points. Let $\overline{\Mmod}_{g,n_-,n_+}$ be its Deligne-Mumford compactification, given by the moduli space of stable nodal curves with $n_-$ input and $n_+$ output marked points, and arithmetic genus $g$.

Let $\Mmod^{fr}_{g,n_-,n_+}$ denote the moduli space of genus $g$ Riemann surfaces with $n_-$ input and $n_+$ output boundary components, each equipped with an analytic $S^1$-parametrization. The parametrization is assumed to be orientation preserving at the inputs and orientation reversing at the outputs. Let $\Mmodbar^{fr}_{g,n_-,n_+}$ be the moduli space of stable nodal Riemann surfaces with $n_-$ input and $n_+$ output analytically $S^1$-parametrized boundaries. We assume that the nodes are away from the boundaries. Also, let $\Mmodhat^{fr}_{g,n_-,n_+}$ be the subspace of $\Mmodbar^{fr}_{g,n_-,n_+}$ given by the moduli spaces of stable curves in which every irreducible component contains at least one output.

In addition, we also consider the moduli spaces $\Mmod^{fr}_{g,n_-,n_+},\Mmodbar^{fr}_{g,n_-,n_+}, \Mmodhat^{fr}_{g,n_-,n_+}$ for the unstable range 
\[(g,n_-,n_+) \in \{(0,0,1), (0,1,0), (0,1,1), (0,0,2), (0,2,0)\}.\]
For $(g,n_-,n_+)$ given by $\{(0,1,1), (0,0,2)$, and $(0,1,0)$, $\Mmod^{fr}_{g,n_-,n_+}$ are all identified with the space of annuli with $S^1$-parametrized boundary components. The spaces $\Mmodbar^{fr}_{g,n_-,n_+}$ and $\Mmodhat^{fr}_{0,0,2}$ are obtained by compactifying these spaces  by also including nodal annuli, thought of as annuli with modulus $\infty$ (with boundaries suitably labeled as inputs and outputs). Consistent with the restriction that each irreducible component has an output, the space $\Mmodhat_{0,1,1}$ coincides with $\Mmod_{0,1,1}$ and $\Mmodhat^{fr}_{0,2,0}$ is empty.\\
For $(g,n_-,n_+)$ equal to $(0,0,1)$ or $(0,1,0)$, all the corresponding moduli spaces are identified with moduli spaces of disks with a parametrized boundary and thus are given by a point. 
Moreover, in the case $(g,n_-,n_+)=(0,1,1)$ for all the three types of moduli mentioned above, we also allow exceptional points corresponding to degenerate annuli with modulus $0$. These points are added to ensure that the \reduced properads we consider below are unital.

Finally, denote by $\Mmod^{fr}_{g,n_-,n_+}(i)$ and $\Mmodbar^{fr}_{g,n_-,n_+}(i)$ the moduli space of, respectively, non-nodal and possibly nodal stable Riemann surfaces with arithmetic genus $g$, $n_-$ input boundaries, $n_+$ output boundaries, and with $i$ marked points disjoint from the nodes and the boundaries. For every $1 \leq j \leq i$ we have forgetful maps
\[ \pi_j\colon \Mmod^{fr}_{g,n_-,n_+}(i) \to \Mmod^{fr}_{g,n_-,n_+}(i-1) \mbox{ and } \pi_j \colon \Mmodbar^{fr}_{g,n_-,n_+}(i) \to \Mmodbar^{fr}_{g,n_-,n_+}(i-1),\]
given by forgetting the $j$th marked point and stabilizing (by collapsing any unstable components created) if necessary.


All the moduli spaces mentioned above are a priori topological stacks. However the following lemma shows that most of these are in fact topological spaces. This observation already occurs in \cite[ Section 3.1]{oancea2020deligne}:
\begin{lemma}
The stacks $\Mmod^{fr}_{g,n_-,n_+}$ and $\Mmodhat^{fr}_{g,n_-,n_+}$ are represented by topological spaces provided $(n_-,n_+) \neq (0,0)$. Moreover, the moduli spaces $\Mmodbar^{fr}_{g,n_-,n_+}$ for $(g,n_-,n_+)$ in the unstable range $2g-2+n_++n_-<0$, are also represented by topological spaces.
\end{lemma}
\begin{proof}
Notice that a Riemann surface with at least one parametrized boundary component has no non-trivial automorphisms: any automorphism of a Riemann surface with parametrized boundary preserves the boundary parametrization and hence, by analytic continuation, is forced to be the identity. Thus it follows that $\Mmod^{fr}_{g,n_-,n_+}$ for $(n_-,n_+) \neq (0,0)$ is in fact a topological space. Similar arguments can be used for nodal Riemann surfaces to show that $\Mmodhat^{fr}_{g,n_-,n_+}$ for  $(n_-,n_+) \neq (0,0)$ and the moduli spaces $\Mmodbar^{fr}_{g,n_-,n_+}$  in unstable range $2g-2+n_-+n_+$ are also given by topological spaces.
\end{proof}

On the other hand $\Mmodbar^{fr}_{g,n_-,n_+}$ are in general not spaces but only topological stacks. The stacky points are given by surfaces containing irreducible components which have no boundaries and have a non-trivial conformal automorphism group.

\subsection{Finite dimensional homotopy types}\label{subsec: fin dim htpy types} Even though we will not need this fact in what follows, let us note that although the spaces $\Mmod^{fr}_{g,n_-,n_+}$, $ \Mmodhat^{fr}_{g,n_-,n_+}$, $\Mmodbar_{g,n_-,n_+}$ are infinite dimensional, they have finite dimensional homotopy types:\\
Let $\smallmfrfull_{g,n_-,n_+}$ be the moduli space of Riemann surfaces with $n_-$ input and $n_+$ output boundary components with each boundary component carrying a marked point (and no analytic $S^1$-parametrization). It can be shown that $\smallmfrfull_{g,n_-,n_+}$ is represented by a topological space and is in fact a  finite dimensional smooth manifold. Moreover, $\smallmfrfull_{g,n_-,n_+}$ has the same homotopy type as $\Mmod^{fr}_{g,n_-,n_+}$: there is a homotopy equivalence $\Mmod^{fr}_{g,n_-,n_+} \xrightarrow{\simeq} \smallmfrfull_{g,n_-,n_+}$ given by mapping a point in $\Mmod^{fr}_{g,n_-,n_+}$ to the underlying Riemann surface with boundary marked points given by base points of the analytic $S^1$-parametrizations of the boundaries.\\
Similarly, spaces $\Mmodhat^{fr}_{g,n_-,n_+}$ and $\Mmodbar_{g,n_-,n_+}$ have homotopy types of finite dimensional spaces $\smallmhat_{g,n_-,n_+}$ and $\smallmbar_{g,n_-,n_+}^{fr}$ given by moduli spaces of stable Riemann surfaces as in $\Mmodhat^{fr}_{g,n_-,n_+}$ and $\Mmodbar_{g,n_-,n_+}$, but without the boundary parametrizations and with each boundary component carrying a marked point (see \cite{liu2020moduli} for details on topology of the spaces $\smallmhat_{g,n_-,n_+}$ and $\smallmbar_{g,n_-,n_+}$). The homotopy equivalences in these cases are realized by maps constructed similarly as before.

In the case of $\Mmod^{fr}_{g,n_-,n_+}$ and $\Mmodbar_{g,n_-,n_+}$ there are alternative descriptions of homotopy types that are more closely related to their appearance in the theory of closed string invariants of symplectic manifolds, namely symplectic cohomology and Gromov-Witten theory:\\
Let $\mlfrfullmod_{g,n_-,n_+}$ denote the moduli space of Riemann surfaces (without boundaries), with $n_-$ input and $n_+$ output marked points and with each marked point carrying a marker i.e. a ray in the tangent space of the marked point. This is a torus bundle over the usual moduli space of Riemann surfaces with marked points.  It can be shown that the stack $\mlfrfullmod_{g,n_-,n_+}$ is represented by a  finite dimensional smooth manifold. Further, there is a homotopy equivalence $\Mmod^{fr}_{g,n_-,n_+} \xrightarrow{\simeq} \mlfrfullmod_{g,n_-,n_+}$ given by mapping a surface in $\Mmod^{fr}_{g,n_-,n_+}$ to the Riemann surface obtained by gluing-in unit disks at the boundary components using the boundary parametrizations (see Section \ref{subsec: gluing riemann surf} below), with the marked points at the centers of the disks and the markers determined by the positive real direction in each disc.\\
Similarly, it can be shown that $\Mmodbar_{g,n_-,n_+}^{fr}$ has the homotopy type of the stack $\mlbar_{g,n_-,n_+}$ which is the finite dimensional moduli stack of stable Riemann surfaces with  $n_-$ input and $n_+$ output marked points (with no markers), in other words the usual Deligne-Mumford compactification of the moduli space of Riemann surfaces (without boundaries). The map $\Mmodbar_{g,n_-,n_+}^{fr} \to \mlbar_{g,n_-,n_+}^{fr}$ is constructed as above by gluing unit disks along boundaries followed by collapsing any unstable spheres generated as a result of this gluing.\\

\subsection{Gluing Riemann surfaces}\label{subsec: gluing riemann surf}
Let $\Sigma_1$ and $\Sigma_2$ be two Riemann surfaces with parametrized boundaries, and let $\gamma_1$ be an output boundary of $\Sigma_1$ and $\gamma_2$ an input boundary of $\Sigma_2$. Then there exists a unique complex structure on gluing $\Sigma_1 \#_{\gamma_1,\gamma_2} \Sigma_2$. For an identification of the boundaries along a diffeomorphism this is a classical fact. For an analytic identification, as in our case, it is possible to provide a simpler argument. We give an outline, the details are described in \cite[Section 3.1]{oancea2020deligne}: the analytic $S^1$-parametrization of $\gamma_1$ can be extended to an analytic identification of a neighborhood of the boundary with an annulus in $\CC$ of the form $\{1 \leq |z| \leq 1+\epsilon_1\}$. Similarly a neighborhood of $\gamma_2$ can be identified with $\{ 1-\epsilon_2 \leq |z| \leq 1\}$. The germs of these identifications are uniquely determined. The gluing of these annuli is canonically identified with the annulus $\{1-\epsilon_2 \leq |z| \leq 1+\epsilon\}$ and thus gets a unique complex structure. This can now be used to obtain a complex structure on the gluing $\Sigma_1 \#_{\gamma_1,\gamma_2} \Sigma_2$.

It is straightforward to generalize this to gluing along multiple boundaries and also to the gluing of nodal Riemann surfaces.

\subsection{\reduced Properads of Riemann surfaces}
We now describe the \reduced properads which will appear in our discussion:
	\begin{itemize}
		\item $\Mfrfull$ is the \reduced properad defined by $\Mfrfull (n_-,n_+) = \coprod_{g \geq 0} \Mmod^{fr}_{g,n_-,n_+}$. The properad compositions are given by gluing the Riemann surfaces along suitable boundaries via the $S^1$-parametrization. We refer to $\Mfrfull$ as the \emph{{\CFT}-properad}.
		\item $\Mbar$ is the \reduced  properad defined by $\Mbar (n_-,n_+) = \coprod_{g \geq 0} \Mmodbar^{fr}_{g,n_-,n_+}$. The properad compositions are again given by gluing the nodal Riemann surfaces along suitable boundaries via the $S^1$-parametrization, followed by collapsing any unstable components created as a result of the gluing. We refer to this as the \emph{Deligne-Mumford properad}.
	\item $\Mfr$ is   the \reduced subproperad of $\Mfrfull$ consisting of all operations which have at least one output i.e. $\Mfr (n_-,n_+) = \coprod_{g \geq 0} \Mmod^{\partial_+}_{g,n_-,n_+}$, where
\[\Mmod^{\partial_+}_{g,n_-,n_+} = \begin{cases} \varnothing &\mbox{ if } n_+=0, \\ \Mmod^{fr}_{g,n_-,n_+} &\mbox{ otherwise }.\end{cases} \]
We refer to $\Mfr$ as the \emph{{\delCFT}-properad}. 
\item $\Mhat$ is the \reduced subproperad of $\Mbar$ defined by $\Mhat (n_-,n_+) = \coprod_{g \geq 0}\Mmodhat^{fr}_{g,n_-,n_+}$. Note that similarly to $\Mfr$, the components with $n_+=0$ are empty. We refer to $\Mhat$ as the \emph{\symplectic properad}
	\end{itemize}
	
	Let $\Mfrfullunst$, $\Mbarunst$, $\Mhatunst$, and $\Mfrunst$ be the \reduced subproperads of $\Mfrfull$,  $\Mbar$, $\Mhat$, and $\Mfr$ respectively, consisting of (possibly nodal) Riemann surfaces in the unstable range. More precisely, these properads are defined by:

	\begin{itemize}
	\item $\Mfrfullunst$ is the \reduced subproperad of $\Mfrfull$ which is empty in all components except the following:
	\begin{align*}
		&\Mfrfullunst(0,1) = \Mmod^{fr}_{0,0,1}, \quad \Mfrfullunst(1,0) = \Mmod^{fr}_{0,1,0}, \quad \Mfrfullunst(1,1) = \Mmod^{fr}_{0,1,1}, \\
		&\Mfrfullunst(0,2) = \Mmod^{fr}_{0,0,2}, \quad \Mfrfullunst(2,0) = \Mmod^{fr}_{0,2,0}.  
	\end{align*}
	\item $\Mbarunst$ is the \reduced subproperad of  $\Mbar$ which is empty in all components except the following:
	\begin{align*}
		&\Mbarunst(0,1) = \Mmod^{fr}_{0,0,1}, \quad \Mbarunst(1,0) = \Mmod^{fr}_{0,1,0}, \quad \Mbarunst(1,1) = \Mmod^{fr}_{0,1,1}, \\
		&\Mbarunst(0,2) = \Mmod^{fr}_{0,0,2}, \quad \Mbarunst(2,0) = \Mmod^{fr}_{0,2,0}.  
	\end{align*}
	The properad operations are defined as in $\Mbar$, i.e. by gluing along boundaries followed by stabilization. 
	\item $\Mhatunst$ is the \reduced subproperad of $\Mhat$ which is empty in all components except the following:
\[
	\Mhatunst(1,1)=\Mmodhat^{fr}_{0,1,1},	\quad \Mhatunst(1,1)=\Mmodhat^{fr}_{0,1,1} , 	 \quad  \Mhatunst(0,2)=\Mmodhat^{fr}_{0,0,2}. 
\]
	\item Finally $\Mfrunst$ is the \reduced subproperad of $\Mfr$  which is empty in all components except the following
\[
	  \Mfrunst(0,1)=\Mmod^{fr}_{0,0,1},	\quad \Mfrunst(1,1)=\Mmod^{fr}_{0,1,1} , 	\quad   \Mfrunst(0,2)=\Mmod^{fr}_{0,0,2}. 
\]
Clearly $\Mfrunst$ is a subproperad of $\Mhatunst,\Mfrfullunst,$ and $\Mbarunst$.
	\end{itemize}
	Lastly, $\Mfrfullnop$ is the io-subproperad of $\Mfrfull$ which is empty  all components except the following:
\[
	  \Mfrfullnop(0,1)=\Mmod^{fr}_{0,0,1},	\quad \Mfrfullnop(1,0)=\Mmod^{fr}_{0,1,0},	\quad \Mfrfullnop(1,1)=\Mmod^{fr}_{0,1,1} , 	\quad   \Mfrfullnop(0,2)=\Mmod^{fr}_{0,0,2}. 
\]
	In Section \ref{subsec: mthm1htppushsec} below, we will also use the following modifications of \reduced properads $\Mfr, \Mfrunst,$ and $\Mhatunst$ which we record here for convenience:
	\begin{itemize}
\item $\mmfrfull$ is the \reduced  subproperad of $\Mfr$ which coincides with $\Mfr$ in all degrees except $(1,1), (0,2)$. In these degrees the genus $0$ components of $\mmfrfull(1,1),\mmfrfull(0,2)$ are the subspaces of $\Mmod^{fr}(1,1), \Mmod^{fr}(0,2)$, respectively, containing only the exceptional annuli which have modulus $0$. 
\item $\Mfrunstth$ is the \reduced subproperad of $\Mfrunst$ defined in an analogous manner.
\item $\Mhatunstth$ is an \reduced subproperad of $\Mhatunst$. In degree $(1,1)$, similarly to the cases above, it only contains the exceptional annuli of modulus $0$. However in this case, we define $\Mhatunstth(0,2)$ to coincide with $\Mhatunst(0,2)$.
\end{itemize}
	
\section{From the \delCFT-properad to the \symplectic properad}\label{sec: hatthem pf pt1}

In this section we present the proof of Theorem \ref{hatthm}(1) . Let us start by recalling the statement:
\begin{theorem}
	$\Mhat$ is the homotopy colimit of the following diagram in the category of \reduced properads
	\begin{equation}\label{repeat: thm1dgm}
		\begin{tikzcd}
		\Mfrunst \ar{r}\ar{d} & \Mhatunst\\
		\Mfr & 
		\end{tikzcd}
	\end{equation}
\end{theorem}

\subsection{The homotopy pushout}\label{subsec: mthm1htppushsec}
Let $\mmfrfull, \Mfrunstth$ and $\Mhatunstth$ be as described at the end of Section \ref{sec: moduli spaces}. Note that
\begin{itemize}
\item $\mmfrfull(n_-,n_+), \Mfrunstth(n_-,n_+)$ and $\Mhatunstth(n_-,n_+)$ have the homotopy type of CW-complexes (\cite[Corollary 1]{milnor1959spaces}), and
\item $\Mfrunstth \to \mmfrfull$ and $\Mfrunstth \to \Mhatunstth$ are \hur cofibrations of underlying topological \reduced sequences (see Notation \ref{notn: hur weq cofib}).
\end{itemize}
Thus using Proposition \ref{compareproposition2}, it follows that
\begin{equation}\label{eq: bar pushout 1}
|B_\bullet \mmfrfull| \bigsqcup^{\Galg}_{|B_\bullet \Mfrunstth|} |B_\bullet \Mhatunstth| 
\end{equation}
computes the homotopy pushout of $\mmfrfull \leftarrow \Mfrunstth \to \Mhatunstth$.  The analogous conclusion for the pushout 
\begin{equation} \label{eq: bar pushout 2} |B_\bullet \Mfr| \bigsqcup^{\Galg}_{|B_\bullet \Mfrunst|} |B_\bullet \Mhatunst| \end{equation}
is not clear from the argument above, since the map $\Mfrunst \to \Mhatunst$ is not a cofibration of topological \reduced sequences. \\
Now consider the inclusion of pushout diagrams:
\begin{equation}\label{inclusionofdiagrams} 
	\begin{tikzcd}
	\mmfrfull \ar{d} & \Mfrunstth\ar{l}\ar{d}\ar{r} & \Mhatunstth \ar{d} \\ 
	\Mfr & \Mfrunst \ar{l}\ar{r}& \Mhatunst
	\end{tikzcd}
\end{equation}
Note that 
\begin{itemize}
\item all the vertical maps are \hur weak-equivalences of \reduced properads (see Notation \ref{notn: hur weq cofib}), and
\item the maps of topological \reduced sequences underlying $\Mfrunstth \to \mmfrfull$ and $\Mfrunst \to \Mhatunst$ are \hur cofibrations.
\end{itemize}
Therefore using Proposition \ref{compareproposition}, we have:
\begin{lemma}
The pushouts \eqref{eq: bar pushout 1} and \eqref{eq: bar pushout 2} are weakly homotopy equivalent.
	  In particular, pushout \eqref{eq: bar pushout 2} viz.
	   \[|B_\bullet \Mfr| \bigsqcup^{\Galg}_{|B_\bullet \Mfrunst|} |B_\bullet \Mhatunst|\] 
	   computes the homotopy colimit of \eqref{repeat: thm1dgm}\\
\null \hfill \qedsymbol	   
\end{lemma}

\subsection{Pushout \eqref{eq: bar pushout 2} is weak homotopy equivalent to $\Mhat$}\label{finalhtpyeq}
From the computation \eqref{barpushout}, we have:
\[|B_\bullet \Mfr| \coprod_{|B_\bullet\Mfrunst|}^{\Galg}  |B_\bullet \Mhatunst|  \simeq  \Big|\{ \GG \big(\GG^{\bullet} \Mfr \coprod_{\GG^{\bullet}\Mfrunst} \GG^{\bullet}\Mhatunst \big) \Big|\]

We prove Theorem \ref{hatthm}(1) using the following strategy: We will show that the map
\begin{equation}\label{eq: pi defined}
\pi \colon \Big|\{ \GG \big(\GG^{\bullet} \Mfr \coprod_{\GG^{\bullet}\Mfrunst} \GG^{\bullet}\Mhatunst \big) \Big| \to \Mhat 
\end{equation}
satisfies the property that every $\Sigma \in \Mhat$ has a neighborhood $U_{\Sigma}$ such that for any finite collection $\Sigma_1,...,\Sigma_k$, 
\begin{equation}
	\label{finintwheq} \pi^{-1}(U_{\Sigma_1,...,\Sigma_k}) \to U_{\Sigma_1,...,\Sigma_k} \mbox{ is a weak homotopy equivalence}
\end{equation}
where $U_{\Sigma_1,...,\Sigma_k} = U_{\Sigma_1} \cap ... \cap U_{\Sigma_k}$.

\begin{remark} Note that here we are using a local criterion for weak equivalences (see for example, \cite[Corollary 1.4]{may1990weak}). We have to resort to this approach since $\pi$ is not a fibration in general and hence just verifying the contractibility of fibers may not be sufficient to imply weak equivalence. However to understand the contracting homotopies constructed below, particularly those in Sections \ref{psihtpy} and \ref{mthm1finalparag}, it might be helpful to think of them as ways of extending the contracting homotopy of a given fiber to a neighborhood consisting of nearby fibers.
\end{remark}

For every $\Sigma \in \Mhat$ we shall in fact construct a neighborhood $U_{\Sigma}$ satisfying the following properties: $\pi^{-1}(U_{\Sigma})$ has a filtration
\[ W_{1,\Sigma} \subset ... \subset W_{n, \Sigma} ... \subset \pi^{-1}(U_{\Sigma})\]
such that for each $n$, 
\[\pi|_{W_{n, \Sigma}} \colon W_{n, \Sigma} \to U_{\Sigma}\]
is a weak homotopy equivalence. We prove this by showing that for each $n$,  $W_{n, \Sigma}$ has a further filtration

\[ F_{n, \Sigma}^{(1)} \subset ... \subset F_{n, \Sigma}^{(m)}... \subset W_{n, \Sigma}\]
such that 
\begin{equation}\label{filtrantncondition}
	\begin{split}
		& F_{n, \Sigma}^{(m)}  \mbox{ admits a fiberwise contracting homotopy \emph{inside} } W_{n, \Sigma},\\[-4pt]
		&\mbox{ onto a section } \chi_{n, \Sigma}^{(m)} \colon U_{\Sigma} \to W_{n, \Sigma} \mbox{ of } \pi|_{W_{n, \Sigma}}
	\end{split}
\end{equation}


Note that this implies \eqref{finintwheq} for any collection of elements of the cover.




\subsubsection{Construction of $U_{\Sigma}, W_{n, \Sigma},$ and $F_{n, \Sigma}^{(m)}$}\label{gammasubsubsect}
Let us start by describing the open set $U_{\Sigma}$. We use a local description of the moduli spaces as outlined in \cite{siebert1996gromov}. Let $\widetilde{\Sigma} = \coprod \Theta_i$ be the normalization of $\Sigma$, where $\Theta_i$ are smooth connected Riemann surfaces. For every node $\nu$ of $\Sigma$, denote by $\hat{\nu}, \check{\nu}$ its pre-images in $\widetilde{\Sigma}$. Let $\Nu$ denote the collection of nodes of $\Sigma$. Let $\widetilde{\Nu}$ be the collection of the points $\hat{\nu}, \check{\nu}$ as $\nu$ varies over $\Nu$. Let $\widetilde{\Nu}_i$ denote the subcollection of such points on $\Theta_i$. Now, let $U_{\Theta_i,\widetilde{\Nu}_i}$ be a small neighborhood of $\Theta_{i,\widetilde{\Nu}_i}$ in the moduli space of Riemann surfaces with the same topological type as $\Theta_i$ and with internal marked points indexed by $\widetilde{\Nu}_i$. Let $\cfam_{\Theta_i, \widetilde{\Nu}_i} \to U_{\Theta_i,\widetilde{\Nu}_i}$ be the universal curve over it, and let $s_\zeta$ be the sections corresponding to the marked points $\zeta \in \widetilde{\Nu}_i$.

Consider the product family 
\[ \cfam_{\widetilde{\Sigma}, \widetilde{\Nu}} = \prod \cfam_{\Theta_i,\widetilde{\Nu}_i} \to U_{\widetilde{\Sigma}, \un{\widehat{\nu}}} =  \prod U_{\Theta_i, \widetilde{\Nu}_i}\]
along with the sections $s_{\zeta}, \zeta \in \widetilde{\Nu}$. Now let $z_\zeta$ be a holomorphic function in a neighborhood of Im$(s_{\zeta})$ which restricts to a holomorphic co-ordinate on each fiber. Shrinking $U_{\Theta_i,\widetilde{\Nu}_i}$ and rescaling the $z_{\zeta}$ if necessary, we may assume that the open sets $R_{\zeta} = \{ |z_{\zeta}| <1\}$ contain no marked points other than $\zeta$ and are isomorphic to $U_{\Theta_i,{\widetilde{\Nu}}_i} \times \DD$ over $U_{\Theta_i,{\widetilde{\Nu}}_i}$. Let $t_{{\nu}} \colon \DD^{\Nu} \to \DD$ be the projection onto the $\nu$th component of $\DD^{\Nu}$. Now, define $U_{\Sigma}$ to be the  neighborhood of $\Sigma$ in $\Mhat$ given by the product
\[ U_{\Sigma} := U_{\widetilde{\Sigma}} \times \DD^{\Nu}. \]

The restriction of the universal curve over $U_{\Sigma}$
\[ \cfam_{\Sigma} \to U_{\Sigma} = U_{\widetilde{\Sigma}} \times \DD^{\Nu} \]
is obtained as follows: for every node $\nu$ in $\Nu$ remove the closed subsets 
\[ \{(z,\un{t}) \in R_{\hat{\nu}} \times \DD^{\Nu} | |z_{\hat{\nu}}(z)| \leq |t_{\nu}| \} \mbox{ and } \{(z,\un{t}) \in R_{\check{\nu}} \times \DD^{\Nu} | |z_{\check{\nu}}(z)| \leq |t_{\nu}| \}\]
from $\cfam_{\widetilde{\Sigma}, \widetilde{\Nu}} \times \DD^{\Nu}$ and identify the rest of $R_{\hat{\nu}}$ and $R_{\check{\nu}}$ via 
\[ z_{\hat{\nu}} \cdot z_{\check{\nu}} = t_{\nu}.\]

Let us now turn to the construction of the filtrations $W_{n, \Sigma}$ and $F_{n, \Sigma}^{(m)}$ satisfying condition \eqref{filtrantncondition}.

Let $\un{x}$ be an element in $\GG \Mfr \coprod_{\GG \Mfrunst} \GG \Mhatunst$. It is an $\left( \Mfr \coprod_{\Mfrunst} \Mhatunst  \right)$-labeled graph, say with underlying \rda-graph $G$ and labeling $\prod_{v \in V(G)} \Sigma_{v}$. Consider the map
\[ \GG^1 \Mfr \coprod_{\GG^1 \Mfrunst} \GG^1 \Mhatunst \to \Mhat\]
induced by properad compositions. It maps $\un{x}$ to the Riemann surface $\Sigma_{\un{x}} = \#_{v \in V(G)} \Sigma_{v}$ obtained by gluing the surfaces $\Sigma_v$ as prescribed by the edges of $G$. In particular, we get an analytic embedding 
\[ \label{gammacircle} \gamma_{\un{x}} \colon \coprod_{E(G) \cup in(G) \cup out(G)} S^1 \to \Sigma_{\un{x}}\]
whose image lies away from the nodes of $\Sigma_{\un{x}}$.

Now consider the analogous maps
\[ \GG^{n+1} \Mfr \coprod_{\GG^{n+1}  \Mfrunst} \GG^{n+1} \Mhatunst \to \Mhat\]
induced by (iterated) properad compositions. Note that these maps factor as 
\[  \GG^{n+1} \Mfr \coprod_{\GG^{n+1}  \Mfrunst} \GG^{n+1} \Mhatunst \to  \GG^1 \Mfr \coprod_{\GG  \Mfrunst} \GG^1 \Mhatunst \to \Mhat \] 
where the first map is the composition of simplicial degeneracy maps $\underbrace{d_0 \circ ... \circ d_0}_{\text{$n$-times}}$ and corresponds to forgetting all the nestings in the underlying graph of an element, and the second map is as described above. In particular, to any 
\[\un{x} \in  \GG^{n+1} \Mfr \coprod_{\GG^{n+1}  \Mfrunst} \GG^{n+1} \Mhatunst,\]
we can assign an embedding $\gamma_{\un{x}}$ as in \ref{gammacircle}. For an example see Figure \ref{gamma_fig}.



\begin{figure}
	\centering
		\includegraphics[width=\textwidth]{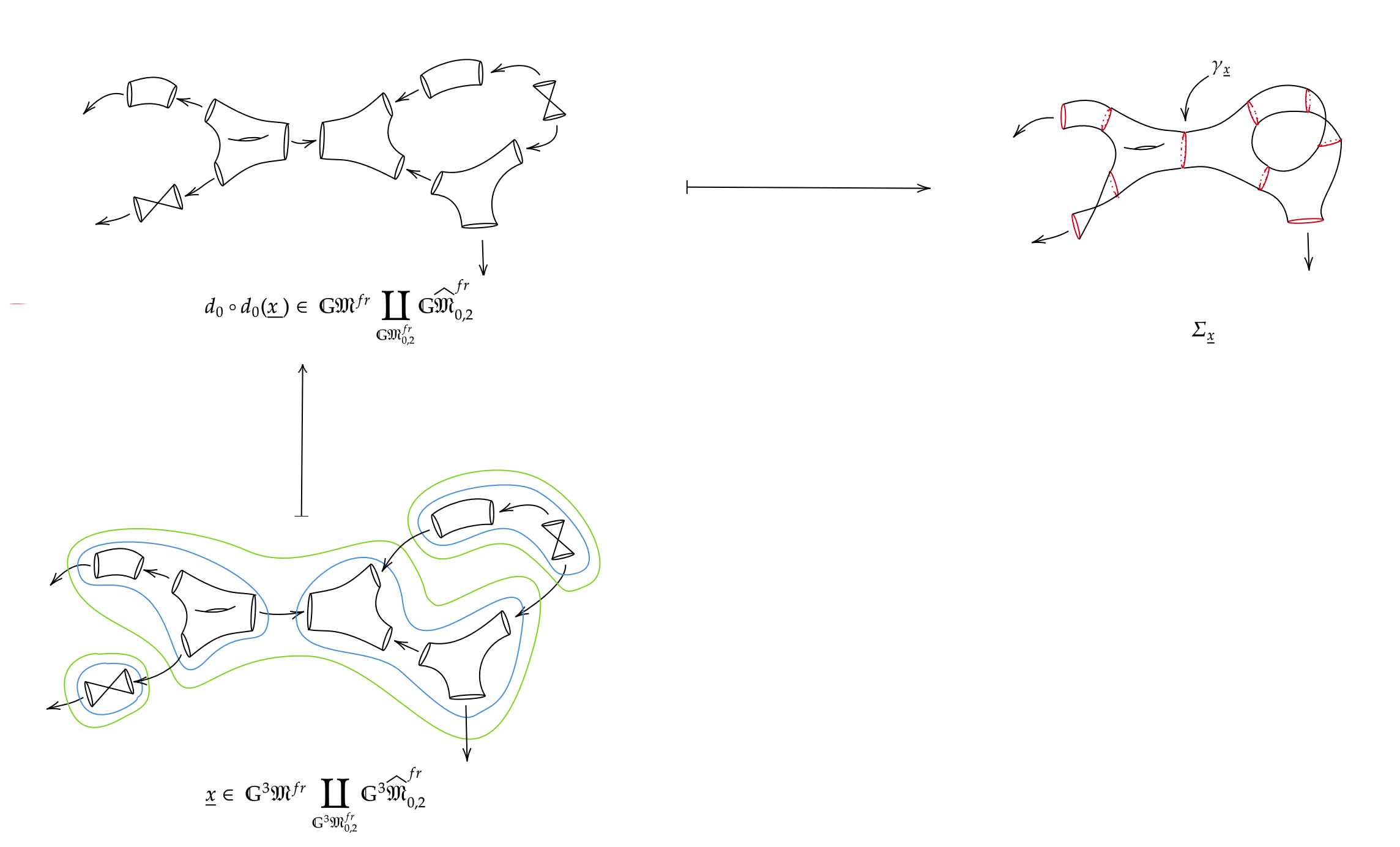}
		\caption{Construction of the embedding $\gamma_{\un{x}}$ for an $\un{x}$ as shown in the bottom left corner. The surface $\Sigma_{\un{x}}$ obtained by gluing the inner-most labels of $\un{x}$ is shown in the top right corner. The image of $\gamma_{\un{x}}$ inside $\Sigma_{\un{x}}$ is indicated in red.}
		\label{gamma_fig}
	\end{figure}

	Recall the description of the family of curves $\catc_{\Sigma}$ over $U_{\Sigma}$ given above. For each $n, m$ and $\nu$, let $A_{n,\nu}^{(m)}  \subset \cfam_{\Sigma}$ be the subset defined, in terms of this descriptions, by 
\[ A_{n,\nu}^{(m)}  =  \left\{ (z, \un{t}) \in R_{\hat{\nu}} \times \DD^{\widetilde{\Nu}} \Bigg \lvert |t_\nu| \in \left[0,\frac{1}{n^2}\right] \mbox{ and }  |z_{\hat{\nu}}(z) | \in \left[ \frac{|t_{\nu}|}{\frac{1}{n}\left( 1 + \frac{1}{m} \right)}, \frac{1}{n}\left( 1 + \frac{1}{m} \right) \right] \right\}. \]
	Set 
	\[A_{n}^{(m)}= \bigcup_{\nu \in \Nu} A_{n,\nu}^{(m)}.\]

	Define $\FF_{n, \Sigma}^{(m)}$ to be the simplicial subspace of $\GG^{\bullet+1} \Mfr \coprod_{\GG^{\bullet+1}  \Mfrunst} \GG^{\bullet+1} \Mhatunst$ consisting of simplices $\un{x}$ such that\\ 
\begin{equation}\label{enum: W subcmplx}
	\begin{minipage}{0.92 \textwidth}
		\begin{enumerate}
				\item the associated embedding $\gamma_{\un{x}}$ is disjoint from $A_{n}^{(m)}$, and
				\item if $\Sigma_u$ is one of the inner-most labels of $\un{X}$
					(i.e. an element of $\underbrace{d_0 \circ ... \circ d_0}_{\text{$n$-times}}$, the labeled graph obtained by forgetting all the nestings of the graph underlying $\un{x}$) which, considered as a subset of $\Sigma_{\un{x}}$, intersects $A_{n}^{(m)}$, then each component of $\Sigma_u \setminus A_{n}^{(m)}$ has at least one output boundary.
		\end{enumerate}
	\end{minipage}
\end{equation}
We then define $F_{n, \Sigma}^{(m)}$ to be an open subset of $\left| \GG^{\bullet+1} \Mfr \coprod_{\GG^{\bullet+1}  \Mfrunst} \GG^{\bullet+1} \Mhatunst \right|$ which fiberwise deformation retracts onto this subcomplex.\\
Finally, we set 
\[\WW_{n, \Sigma}=\bigcup_{m} \FF_{n, \Sigma}^{(m)} \quad \mbox{ and } \quad W_{n, \Sigma} = \bigcup_{m} F_{n, \Sigma}^{(m)}.\]
Note that $W_{n, \Sigma}$ deformation retracts onto the geometric realization of $\WW_{n, \Sigma}$.

Let $U_{\Sigma}^{sing}$ denote the subset
\[ \{ (z,\un{t})| |t_{\nu}| =0 \mbox{ for some } \nu\}. \]
This is the locus of nodal Riemann surfaces in $U_{\Sigma}$. Denote by $U^n_{\Sigma}$ the neighborhood of $U_{\Sigma}^{sing}$ given by
\[  \{(z,\un{t})| |t_{\nu}| \leq \frac{1}{n^2} \mbox{ for some } \nu\}. \]
In the constructions below we will also use the subcomplex $\FF_{n, \Sigma}^{(m)} \cap \pi^{-1}(U^n_{\Sigma})$. Denote this by $\FF_{n, \Sigma}^{(m)}|_{U^n_{\Sigma}}$. $F^n (U_{\Sigma}) \cap \pi^{-1}(U^n_{\Sigma})$ give an open subset of $\pi^{-1}(U_{\Sigma})$ fiberwise deformation retracting onto this subcomplex.

%

\begin{figure} 
	\centering
	\begin{subfigure}[b]{0.3\textwidth}
		\includegraphics[width=\textwidth]{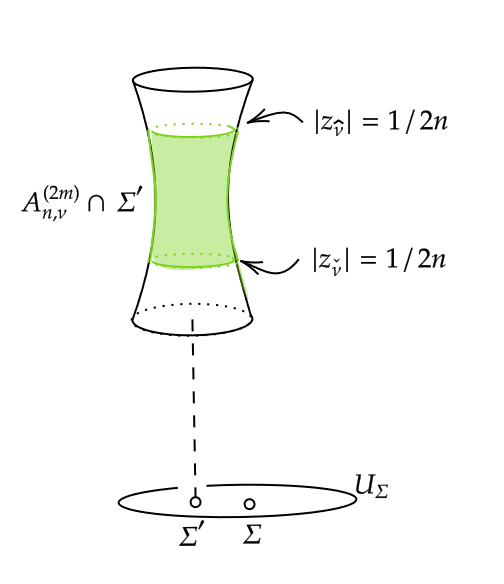}
		\caption{Part of $\Sigma^\prime$ lying in $A^n_{\nu}$, with  $\Sigma^\prime$ thought of as a fiber of $\catc_{\Sigma}$}
		\label{filtrationfig1}
	\end{subfigure}
	\begin{subfigure}[b]{0.35\textwidth}
		\includegraphics[width=\textwidth]{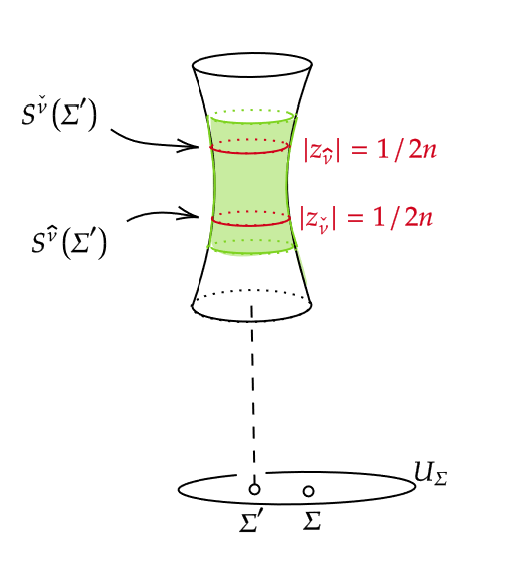}
		\caption{Circles $S_{\hat{\nu}}, S_{\check{\nu}}$}
		\label{filtrationfig2}
	\end{subfigure}
	\caption{}
\end{figure}

\subsection{Proof of  condition \eqref{filtrantncondition}}\label{prfcondfiltrtop}
Let us now turn to the fiberwise contractions. We begin by describing a number of preliminary notions and maps:

	\subsubsection{Annuli $A_{\nu}$ and curves $S^{\check{\nu}}, S^{\hat{\nu}}$ }\label{curveannuluspara}
For any $\Sigma^{\prime} \in U^n_{\Sigma}$, denote by $A_{\nu}(\Sigma^\prime)$ the annulus 
\[ A^{(2m)}_{n,\nu} \cap \Sigma^\prime \]
where we think of $\Sigma^\prime$ as sitting inside the universal curve $\cfam_{\Sigma}$ as a fiber. Denote by $S^{\check{\nu}}, S^{\hat{\nu}}$ the boundaries of this annulus given by 
\[ S^{\hat{\nu}}(\Sigma^{\prime}) = \left\{(z,t_{\hat{\nu}}(\Sigma^{\prime})) \Big \lvert |z_{\hat{\nu}}(z)| = \frac{1}{n}\left(1 + \frac{1}{2m} \right) \right\} , S^{\check{\nu}}(\Sigma^{\prime}) = \left\{(z,t_{\hat{\nu}}(\Sigma^\prime)) \Big \lvert |z_{\hat{\nu}}(z)| =\frac{|t_{\hat{\nu}}|(\Sigma^\prime)}{\frac{1}{n}\left(1 + \frac{1}{2m} \right)}\right\} \subset \Sigma^{\prime}\]
Now consider the graph $G_{\Sigma^{\prime}}$, with
\begin{itemize}
	\item vertices given by the connected components of $\Sigma^{\prime} \setminus \bigcup_{\Nu} \partial A_{\nu}(\Sigma^{\prime})$, and
	\item edges correspond to the connected components of $\bigcup_{\Nu} \partial A_{\nu}(\Sigma^{\prime})$, with all edges directed away from vertices labeled by the annuli $A_{\nu}(\Sigma^{\prime})$.
\end{itemize}
Denote by $\un{x}_{\Sigma^{\prime}}$ the $0$-simplex with the underlying graph $G_{\Sigma^\prime}$ and with the labeling of each vertex given by the corresponding connected component of $\Sigma^{\prime} \setminus \bigcup_{\Nu}\partial A_{\nu}(\Sigma^{\prime})$. 
\subsubsection{Graphs $H_v$}\label{hvgrph} Let $\un{x}$ be an element in the space of $k$-simplices $(\FF_{n, \Sigma}^{(m)}|_{U^n_{\Sigma}})^k$. Denote the underlying \rda-graph by $G$. Let $\un{x}_v$ denote the label of $v \in V(G)$ in $\un{x}$ and let $\Sigma^{\prime}_v$ be the surface obtained by gluing the inner-most labeles of $\in{x}_v$. Since $\un{x}$ lies in $\FF_{n, \Sigma}^{(m)}$, it follows that for any $\nu \in \Nu$, the intersection $A_{\nu}(\Sigma^{\prime}) \cap \Sigma^{\prime}_v$ is either empty or $A_{\nu}(\Sigma^{\prime})$. We can then construct a graph $H_v$ as follows:
		\begin{itemize}
			\item the vertices of $H_v$ are given by the connected components of $\Sigma^{\prime}_v \setminus \bigcup_{\Nu} \partial A_{\nu}$, and
			\item the edges of $H_v$ are given by the connected components of $\bigcup_{\Nu} \partial A_{\nu}$ lying inside $\Sigma^{\prime}_v$, with each edge directed away from the vertex corresponding to a component of the form $A_{\nu}$ for some $\nu$.
		\end{itemize}
		The following observation will turn out to be important in the construction of the cut maps and cut homotopies below in Sections \ref{cutmaps} and \ref{para: cut homotopy}:
		\begin{lemma}\label{rmk: at least one output}
		Every vertex in $H_v$ has at least one output.
		\end{lemma}
		\begin{proof}
			The proof follows by observing that $\Sigma^{\prime}_v$ is obtained by composing labels of a subgraph of $\underbrace{d_0 \circ ... \circ d_0}_{\text{$n$-times}}(\un{x})$, each of which has at least one output and moreover satisfies the condition mentioned in \ref{enum: W subcmplx}, (2).
		\end{proof}
		

\begin{figure}
	\centering
		\includegraphics[width=\textwidth]{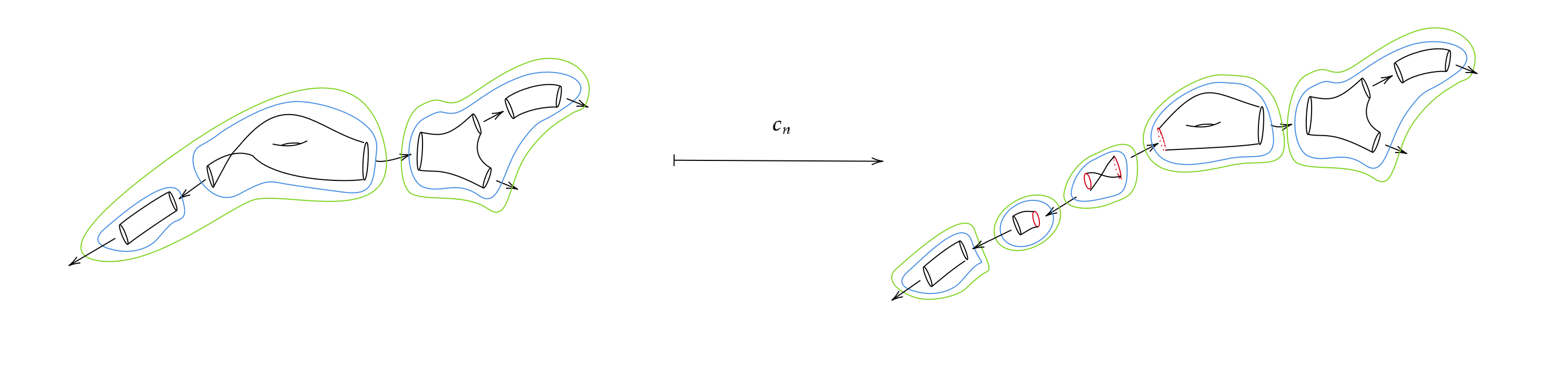}
		\caption{An example of cut map $c_n$ for $n=2$}
		\label{cutmap_fig}
	\end{figure}


	\subsubsection{The cut maps}\label{cutmaps} We define a simplicial map 
	\[c \colon \FF_{n, \Sigma}^{(m)}|_{U^n_{\Sigma}} \to \WW_{n, \Sigma}|_{U^n_{\Sigma}}\]
	 which will be used in the definitions of homotopies below. Intuitively, it corresponds to making `cuts' to the underlying Riemann surface along the circles $S^{\hat{\nu}}$ and $S^{\check{\nu}}$ (see Figure \ref{cutmap_fig}).
	Maps 
	\[c_r \colon (\FF_{n, \Sigma}^{(m)}|_{U^n_{\Sigma}})^r \to (\WW_{n, \Sigma}|_{U^n_{\Sigma}})^r\]
	 are defined inductively as follows:
		\begin{itemize}
			\item Base case: Let $\un{x} \in (\FF_{n, \Sigma}^{(m)}|_{U^n_{\Sigma}})^0$ and let $G$ be the graph underlying $\un{x}$. For a vertex $v \in G$ let $\Sigma_v$ and $H_v$ be as described above in Section \ref{hvgrph}. Consider an element $\un{y}_v \in \WW^0_{n,\Sigma_v}$ whose underlying \rda-graph is $H_v$ and whose vertex labels are given by the corresponding connected components of $\Sigma_v \setminus \bigcup_{\Nu} \partial A_{\nu}$.  The fact that such a labeled \rda-graph indeed lies in $\WW^0_{n, \Sigma_v}$ follows from Lemma \ref{rmk: at least one output}.
			
			We then define $c_0(\un{x})$ as the graph composition 
			\begin{equation}\label{c0eq} c_0(\un{x}) := \mu_G( \prod \un{y}_v ). \end{equation}
			\item Assume inductively that $c_i$ for $0 \leq i <r$, have been defined and satisfy the following property: For any $\un{x} \in(\WW_{n, \Sigma}|_{U^n_{\Sigma}})^i$ with underlying graph $G$, 
			\[c_{i} (\un{x}) = \mu_{G}( \prod \un{y}_v)\] 
			for some $\un{y}_v \in (\WW_{n, \Sigma}|_{U^n_{\Sigma}})^i$. Notice that the condition is satisfied for $i=0$.\\
		Now suppose $\un{x} \in(\WW_{n, \Sigma}|_{U^n_{\Sigma}})^r$, having underlying \rda-graph $G$. Consider the element $c_{r-1}(d_0 \un{x}) \in(\WW_{n, \Sigma}|_{U^n_{\Sigma}})^{r-1}$. From the induction hypothesis we know that 
		$c_{r-1}(d_0 \un{x}) =  \mu_{G_{d_0 \un{x}}}(\prod \un{y}^{\prime \prime}_{w})$
		 for some $\un{y}^{\prime \prime}_w$, $w \in V( G_{d_0\un{x}} )$, where $G_{d_0\un{x}}$ is the \rda-graph underlying $d_0\un{x}$. But since  $G_{d_0 \un{x}}  = \mu_G( \prod G_v)$ for some graphs $G_v$, we can express $c_{r-1}(d_0 \un{x})$  as a graph composition $\mu_G( \prod \un{y}^{\prime}_v )$ for some $\un{y}^{\prime}_v \in (\WW_{n, \Sigma}|_{U^n_{\Sigma}})^{r-1}$. For $v \in G$, consider an element $\un{y}_v \in ^r_{\Sigma}$ defined as follows: the \rda-graph underlying $\un{y}_v$ is $H_v$. The label of a vertex $u \in  H_v$ in $\un{y}_v$ is the labeled subgraph of $\un{y}^{\prime}_v$ consisting of vertices and edges mapping under $\pi$ into the connected component $\Sigma_u$ of $\Sigma \setminus \bigcup_{\Nu}\partial A_{\nu}$ corresponding to $u$, as in Section \ref{hvgrph}. Then, we set 
		\[c_r(\un{x}) = \mu_G( \prod \un{y}_v).\]
		\end{itemize}

\subsubsection{The cut homotopies}\label{para: cut homotopy} We define a fiber preserving homotopy 
\[\Phi_t \colon F_{n, \Sigma}^{(m)}|_{U^n_{\Sigma}} \to W_{n, \Sigma}|_{U^n_{\Sigma}}\]
from the inclusion $F_{n, \Sigma}^{(m)}|_{U^n_{\Sigma}} \hookrightarrow W_{n, \Sigma}|_{U^n_{\Sigma}}$ to the geometric realization of the cut map, defined above in \ref{cutmaps}. It turns out to be easier to write down the expression for the inverse homotopy 
\[\widetilde{\Phi}_t \colon F_{n, \Sigma}^{(m)}|_{U^n_{\Sigma}} \to W_{n, \Sigma}|_{U^n_{\Sigma}}\]
going from the geometric realization of the cut map to the inclusion $F_{n, \Sigma}^{(m)}|_{U^n_{\Sigma}} \hookrightarrow W_{n, \Sigma}|_{U^n_{\Sigma}}$.\\
$\widetilde{\Phi}$ is defined as the geometric realization of a simplicial homotopy 
\[\widetilde{\phi} \colon \FF_{n, \Sigma}^{(m)}|_{U^n_{\Sigma}} \to \WW_{n, \Sigma}|_{U^n_{\Sigma}},\]
which in turn consists of maps $\widetilde{\phi}_j \colon(\FF_{n, \Sigma}^{(m)}|_{U^n_{\Sigma}})^r \to (\WW_{n, \Sigma}|_{U^n_{\Sigma}})^{r+1}$ for $0 \leq j \leq r$, satisfying the usual conditions of a simplicial homotopy.
\begin{itemize}
\item Base Case: Let $\un{x} \in(\FF_{n, \Sigma}^{(m)}|_{U^n_{\Sigma}})^0$ and let $G$ be the graph underlying $\un{x}$. Suppose that $c_0(\un{x})=\mu_G(\prod \un{y}_v)$. Then, we define 
	\[ \widetilde{\phi}_0(\un{x}) := \mu_G( \prod \eta(\un{y}_v)), \]
	where we recall from Section \ref{GGtripsubsec} that $\eta \colon \mathbbm{1}_{\redTopSeq} \to \GG$ is the unit of the monad $\GG$.
\item Assume inductively that the maps $\widetilde{\phi}_j$ have been defined for all degrees $i < r$ and satisfy the following property similar to the one used above in \ref{cutmaps}: For any $\un{x} \in (\FF_{n, \Sigma}^{(m)}|_{U^n_{\Sigma}})^{i}$, with the underlying graph $G$, 
\[\widetilde{\phi}_{j} (\un{x}) = \mu_{G}( \prod \un{y}_v)\]
 for some $\un{y}_v \in  (\FF_{n, \Sigma}^{(m)}|_{U^n_{\Sigma}})^{i}$. The maps $\widetilde{\phi}_j \colon (\FF_{n, \Sigma}^{(m)}|_{U^n_{\Sigma}})^{r} \to (\WW_{n, \Sigma}|_{U^n_{\Sigma}})^{r+1}$ are then defined as follows for $\un{x} \in  (\FF_{n, \Sigma}^{(m)}|_{U^n_{\Sigma}})^{r}$ with underlying \rda-graph  $G$: 
	\begin{itemize}
			\item $j=0$: For $\un{x} \in  (\FF_{n, \Sigma}^{(m)})^{r}$, suppose that $c_r(\un{x}) = \mu_G(\prod \un{y}_v)$. Then define
		\[ \widetilde{\phi}_j(\un{x}) := \mu_G( \prod \eta(\un{y}_v)) .\]
		\item $j > 0$: From the condition in the inductive hypothesis we have $\widetilde{\phi}_{j-1}(d_0\un{x}) = \mu_G( \prod \un{y}_v)$, for some $\un{y}_v$. Then define
		 \[\widetilde{\phi}_j(\un{x}) := \mu_G ( \prod \eta(\un{y}_v) ).\]
	\end{itemize}
\end{itemize}

\subsubsection{The homotopy $\Psi$}\label{psiproprmain} Let $\WW_{n, \Sigma}|_{U^n_{\Sigma}}^M$ be the subcomplex of $\WW_{n, \Sigma}|_{U^n_{\Sigma}}$ consisting of simplices $\un{x}$ in which the annuli $A_\nu$ are isolated in the following sense: if $\un{x}_v$  is the label of a vertex in $\un{x}$, $\Sigma_v$ the surface obtained by gluing the inner-most labels of $\un{x}_v$, then 
\[\Sigma_v \cap A_{\nu} \neq \phi \Rightarrow \un{x}_v = \eta^r(A_{\nu}), \mbox{ for any } \nu \in \Nu.\]
Let $W_{n, \Sigma}|^M_{U^n_{\Sigma}}$ be the geometric realization of $\WW_{n, \Sigma}|_{U^n_{\Sigma}}^M$.\\
Let $\chi^M_{n, \Sigma} \colon U_\Sigma^{n}  \to W_{n, \Sigma}|^M_{U^n_{\Sigma}}$ be a section of $\pi$ over $U_\Sigma^{n}$ defined as follows: the image of $\chi_{n, \Sigma}^M $ is contained inside the geometric realization of the $0$-skeleton $(\WW_{n, \Sigma}|_{U^n_{\Sigma}})^0$ with the associated map given by
\[ \chi_{n, \Sigma}^M \colon \Sigma^\prime \mapsto \un{x}_{\Sigma^{\prime}} \mbox{ for every } \Sigma^{\prime} \in U^n_{\Sigma},\]
where $\un{x}_{\Sigma^{\prime}}$ is a $0$-simplex as defined in \ref{curveannuluspara}. Denote the constant simplicial space corresponding to the image of $\chi_{n, \Sigma}^M$ by $\mathbb{X}_{n, \Sigma}^M$ (recall a constant simplicial space is a simplicial space with all face and degeneracy maps given by identities).


We will construct a fiber-preserving homotopy $\Psi$ contracting $W_{n, \Sigma}|_{U^n_{\Sigma}}^M$ onto the image of $\chi_{n, \Sigma}^M$. This will again be a geometric realization of a simplicial homotopy $\psi$, which will be constructed by specifying extra degeneracies and augmentation on $\WW_{n, \Sigma}|_{U^n_{\Sigma}}^M$ (cf. \cite[Section 4.5]{riehl2014categorical}). For this we need to define:
	\begin{enumerate}
	 \item Extra degeneracies: $s_{-1} \colon (\WW_{n, \Sigma}|_{U^n_{\Sigma}}^M)^r \to (\WW_{n, \Sigma}|_{U^n_{\Sigma}}^M)^{r+1}$ for all $r \geq -1$, where $(\WW_{n, \Sigma}|_{U^n_{\Sigma}}^M)^{-1} := \mathbb{X}_{n, \Sigma}^M$, and
	 \item Augmentation: $d_0 \colon (\WW_{n, \Sigma}|_{U^n_{\Sigma}}^M)^0 \to (\WW_{n, \Sigma}|_{U^n_{\Sigma}}^M)^{-1}$,
	\end{enumerate}
where $(\WW_{n, \Sigma}|_{U^n_{\Sigma}}^M)^{-1}= \mathbb{X}_{n, \Sigma}^M$, satisfying:
	\begin{enumerate}
		\item $d_{0}s_{-1}$ is identity on $(\WW_{n, \Sigma}|_{U^n_{\Sigma}}^M)^r$
		\item for $r \geq 1, 0 \leq i,j \leq r$, 
			\begin{enumerate}
			\item $d_{i+1}s_{-1} = s_{-1}d_i$
			\item $s_{j+1}s_{-1} = s_{-1}s_j$
			\end{enumerate}
		on $(\WW_{n, \Sigma}|_{U^n_{\Sigma}}^M)^r$.
	\end{enumerate}
	The contracting homotopy $\psi$ underlying $\Psi$ is then obtained by setting $\psi_i \colon (\WW_{n, \Sigma}|_{U^n_{\Sigma}}^M)^r \to (\WW_{n, \Sigma}|_{U^n_{\Sigma}}^M)^{r+1}$  to be $\underbrace{s_0 \hdots s_0}_{i-\mbox{times}} s_{-1} \underbrace{d_0 \hdots d_0}_{i-\mbox{times}}$ for $0 \leq i \leq r$.
	
	\begin{itemize}\label{extradegpsi}

	\item Extra degeneracies: We define $s_{-1} \colon (\WW_{n, \Sigma}|_{U^n_{\Sigma}}^M)^r \to (\WW_{n, \Sigma}|_{U^n_{\Sigma}}^M)^{r+1}$. Let $\un{x} \in (\WW_{n, \Sigma}|_{U^n_{\Sigma}}^M)^r$ and $\Sigma_{\un{x}}$ the Riemann surface obtained by gluing the inner-most labels of $\un{x}$. Let $G_{\Sigma_{\un{x}}}$ be the graph for the Riemann surface $\Sigma_{\un{x}}$ as defined in \ref{curveannuluspara}. From the definition of $\WW_{n, \Sigma}|_{U^n_{\Sigma}}^M$, it follows that the vertices of $\un{x}$ can be regrouped into subgraphs $\un{y}_v \in  (\WW_{n, \Sigma}|_{U^n_{\Sigma}}^M)^r$ labeled by $v\in G_{\Sigma_{\un{x}}}$ so that 
$\un{x} = \mu_{G_{\Sigma_{\un{x}}}} ( \prod \un{y}_v)$.
	Then  set 
	\[ s_{-1}(\un{x}) := \mu_{G_{\Sigma_{\un{x}}}} \left( \prod \eta(\un{y}_v)\right). \]
\item Augmentation: Note that for $\un{x} \in  (\WW_{n, \Sigma}|_{U^n_{\Sigma}}^M)^0 $, $d_1(s_{-1}\un{x}) = \un{x}_{\Sigma_{\un{x}}} \in \mathbb{X}_{n, \Sigma}^M$. We define $d_0 \colon (\WW_{n, \Sigma}|_{U^n_{\Sigma}}^M)^0 \to \mathbb{X}_{n, \Sigma}^M$ as
	\[ d_0 = d_1 \circ s_{-1}. \]
	\end{itemize} 

	\subsubsection{The homotopy ${\Psi}^{ns}$}\label{psihtpy} We construct another fiber-preserving homotopy, this time contracting $W_{n, \Sigma}^{ns}=W_{n, \Sigma} \cap \pi^{-1}(U_{\Sigma}^{ns})$ onto a section ${\chi}^{ns}$ of $\pi$ over $U_{\Sigma}^{ns}$. Here 
\[U_{\Sigma}^{ns} = \{(z,\un{t})| t_{\nu} \neq 0 \mbox{ for all } \nu\} \subset U_{\hat{\Sigma}} \times \DD^{\Nu} = U_{\Sigma}\]
is the subspace of non-singular curves in $U_{\Sigma}$. Let $\WW_{n, \Sigma}^{ns}$ be the corresponding subcomplex of $\WW_{n, \Sigma}$.

The section ${\chi}_{\Sigma}^{ns}$ is contained in the $0$-skeleton of $\WW_{n, \Sigma}^{ns}$ and the corresponding map is given by \[\Sigma^\prime \mapsto \widetilde{\un{x}}_{\Sigma^\prime}\]
where $\tilde{\un{x}}_{\Sigma^\prime}$ is the $0$-simplex such that the underlying graph has a single vertex, no internal edges, and the label of the vertex is given by $\Sigma^\prime$. Note that $\Sigma^\prime$ is non-singular and hence $\widetilde{\un{x}}_{\Sigma^\prime}\in \GG \Mfr \subset (\WW_{n, \Sigma})^0$.

Denote by $\mathbb{X}^{ns}_{\Sigma}$ the simplicial subspace of $\WW_{n, \Sigma}^{ns}$ corresponding to the image $\chi_{\Sigma}^{ns}$. It is a constant simplicial space and is contained inside the $0$-skeleton of $\WW_{n, \Sigma}^{ns}$.
Again we shall construct the contraction by specifying extra degeneracies and augmentation as in \ref{psiproprmain}. We continue using the same notation for these as in \ref{psiproprmain}.
	\begin{itemize}
		\item Extra degeneracies: In this case the map $s_{r+1} \colon (\WW_{n, \Sigma}^{ns})^r \to  (\WW_{n, \Sigma}^{ns})^{r+1}$ is defined by 
			\[ s_{-1}(\un{x}) _= \eta(\un{x}) .\]
		\item The augmentation $d_0 \colon  (\WW_{n, \Sigma}^{ns})^0 \to \mathbb{X}^{ns}_{\Sigma}$ is defined by 
			\[ d_0 := d_1 \circ s_{-1} .\]
	\end{itemize}
	The extra degeneracies $s_{-1}$'s and augmentation $d_0$ satisfy conditions analogous to those mentioned in \ref{psiproprmain}. The simplicial homotopy ${\psi}^{ns}$ underlying ${\Psi}^{ns}$ is obtained from $s_{r+1}$'s and $d_0$ by formulas similar to those in \ref{psiproprmain}.

\subsubsection{The contraction for proving condition \eqref{filtrantncondition}}\label{mthm1finalparag}
We are finally ready to define the homotopy needed to show that condition \eqref{filtrantncondition} is satisfied:\\
We construct a homotopy $\Omega$ from the inclusion $F_{n, \Sigma}^{(m)} \hookrightarrow W_{n, \Sigma}$ to a map which takes $F_{n, \Sigma}^{(m)}$ onto a section of $\pi$ over $U_{\Sigma}$. The section coincides with $\chi_{n, \Sigma}^M$ near the singular locus $U_{\Sigma}^{sing}$, and with $\chi_{\Sigma}^{ns}$ away from it.

Fix a continuous function $\zeta \colon U_{\Sigma} \to [0,1]$ such that 
\begin{align*}
\zeta &\equiv 1 \mbox{ on } \bigcup_{\Nu} \{|t_{\nu}|\leq \frac{1}{2n^2}\} \mbox{ and }\\
\zeta &\equiv 0 \mbox{ outside of } U^n_{\Sigma} = \bigcup_{\Nu} \{|t_{\nu}|\leq \frac{1}{n^2}\}
\end{align*}
Set $\kappa = \zeta \circ \pi$. The homotopy $\Omega$ is defined as follows:
	\begin{align*}
		\mbox{ On } \pi^{-1}\left(\bigcup_{\Nu} \{|t_{\nu}|\leq \frac{1}{2n^2}\}\right)&: \Omega(\un{x},t) = 
		\begin{cases} 
			\Phi(\un{x},2t), & t < 1/2 \\
			\Psi(\Phi(\un{x},1), 2t-1) & t \geq 1/2
		\end{cases}\\
	\mbox{ On } \pi^{-1}(U_{\Sigma}^{ns})&: \Omega(\un{x},t) = 
		\begin{cases} 
			\Phi(\un{x},2t), & t <  \kappa(\un{x})/2 \\
			\Phi({\Psi^{ns}}(\un{x},\frac{t-\kappa/2}{1-\kappa/2}), \kappa(\un{x})) & t \geq \kappa({\un{x}})/2
		\end{cases}
	\end{align*}
	To ensure that $\Omega$ is well-defined it suffices to check that the two definitions agree on points lying over $U_{\Sigma}^{ns} \bigcap \bigcup_{\Nu} \{|t_{\nu}|\leq \frac{1}{2n^2}\}$
	\[ \Psi ( \Phi (\un{x},1),t) = \Phi ({\Psi^{ns}} ( \un{x} , t) ,1 ) \mbox{ for all } t\]
Since $\Phi(\un{x},1)=c(\un{x})$, this amounts to checking for every $r$ and $0 \leq i \leq r$,
\[\psi_i c_r = c_{r+1} \psi^{ns}_{i}\]
on $r$-simplices lying over $U_{n, \Sigma}^{ns} \cap \bigcup_{\Nu} \{|t_{\nu}|\leq \frac{1}{2n^2}|\}$. Since $c$ is simplicial, from the definitions of $\psi_i,\psi^{ns}_i$ we are reduced to checking 
\[ \psi_0 c_r =  c_{r+1} \psi^{ns}_0 \]
This can be verified directly from the definition of $\phi_0$ and $\psi^{ns}_0$.

This shows that the filtration $\{F_{n, \Sigma}^{(m)}\}$ constructed in Section \ref{gammasubsubsect} satisfies the condition \eqref{filtrantncondition} and thus completes the proof of the first part of Theorem \ref{hatthm}.

\section{From the \delCFT-properad to the \CFT-properad}\label{fulltcftsec}
In this section we prove the first part of Theorem \ref{frtofullthm}. Let us start by recalling the statement:
\begin{theorem}
$\Mfrfull$ is the homotopy colimit of the following diagram in the category of \reduced properads
\begin{equation}\label{eq: fulltcftsec thm}
	\begin{tikzcd}
		\Mfrunst \ar{r}\ar{d} & \Mfrfullnop \\
		\Mfr &
	\end{tikzcd}
\end{equation}
\end{theorem}

To compute the homotopy pushout of the diagram \eqref{eq: fulltcftsec thm}, we first replace $\Mfrfull$ and $\Mfrfullunst$ with homotopy equivalent properads $\mfrfull$ and $\mfrfullnopair$ which are more suitable for the computation. A suitable modification of the strategy used for proving Theorem \ref{hatthm}(1) is then used to complete the proof.

\subsection{Properads $\mfrfullnopair$ and $\mfrfull$}\label{subsec: splitsurf}

Define $\mfrtilde^{fr}_{g,n_-,n_+}(k)$ to be the moduli space of tuples $(\Sigma, \un{s} , \un{t})$, where
\begin{itemize}
	\item $\Sigma \in \Mmod^{fr}_{g,n_-,n_+}$,
	\item $\un{s} = (s_1,...,s_k)$ a tuple of $k$ labeled marked points in the interior of $\Sigma$, and 
	\item $\un{t} = (t_1,...,t_k) \in [0,1]^k$, such that $t_i=1$ for at least one $i$ whenever $n_+=0$.
\end{itemize}
Define $\mfrfullmod_{g,n_-,n_+}(k)$ to be the quotient 
\[\mfrfullmod_{g,n_-,n_+}(k):= \mfrtilde^{fr}_{g,n_-,n_+}(k)  / \Sigma_k\]
where $\Sigma_k$ acts on $\mfrfullmod_{g,n_-,n_+}(k)$ by permuting the labels of the marked points and the co-ordinates of $\un{t}$. We can think of  $\mfrfullmod_{g,n_-,n_+}(k)$ as the moduli space of Riemann surfaces with $k$ unlabeled marked points such that each point carries a weight in $[0,1]$ and at least one of the weights is positive when the surface has no outputs. We refer to these marked points as \emph{weighted \spots} and to a point in $\mfrfullmod_{g,n_-,n_+}(k)$ as a `\emph{\splitsurface}'.\\
Now define 
\[\mfrfullmod_{g,n_-,n_+} := \coprod_{k \geq 0} \mfrfullmod_{g,n_-,n_+}(k)/ \sim\]
where $\sim$ identifies a \splitsurface\  with the one obtained by forgetting the \spots which have weight $0$. More precisely, $\mfrfullmod_{g,n_-,n_+}$ is obtained as the colimit 
\[ \mfrfullmod_{g,n_-,n_+} = \varinjlim \big[ F_0 \mfrfullmod_{g,n_-,n_+} \to F_1 \mfrfullmod_{g,n_-,n_+} \to \hdots \big] \] 
where the spaces $F_k \mfrfullmod_{g,n_-,n_+}$ are defined inductively, satisfying the following properties:
\begin{enumerate}
	\item $F_0 \mfrfullmod_{g,n_-,n_+}  = \mfrfullmod_{g,n_-,n_+}(0)$ 
\item $F_k \mfrfullmod_{g,n_-,n_+} $ admits a map from $\big[ \mfrfullmod_{g,n_-,n_+}(k+1) \big]^{-}$, where $\big[ \mfrfullmod_{g,n_-,n_+}(k+1) \big]^{-}$ is the subspace of  $\mfrfullmod_{g,n_-,n_+}(k+1)$ of {\splitsurfaces} where at least one \spot has weight $0$.
	\item $F_{k+1} \mfrfullmod_{g,n_-,n_+} $ is obtained inductively from $F_k \mfrfullmod_{g,n_-,n_+} $ as the pushout
	\[ 
		\begin{tikzcd}\label{polkapush}
			\big[ \mfrtilde^{fr}_{g,n_-,n_+}(k+1) \big]^{-} \ar{r}\ar{d} &  \mfrtilde^{fr}_{g,n_-,n_+}(k+1) \ar{d}\\
			F_{k} \mfrfullmod_{g,n_-,n_+} \ar{r} &  F_{k+1} \mfrfullmod_{g,n_-,n_+}.
		\end{tikzcd}
	\]
\end{enumerate}

\begin{notation}
	\begin{enumerate}
		\item The image of a tuple $(\Sigma,\un{x},\un{t})$ in $\mfrfullmod_{g,n_-,n_+}$ will be denoted by $[\Sigma,\un{x},\un{t}]$
		\item We fix notation for some surfaces which appear frequently in the later sections: A \splitsurface\  given by a disk with an output  (respectively, input) boundary along with a single positive weight \spot will be called a positive cup (respectively, cap). Note that the \spot in a positive weight cup necessarily has weight $1$.
	\end{enumerate}
\end{notation}

\begin{proposition}\label{polkaproposition}
	The forgetful map $\mfrfullmod_{g,n_-,n_+} \to \Mmod^{fr}_{g,n_-,n_+}$ is a weak homotopy equivalence.
\end{proposition}
\begin{proof}
	\emph{Case $n_+>0$:} In this case can we consider the section of the forgetful map given by mapping a Riemann surface $\Sigma$ to a \splitsurface\  with underlying Riemann surface $\Sigma$ and with no \spots. Then, there is a homotopy contracting $\mfrfullmod_{g,n_-,n_+}$ onto the image of this section given by shrinking the weights of all the \spots to $0$. 

	\emph{General Case:} The map $\mfrfullmod_{g,n_-,n_+} \to \Mmod^{fr}_{g,n_-,n_+}$ is a fibration and hence it suffices to show that the fiber of this map is contractible.\\
	Fix a surface $\Sigma \in \Mmod^{fr}_{g,n_-,n_+}$. We will show that the fiber over $\Sigma$ has an open cover $\{W_x\}$, indexed by points $x \in\Sigma$, such that every finite intersection $W_{x_1,...,x_n} := W_{x_1} \cap ... \cap U_{x_n}$ is non-empty and contractible.\\
The open sets $W_x$ are defined to be subsets of the fiber consisting of {\splitsurfaces} which have no \spot with weight $\geq 1/3$ at $x$. It is clear that $\{W_x\}$ forms an open cover of the fibre, with every finite intersection non-empty.\\
Let $J=\{x_1,...,x_n\}$ be a finite subspace of $\Sigma$. We now turn to showing contractibility of the finite intersection $W_J=W_{x_1} \cap ... \cap W_{x_n}$. To show this we will exhibit a filtration 
\[F_1 \subset F_2 \subset ... \subset F_n \subset ...  \mbox{ of $W_{x_1,...,x_n}$}\]
such that the inclusion of each $F_n$ is null-homotopic inside $W_J$.\\
The filtration is defined as follows: Choose a conformal embedding $\phi_{x_1} \colon \DD \to \Sigma$ with $\phi_{x_1}(0)=x_1$. For $n \in \NN$, define $F_n$ to be the subspace of the fiber consisting of {\splitsurfaces} for which all the \spots with weight $\geq 1/3$ are contained inside $\Sigma \setminus {\phi_{x_1}}(|z| \leq 1/n)$. The subspaces $F_n$ give an increasing and exhausting sequence of open subsets of $W_J$.\\
We now prove that each $F_n$ is null-homotopic \emph{inside} $W_J$. Consider the point $x_0 \in \Sigma$ given by $x_0 = \phi_{x_1}(1/2n)$, and the \splitsurface\  in $W_J$ which has a single weight $1$ \spot at $x_0$ viz. $[\Sigma,(x_0),(1)]$. We show that $F_n$ can be contracted to this {\splitsurface}  \emph{inside $W_J$}.\\
	The contraction homotopy is given by concatenation of $3$ homotopies $H_1,H_2,H_3$:
	\begin{enumerate}
		\item $H_1$: Consider a function $\tau \colon [0,1] \to [0,1]$ which is identically $0$ on $[0,1/3]$ and identically $1$ on $[2/3,1]$. Let $f_{1}$ be the map $[(\Sigma,\un{x}, (t_1,...,t_k))] \mapsto [ (\Sigma,\un{x}, ( \tau(t_1),...,\tau(t_k))) ]$. Then $H_1$ is a homotopy from the identity to $f_{1}$. The homotopy is simply given by linearly interpolating weights of the \spots.
		\item $H_2$: Now consider the map $f_{2}$ which takes any surface $[(\Sigma,\un{x},\un{t})]$ in $F_n$ to the surface obtained by adding the \spot $x_0$ with weight $1$ to $f_1([\Sigma,\un{x},\un{t}])$. (In particular, note that $f_2$ takes values outside $F_n$.) Define $H_2$ as the homotopy between $f_1$ and $f_2$, given by interpolating the weight of $\gamma_0$ from $0$ to $1$. The well-definedness (and continuity) of $f_2$ and $H_2$ follows from the restriction in the definition of the subset $F_n$ that any \spot contained inside $\phi_{x_1}(\{|z| \leq \frac{1}{n}\})$ has weight $< 1/3$.
		\item $H_3$: Finally, $H_3$ is a homotopy from $f_2$ to the constant map $f_3$ with value $[\Sigma,(x_0),(1)]$. The homotopy $H_3$ is given by decreasing the weights of \spot outside $\phi_{x_1}(\{|z| \leq \frac{1}{n}\})$ to $0$.
	\end{enumerate}
\end{proof}

The properads $\mfrfull$ and $\mfrfullnopair$ are now defined as follows:\\
$\mfrfull$ is defined in a manner similar to $\Mfrfull$ with the space $\Mmod^{fr}_{g,n_-,n_+}$ replaced by the corresponding spaces $\mfrfullmod_{g,n_-,n_+}$.\\
$\mfrfullnopair$ is the \reduced subproperad of $\mfrfull$ defined as follows:
	\begin{itemize}
		\item $\mfrfullnopair(0,1)$ is the subspace of $ \mfrfullnopair_{0,0,1}$ having a positive weight cap i.e. a disk with output having a single weight $1$ \\spot.
		\item $\mfrfullnopair(1,0) =\Mmod_{0,1,1}$ is the subspace of $ \mfrfullnopair_{0,1,0}$ consisting of disks with inputs having a single \spot. Note that the weight of the spot is allowed to vary in $[0,1]$. 
		\item $\mfrfullnopair(1,1) =\Mmod_{0,1,1}$ i.e. the subspace of $\Mmod_{0,1,1}$ \splitsurfaces\  having no positive weight \spots. Similarly,
		\item $\mfrfullnopair(0,2) =\Mmod_{0,0,2}$ i.e. the subspace of $\Mmod_{0,0,2}$ \splitsurfaces\  having no positive weight \spots. 
	\end{itemize}
The properad compositions are given by the maps induced by the gluing of Riemann surfaces underlying the {\splitsurfaces}.

The following is a consequence of Proposition \ref{polkaproposition}:
\begin{corollary}
Maps
\[\mfrfullnopair \to \Mfrfullnop \quad \mbox{ and } \quad \mfrfull \to \Mfrfull\]
are weak equivalences of \reduced properads.\\
\null \hfill \qedsymbol
\end{corollary}

\subsection{Homotopy pushout of \eqref{eq: fulltcftsec thm}}
As mentioned before, to find a model for the homotopy pushout  we shall consider the diagram 
\begin{equation}\label{eq: fullsftsec alt} \mfr \leftarrow \mfrunst \to \mfrfullnopair \end{equation}
which is weakly equivalent to \eqref{eq: fulltcftsec thm}. It follows that
\[ | B_{\bullet} \mfr | \coprod_{|B_{\bullet} \mfrunst |}^{\Galg}|B_{\bullet}\mfrfullnopair|  \]
computes the homotopy pushout of \eqref{eq: fullsftsec alt} and thus of \eqref{eq: fulltcftsec thm}.
From the computation \eqref{barpushout}, we have:
\[| B_{\bullet} \mfr | \coprod_{|B_{\bullet} \mfrunst |}^{\Galg}|B_{\bullet}\mfrfullnopair| = \big| \GG( \GG^{\bullet} \mfr \sqcup_{\GG^{\bullet} \mfrunst} \GG^{\bullet} \mfrfullnopair) \big|\]

To complete the proof of Theorem \ref{frtofullthm} it now suffices to show that:
\begin{equation}\label{eq: fulltcftsec weq}
\pi \colon \big| \GG\big( \GG^{\bullet} \mfr \sqcup_{\GG^{\bullet} \mfrunst} \GG^{\bullet} \mfrfullnopair\big) \big| \to \mfrfull
\end{equation}

\subsection{Map \eqref{eq: fulltcftsec weq} is a weak equivalence of properads}\label{subsec: fulltcftsec htpyeq}
	The strategy of proof is similar to that of Section \ref{finalhtpyeq}. In this case instead of cuts made around nodal points, we will construct cuts around positive weight \spots. This gives a canonical way of decomposing an element of $\mfrfull$ into a surface in $\mfr$ and a collection positive weight caps and cups. As in Section \ref{finalhtpyeq}, some care will be needed to extend this decomposition continuously to a neighborhood of the \splitsurface. Cut homotopies, analogous to those in Section \ref{finalhtpyeq}, will then be constructed to homotope the points in the pushout to the zero simplices corresponding to such decompositions.

To highlight the parallels between the formal structure of the argument here and in Section \ref{finalhtpyeq}, we will use the same notation as in Section \ref{finalhtpyeq} for the analogous notions here. As we shall see, having suitably adopted various definitions to the current context, many of the maps and homotopies will be given by the same expressions as in Section \ref{finalhtpyeq}.


As before, we will find a cover $\{U_{\lambda}\}_{\lambda \in \Lambda}$ of $\mfrfullmod_{g,n_-,n_+}$ such that for any finite subset $J =\{ \lambda_1, ... , \lambda_k\} \subset \Lambda$, 
\[\pi^{-1}|_{U_J} \colon \pi^{-1}(U_J) \to U_J\]
is a weak-equivalence, where $U_J:= U_{\lambda_1} \cap ... \cap U_{\lambda_k}$. 

We will in fact show that each $J$, $U_J$ has a filtration 
\[ W_{1,J} \subset \hdots \subset W_{n,J} \subset \hdots \pi^{-1}(U_J)\]
such that
\begin{equation}\label{filtrfulltcft}
	\pi|_{W_{n,J}} \colon W_{n,J} \to U_J \mbox{ is a weak homotopy equivalence.} 
\end{equation}
Similarly to Section \ref{finalhtpyeq}, we prove this by showing that each $W_{n,J}$ has a further filtration 
\[ F^{(1)}_{n,J} \subset ... \subset  F^{(m)}_{n,J} \subset W_{n,J},\]
satisfying a condition analogous to \eqref{filtrantncondition}, with $W_{n,\Sigma}$ there replaced by $W_{n,J}$.
\begin{remark}
	Notice that unlike Section \ref{finalhtpyeq}, $\chi_n$ is not a section of $\pi$ in this case. In particular the homotopies as in \eqref{finalhtpyeq} which we construct, will not be fiber-preserving. Thus, in contrast to Section \ref{finalhtpyeq}, it is not sufficient to work with a single $U_{\lambda}$ and instead  it becomes necessary to construct the homotopies over all finite intersections $U_J$.
\end{remark}

\subsubsection{Construction of $U_{\lambda}, W_{n,J},$ and $F^{(m)}_{n,J}$}
We start by describing the spaces $U_{\lambda}$. Let $U_{\alpha, i}$ be the open subset of $\mfrfullmod_{g,n_-,n_+}$ consisting of \splitsurfaces\ containing exactly $i$ \spots with weight $>\alpha$ and no \spot with weight $=\alpha$. Also, let $Z_{\alpha,i} \subset U_{\alpha,i}$ be the subspace of \splitsurfaces\  containing exactly $i$ positive weight \spots all of which have weight $1$. There exists a deformation retraction $r_{\alpha,i} \colon U_{\alpha,i} \to Z_{\alpha,i}$ obtained by homotoping the weight $w_{\gamma}$ for an \spot $\gamma$ to 
\[ \begin{cases}
	0 &\mbox{ if } w_{\gamma}<\alpha\\
	1 &\mbox{ if } w_{\gamma}>\alpha.
\end{cases}\]
Notice that we have a covering map $\Mmod^{fr}_{n_-,n_+}(i) \to Z_{\alpha,i}$, where are the spaces $\Mmod^{fr}_{n_-,n_+}(i)$ are as defined at the end of Section \ref{subsec: moduli spaces def}. The fiber over a \splitsurface\  in $Z_{\alpha,i}$ is given by different ways of labeling the \spots. Consider a covering $\{W^{\delta}_{\alpha,i}\}_{\delta \in \Delta_{\alpha,i}}$  of $\Mmod^{fr}_{g,n_-,n_+}(i)$, where $ \Delta_{\alpha,i}$ is some indexing set, satisfying the following conditions:
	\begin{itemize}
		\item $W^{\delta}_{\alpha,i}$ maps homeomorphically onto an open subset of $Z_{\alpha,i}$.
		\item 
			Over $W_{\alpha,i}^{\delta}$, there exists a continuous choice of analytically embedded disks around the \spots in the following sense: Let $\mathcal{C}_{g,n_-,n_+}(i) \to \Mmod^{fr}_{g,n_-,n_+}(i)$ denote the universal curve. Then, there exists maps
		\begin{equation} \label{eq: sigma maps}
			\sigma_{\alpha, i}^{\delta} \colon W_{\alpha,i}^{\delta} \times \sqcup_{j=1}^i \DD \to \mathcal{C}_{g,n_-,n_+}(i)|_{W_{\alpha,i}^{\delta}} 
		\end{equation}
	over $W_{\alpha,i}^{\delta}$, such that $\sigma_{\alpha,i}^{\delta}$ is a fiberwise embedding which maps $0 \in \DD$ in the $j$th disk in the second factor to the $j$th \spot of the fiber. Moreover, we assume that the embeddings given by $\sigma_{\alpha,i}^{\delta}$ are disjoint from the boundaries.
	\end{itemize}

	Let $Z^{\delta}_{\alpha,i}$ be the image of $W^{\delta}_{\alpha,i}$ in $Z_{\alpha,i}$. The collection $\{Z^{\delta}_{\alpha,i}\}$ forms a cover of $Z_{\alpha,i}$. Define $U_{\alpha,i}^{\delta} = r_{\alpha,i}^{-1} (Z_{\alpha,i}^{\delta}) $, where $r_{\alpha,i} \colon U_{\alpha,i} \to Z_{\alpha,i}$ is the retraction mentioned above. The collection $U_{\alpha,i}^{\delta}, \delta \in \Lambda_{\alpha,i}$ gives a covering of $U_{\alpha,i}$.
	Putting all the $\{U_{\alpha,i}^{\delta}\}$'s together as $\alpha, n ,$ and $\delta$ vary in $(0,1), \ZZ_{\geq 0}$, and $\Delta_{\alpha,i}$ respectively, we obtained the desired cover $\{U_{\lambda}\}_{\lambda \in \Lambda}$ of $\mfrfullmod_{g,n_-,n_+}$.

	For a finite collection $J \subset \Lambda$, we now turn to constructing a filtrations $W_{n,J}$ and $F^{(m)}_{n,\Sigma}$ satisfying a condition analogous to \eqref{filtrfulltcft}. Let $\alpha_1, ... , \alpha_k$ be the constants for $J$ as discussed above. Without loss of generality, assume that $0 = \alpha_{0} < \alpha_{1} \leq ... \leq \alpha_{k} < \alpha_{k+1} = 1$. From the definitions of the sets $U_{\lambda_i}$ it is clear that for any surface in the intersection $U_J$ the number of \spots with weight in $>\alpha_r$ stays constant, viz. $i_r$, for every $r$. 

	We now construct a filtration $F^{(m)}_{n,J}$ satisfying \eqref{filtrfulltcft}.
	Analogous to Section \ref{gammasubsubsect}, the filtered pieces are defined to be a neighborhood deformation retracting onto of a simplicial subcomplex $\mathbb{F}^{(M)}_{n,J}$ constructed as follows:
Consider the fiberwise embeddings 
\[(\sigma_{\alpha_{1},i_{1}}^{\lambda_{1}})^{(m)} \colon W_{\alpha_{1},i_{1}}^{\lambda_{1}} \times \left\{ |z| \leq \frac{1}{n}\left(1+\frac{1}{m}\right)\right\}  \to {\mathcal{C}_{g,n_-,n_+}(i)}|_{W_{\alpha,i}^{\lambda}} \]
obtained by restricting the maps $\sigma_{\alpha_{1},i_{1}}^{\lambda_{1}}$ from \eqref{eq: sigma maps}. Let $\phi^{\lambda_1}_{\alpha_{1},i_{1}} \colon U_{\alpha_{1},i_{1}}^{\lambda_{1}} \to W_{\alpha_{1},i_{1}}^{\lambda_{1}}$ denote the composition of the retraction $r_{\alpha_1,i_1} \colon U^{\lambda}_{\alpha_1,i_1}\to Z^{\lambda}_{\alpha_1,i_1}$ and  homeomorphism $ Z^{\lambda}_{\alpha_1,i_1} \to  W^{\lambda}_{\alpha_1,i_1}$. Pulling back the maps $(\sigma_{\alpha_{1},i_{1}}^{\lambda_{1}})^{(m)}$ by $\phi^{\lambda_1}_{\alpha_{1},i_{1}}$ we obtain
\[ (\phi_{\alpha_{1},i_{1}}^{\lambda_{1}})^*(\sigma_{\alpha_{1},i_{1}}^{\lambda_{1}})^{(m)} \colon U_{\alpha_{1},i_{1}}^{\lambda_{1}} \times  \left\{ |z| \leq \frac{1}{n}\left(1+\frac{1}{m}\right)\right\} \to (\phi_{\alpha_{1},i_{1}}^{\lambda_{1}})^* {\mathcal{C}_{g,n_-,n_+}(n)}|_{W_{\alpha,i}^{\lambda}} \]
Denote the image of this map along $A^{(m)}_{n,j}$ and set $A^{(m)}_n = \bigcup_{j} A^{(m)}_{n,j}$.\\
Similar to Section \ref{gammasubsubsect}, given an element $\un{x} \in   \GG^{n+1}\mfr \coprod_{\GG^{n+1}\mfrunst} \GG^{n+1}\mfrfullnopair$, with underlying \rda-graph $G$, we can define an analytic embedding
\[ \gamma_{\un{x}} \colon \coprod_{E(G) \cup in(G) \cup out(G)} S^1 \to \Sigma_{\un{x}}\]
with $\Sigma_{\un{x}}$ the Riemann surface obtained by gluing the inner-most labels of $\un{x}$ (ignoring any \spot present on the labels).
Let $\mathbb{F}^{(m)}_{n,J}$ denote the subcomplex of $\GG^{\bullet}\mfr \coprod_{\GG^{\bullet}\mfrunst} \GG^{\bullet}\mfrfullnopair$ consisting of simplices $\un{x}$ such that the corresponding embedding $\gamma_{\un{x}}$ is disjoint from $A^{(m)}_n$. We then set $F^{(m)}_{n,J}$ to be the open subset of $|\GG^{n+1}\mfr \coprod_{\GG^{n+1}\mfrunst} \GG^{n+1}\mfrfullnopair|$ corresponding to this subcomplex.
Set 
\[\mathbb{W}_{n,J} = \bigcup_{m} \mathbb{F}^{(m)}_{n,j} \mbox{ and } W_{n,J} = \bigcup_{m} F^{(m)}_{n,j}.\]
Note that $W_{n,J}$ deformation retracts onto the subcomplex 

\begin{notation}
Denote the simplicial complex underlying $\pi^{-1}(U_{\lambda_1 ...,\lambda_k})$ by $\WW_{n,J}$.
\end{notation}

\subsubsection{Construction of maps $\chi_n \colon U_J \to F^n \pi^{-1}(U_J)$}\label{subsec: full tcft chi}
We now construct maps $\chi_n$ as in \eqref{filtrfulltcft}. Note that unlike Section \ref{finalhtpyeq}, $\chi_n$ will not be a section in this case of $\pi$. The image of the map lies in the subspace of $0$-skeleton of $\pi^{-1}U_J$.

Recall that from the construction of $U_{\lambda_1}$ it follows that any surface with weighted \spot in $U_{\lambda_1}$ has $i_1$ \spot with weight $>\alpha_1$ and there is a way of labeling these \spot as $1,...,i_1$ with the labels varying continuously over $U_{\lambda_1}$.\\
Let $\un{\Sigma}$ be an element of $U_J$ and let $\Sigma$ be the Riemann surface underlying $\un{\Sigma}$. We identify $\Sigma$ with the fiber over $\un{\Sigma}$ of the pullback of $\mathcal{C}_{g,n_-,n_+}(i_1)$ to $U_J$. Let $x_1,...,x_{i_1}$ denote the corresponding marked points on $\Sigma$. Denote by $A_j^n(\un{\Sigma}) \subset \Sigma$ the image of the (pullback of) map $(\sigma_{\alpha_,i_1}^{\lambda_1})^n$ restricted to the fiber above $\un{\Sigma}$ and by $w_j(\Sigma)$ the weight of the $j$th \spot, $0 \leq j \leq i_1$. 
Set 
\[ \un{A}^n_j(\un{\Sigma}) :=
	\begin{cases}
		[A^{(m)}_{n,j}(\un{\Sigma}),(x_j),\rho(w_j)] \in \mfrfullmod_{0,0,1} &\mbox{ for } w_j < \alpha_n \\ 
		[A^{(m)}_{n,j}(\un{\Sigma}),(x_j),\rho(w_j)=1] \in \mfrfullmod_{0,1,0} &\mbox{ for } w_j > \alpha_n 
	\end{cases}
\]
where $\rho$ is the step function defined on $[0,1] \setminus \{\alpha_1,...,\alpha_k\}$, given by:
\begin{equation}\label{rhofunction} \rho(x) := \begin{cases}
				0 & \mbox{ for } x \in (0,\alpha_1) \\
				\frac{\alpha_{r} + \alpha_{r+1}}{2} &\mbox{ for } x \in (\alpha_{r}, \alpha_{r+1}), 0<r<n \\ 
				1 & \mbox{ for } x \in (\alpha_k,1) 
			\end{cases} 
\end{equation}
(the parametrization of the boundary $\partial A^{(m)}_{n,j}$ is given by the embeddings $\sigma^{\lambda_1}_{\alpha_1,i_1}$).

Consider the complement $\Sigma \setminus \bigcup_j \partial A^{(m)}_{n,j}(\un{\Sigma})$. Denote by $G_J(\un{\Sigma})$ the graph with vertices given by the connected component of $\Sigma \setminus \bigcup_j \partial A^{(m)}_{n,j}(\un{\Sigma})$ and edges given by the components of $\partial A^{(m)}_{n,j}(\un{\Sigma})$, directed away from the vertices corresponding to disks $A^{(m)}_{n,j}(\un{\Sigma})$ if $w_j < \alpha_n$ and towards them if $w_j > \alpha_n$.

Consider a labeling of $G_J(\un{\Sigma})$ with labels of vertices given by the corresponding connected components of $\Sigma \setminus \bigcup \partial A^{(m)}_{n,j}$. The boundaries are considered as inputs or outputs according to the orientations of the edges mentioned above. \\
This gives an element in  the $0$-skeleton $(\WW_{n,J})^0$. Denote this element by $x_{\un{\Sigma}}$.

We then define the map $\chi_n \colon U_J \to W_{n,J}$ by
\[\un{\Sigma} \mapsto x_{\un{\Sigma}} \]

\subsection{Verifying condition \ref{filtrfulltcft}}
We construct on $W_{n,J}$ a homotopy between the identity and $\chi_{n} \circ \pi$. We start by describing notions analogous to those in Section \ref{prfcondfiltrtop}.

\subsubsection{The graphs $K_v$}
These are analogous to graphs $H_v$ defined in Section \ref{hvgrph}.\\
Let $\un{x}$ be an element in $\mathbb{W}_{n,J}$ with underlying graph $G$, and let $\un{x}_v$ be the label of $v \in V(G)$ in $\un{x}$. Let $\un{\Sigma}_{\un{x}}$ and $\un{\Sigma}_{\un{x}_v}$ denote the \splitsurfaces\  obtained by gluing the inner-most labels of $\un{x}$ and $\un{x}_v$ respectively.\\
	Then, denote by $K_v$ the graph with: 
		\begin{itemize}
			\item vertices given by the connected components of $\un{\Sigma}_{\un{x}_v} \setminus \bigcup_{n} \partial A^{(m)}_{n,j} (\un{\Sigma}_{\un{x}})$, and
			\item edges given by connected components of $\bigcup_{n} \partial A^{(m)}_{n,j} (\un{\Sigma}_{\un{x}})$. The edges corresponding to $\partial A^{(m)}_{n,j}$ are directed away from the disk $A^{(m)}_{n,j} (\un{\Sigma}_{\un{x}})$ if $w_j < \alpha_n$ and towards it if $w_j > \alpha_n$.
		\end{itemize}
		
		\subsubsection{Cut map $c \colon F^{(m)}_{n,J} \to W_{n,J}$}\label{subsubsec: fulltcft cut maps}
		As before $c$ is defined as the geometric realization of a map of simplicial spaces given by $c_r \colon \FF^{(m)}_{n,J} \to (\WW_{n,J})^r$. Maps $c_r$ are defined inductively as follows:
		\begin{itemize}
			\item $r=0$: Let $\un{x} \in (\WW_{n,J})^0$ be an element with underlying graph $G$. Then $c_0(\un{x})$ is defined by the same expression as in \eqref{c0eq}, where now $\un{y}_v, v \in V(G)$ are defined as follows: let $\un{x}_v$ denote the label of $v$ in $\un{x}$. Then $\un{y}_v$ is defined to be the element of $(\WW_{n,J})^0$ with underlying \rda-graph $K_v$ as above and with the labels of vertices in $K_v$ given as follows:
			\begin{itemize}
				\item For vertices corresponding to connected components $\Sigma_{\un{x}}\setminus  \bigcup_i \partial A^{(m)}_{n,j}$ other than the disks $A^{(m)}_{n,j}$, the label is given by the underlying Riemann surface, treating the boundaries as inputs or outputs according to the edges of $K_v$ and forgetting any \spot present on this surface (equivalently by the \splitsurfaces\  containing no positive weight \spot corresponding to the connected component).
				\item For vertices corresponding to the annuli $A^{(m)}_{n,j}$, the labels are given by the Riemann surface with weighted \spot given by the annuli $A^{(m)}_{n,j}$ and with both boundaries treated as inputs.
			\end{itemize}
		\item The inductive step for constructing the map $c_r$ assuming the existence of $c_i$ for $i <r$ is performed exactly as in section \ref{cutmaps}, with $(\WW_{\Sigma})^i, H_v$, and $A_\nu$ respectively replaced by $(\WW_{n,J})^i, k_v$, and $A^{(m)}_{n,j}(\un{\Sigma})$.
		\end{itemize}
		\subsubsection{Map $f_{\rho}$ and homotopy $H_\rho$} 
		 Recall the step function $\rho$ from Section \ref{rhofunction}. We define $f_{\rho} \colon \WW_{n,J}\to \WW_{n,J}$ to be the geometric realization of the map of simplicial spaces which changes the label of a simplex by applying $\rho$ to the weights of the \spot in the label. Function $\rho$ is continuous on its domain and thus the induced map $f_{\rho} \colon \WW_{n,J}\to \WW_{n,J}$ is continuous as well.\\ 
		 Let $H_{\rho}$ be the homotopy from the identity to $f_{\rho}$ obtained by linearly interpolating the weights of the \spot in the labels.
		 \subsubsection{Cut homotopy $\Phi_t \colon F^{(m)}_{n,J} \to W_{n,J}$}
		We now construct a homotopy from $f_{\rho}$ to the cut map $c$.\\
		It defined exactly as the cut homotopy in Section \ref{para: cut homotopy}, using the same formulas as there, provided 
		\begin{enumerate}
			\item we replace the subcomplex $\WW_{n,\Sigma}$ there with subcomplex $\WW_{n,J}$ 
			\item the simplicial components $c_i$ of the cut maps being used are as in Section \ref{subsubsec: fulltcft cut maps} here, instead of Section \ref{cutmaps}.
		\end{enumerate}

	\subsubsection{The Homotopy $\Psi$} 
	Finally we have a homotopy $\Psi$ form $c$ to the map $\chi_n \circ \pi$. The homotopy is again defined exactly as in Section \ref{psiproprmain}, with the following straightforward adjustments:
	\begin{enumerate}
		\item Subcomplex $\WW_{\Sigma}^M$ is replaced by the subcomplex $\WW_{n,J}^n$ given by simplices with labels $\un{x}$, such that for a vertex $v$ with label $\un{x}_v$ in $\un{x}$, 
		\[\pi(\un{x}_v) \cap A^{(m)}_{n,j}(\un{\Sigma}) \neq \varnothing \Rightarrow \un{x}_v = \eta^r( A^{(m)}_{n,j}(\un{\Sigma})) \mbox{ for some $r$}\] 
		\item Section $\chi_{n, \Sigma}^M$ is replaced by $\chi_n$ as in \ref{subsec: full tcft chi}.
		\item graph $G_{\pi(\un{x})}$ is replaced by $G_{J}(\un{x})$ as in \ref{subsec: full tcft chi}.
	\end{enumerate}
	 The homotopy $\Psi$ is then defined using extra degeneracies and augmentations as in \ref{extradegpsi}.

	\subsubsection{Contraction homotopy proving condition \eqref{filtrfulltcft}}  This is just the concatenation of the cut homotopy $\Phi$, from identity to the cut map, and homotopy $\Psi$, from $c$ to $\chi_n\circ \pi$.
	
	This completes the proof of condition \ref{filtrfulltcft} and hence of the proof of Theorem \ref{frtofullthm}(1).

\section{From the \CFT-properad to the Deligne-Mumford properad}\label{mbarsec}
We now prove that homotopy trivializing the \reduced subproperad of annuli $\Mfrfullunst$ in $\Mfrfull$ gives the Deligne-Mumford properad $\Mbar$.

Let us recall the statement of the Theorem:
\begin{theorem}\label{mbarmainthm}
$\Mbar$ is the homotopy colimit of the following diagram in the category of \reduced properads
	\begin{equation}\label{eq: mbarsec thm}
		\begin{tikzcd}
			\Mfrfullunst \ar{r} \ar{d} & \Mbarunst \\
			\Mfrfull &
		\end{tikzcd}
	\end{equation}
\end{theorem}


As explained below in Section \ref{subsec: mbar as stk proprd}, properad $\Mbar$ takes values in stacks. In this section we shall show that the homotopy colimit of the diagram \eqref{eq: mbarsec thm} in the category of topological properads has the property that for every $(n_-,n_+)$ the corresponding space of operations has the same homotopy type as the stack $\Mbar(n_-,n_+)= \coprod_{g \geq 0} \Mmod^{fr}_{g,n_-,n_+}$, in the sense of Section \ref{subsubsec: htpy type of top stk}. The statement in the category of \reduced properads then follows using Proposition \ref{prop: stk prpd colim}.

A combination of the strategies used for proving Theorem \ref{hatthm} and Theorem \ref{frtofullthm}(1) will be used for proving Theorem \ref{mbarmainthm}. To highlight the similarity between the argument here and in Section \ref{finalhtpyeq} as well as Section \ref{subsec: fulltcftsec htpyeq}, we will use the same notation as in these sections for the analogous notions here. As outlined in Remark \ref{rmk: seq of nodes} below, to account for the appearance of composition of a sequence of nodal cylinders as well as nodal cylinders composed with disks with positive weight \spots, some changes are needed when constructing the open sets and filtration as in Section \ref{gammasubsubsect}. In the rest of the section we carry out this proof.

\subsection{Homotopy pushout of \eqref{eq: mbarsec thm}}
Similar to Section \ref{subsec: fulltcftsec htpyeq}, it turns out that for computing the pushout \ref{eq: mbarsec thm}, it is convenient to replace $\Mfrfull, \Mfrfullunst, \Mbar$ and $\Mbarunst$ with homotopy equivalent properads $\mdfrfull, \mdfrfullunst, \mdbar$ and $\mdbarunst$. The definitions of these moduli spaces are of similar flavor as those in Section \ref{subsec: splitsurf} using \splitsurfaces, but are slightly different. We now describe these:

Moduli spaces $\mdfrfull_{g,n_-,n_+}$, underlying properad $\mdfrfull$, are defined analogously as the moduli spaces $\mfrfull_{g,n_-,n_+}$ except the following: We do not impose the restriction that an irreducible component without an output necessarily has a \spot of weight $1$, instead we require that if $(g,n_-,n_+)=(0,0,1)$ or $(0,1,0)$, i.e. if with the Riemann surface underlying the \splitsurface\ is given by a disk, then it contains at least one weight $1$ \spot.\\ 
Properad $\mdbar$ is built using moduli spaces $\mdbar_{g,n_-,n_+}$ which are defined similarly as $\mdfrfull_{g,n_-,n_+}$,  with the only difference being that we define it using stable nodal surfaces with weighted \spots instead of smooth surfaces. We continue to impose the restriction that the \splitsurface\ contains at least one weight $1$ \spot when $(g,n_-,n_+)=(0,0,1)$ or $(0,1,0)$. Note that the identification used for defining $\mdbar$ by forgetting weight $0$ \spots may involve collapsing unstable components.\\
Finally, $\mdfrfullunst$ and $\mdbarunst$ are defined as the subproperad consisting of (possibly nodal) \splitsurfaces\  such that the underlying surface, obtained by forgetting all the \spots, is a (possibly nodal) annulus, or a disk.
%


The following proposition is proved using an argument similar to the one used for Proposition \ref{polkaproposition}:
\begin{proposition}
The maps
\[\mdfrfull \to \Mfrfull, \quad \mdfrfullunst \to \Mfrfullunst, \quad \mdbar \to \Mbar \quad \mbox{and} \quad \mdbarunst \to \Mbarunst,\] 
given by mapping a \splitsurface\ to the underlying nodal Riemann surface, are a homotopy equivalence of properads.\\
\null \hfill \qedsymbol
\end{proposition}

Thus to compute the homotopy pushout of diagram \eqref{eq: mbarsec thm}, we use the homotopy equivalent pushout diagram
\[ \mdfrfull \leftarrow \mdfrfullunst \to \mdbarunst \] 
Using Proposition \ref{lem: bar pushout} and \ref{compareproposition}, and an argument similar to the one used in Section \ref{subsec: mthm1htppushsec}, we know that the homotopy pushout of the diagram \eqref{eq: mbarsec thm} coincides with the topological \reduced properad
\[ \left| \GG \left( \GG^{\bullet}\mdfrfull \coprod_{\GG^{\bullet}\mdfrfullunst} \GG^{\bullet} \mdbarunst \right) \right|. \]
Consider the induced map
\begin{equation}\label{eq: mbarsec weq}
	\left| \GG \left( \GG^{\bullet}\mdfrfull \coprod_{\GG^{\bullet}\mdfrfullunst} \GG^{\bullet} \mdbarunst \right) \right|  \to \mdbar.
\end{equation}
We now show that this is a weak equivalence of \reduced properads in stacks, in the sense of Section \ref{subsubsec: htpy type of top stk} below.

\subsection{\eqref{eq: mbarsec weq} is a weak-equivalence}
As in Section \ref{finalhtpyeq},  we will prove that for every stable nodal \splitsurface\  ${\un{\Sigma}}$, there is a chart $U_{{\un{\Sigma}}}$ of $\mdbar$ such that 
for the atlas
\[ U= \coprod_{{\un{\Sigma}}_1,...,{\un{\Sigma}}_k \in \mdbar} U_{{\un{\Sigma}}_1,...,{\un{\Sigma}}_k} \to \mdbar\]
where 
$U_{{\un{\Sigma}}_1,...,{\un{\Sigma}}_k} = U_{{\un{\Sigma}}_1} \times_{\mdbar} ... \times_{\mdbar} U_{{\un{\Sigma}}_k}$, and for the pullback
\[ \pi \colon \left| \GG \left( \GG^{\bullet}\mdfrfull \coprod_{\GG^{\bullet}\mdfrfullunst} \GG^{\bullet} \mdbarunst \right) \right| \times_{\mdbar} U \to U\]
of \eqref{eq: mbarsec weq} to this atlas, the following property holds: for every finite collection ${\un{\Sigma}}_1,...,{\un{\Sigma}}_k$, 
\begin{equation}
\pi^{-1}(U_{{\un{\Sigma}}_1,...,{\un{\Sigma}}_k}) \to U_{{\un{\Sigma}}_1,...,{\un{\Sigma}}_k} \mbox{ is a weak homotopy equivalence}.
\end{equation}
We will in fact show that open sets $\pi^{-1}(U_{{\un{\Sigma}}})$ have a covering $\{W_{n, {\un{\Sigma}}}\}_{n \in \NN}$ such that for every $I \subset \NN$ non-empty, 
\[ \pi|_{W_{I, {\un{\Sigma}}}} \colon W_{I, {\un{\Sigma}}} \to U_{{\un{\Sigma}}} \mbox{ is a weak homotopy equivalence} ,\]
where $W_{I,{\un{\Sigma}}} = \bigcap_{n \in I} W_{n, {\un{\Sigma}}}$.
%
%
%
%
%
%
%
%
As in Section \ref{finalhtpyeq} we shall prove this by exhibiting a filtration 
\[ F_{I, {\un{\Sigma}}}^{(1)} \subset ... \subset F_{I, {\un{\Sigma}}}^{(m)}... \subset W_{I, {\un{\Sigma}}}\]
which satisfies a condition analogous to Condition \eqref{filtrantncondition}, with $W_{n, {\un{\Sigma}}}$ there replaced by $W_{I,{\un{\Sigma}}}$, and section $\chi_{n,{\un{\Sigma}}}^{(m)}$ there replaced by a section $\chi_{n,{\un{\Sigma}}}^{(m)} \colon U_{{\un{\Sigma}}} \to W_{I,{\un{\Sigma}}}$. 
%
%
%
%
%
%
%



\subsubsection{Construction of $U_{{\un{\Sigma}}}, W_{n, {\un{\Sigma}}}$, and $F_{I, {\un{\Sigma}}}^{(m)}$}\label{gammasubsubsect_red}
Suppose that $\un{\Sigma}$ is a nodal \splitsurface. Let $i$ be the number of positive weight \spots on $\un{\Sigma}$ and $0 < \alpha <1 $ be  such that there exists no \spot with weight $\leq \alpha$ in $\un{\Sigma}$. $\Sigma$ be the nodal Riemann surface underlying $\un{\Sigma}$. \\ 
The open set $U_{{\un{\Sigma}}}$ is defined to be the set of nodal \splitsurfaces\  satisfying the following conditions 
\begin{enumerate}
	\item The underlying nodal Riemann surface lies in an open set $U_{\Sigma}$ constructed exactly as in Section \ref{gammasubsubsect}, the only difference being this time we work with ${\un{\Sigma}}$ in $\Mmodbar^{fr}_{g,n_-,n_+}(i)$, the moduli space of stable nodal Riemann surfaces with $i$ marked points, instead of $\Mmodhat^{fr}_{g,n_-,n_+}$. 
	\item Shrinking $U_{\Sigma}$ if necessary, we may assume that there exists a choice of charts $\sigma_{\Sigma^\prime,j}, 1 \leq j \leq i$ centered at the $j$th \spot in each surface $\Sigma^\prime \in U_{\Sigma}$, varying continuously with $\Sigma^\prime$. Then, we assume that there are no \spots with weight $\geq \alpha$ lying outside the images of these charts.
\end{enumerate}

The open sets $W_{n, {\un{\Sigma}}}$ are constructed as follows:\\
For a simplex
\[\un{x} \in \GG^{n+1} \mdfrfull \coprod_{\GG^{n+1}  \mdfrfullunst} \GG^{n+1} \mdbarunst\]
define $\gamma_{\un{x}}$ as in Section \ref{gammasubsubsect}. 
\begin{remark}\label{rmk: gamma seq nodes} In the current situation, the map $\gamma_{\un{x}}$ may not be an embedding and moreover its image may not be disjoint from the nodes or the \spots. More precisely, the image of a component $S^1$ under $\gamma_{\un{x}}$ is either disjoint from the nodes and \spots, or is contained inside one. Components mapping into nodes arise when we consider simplices $\un{x}$ for which the associated labeled graph contains a string of nodal annuli. Such strings do not appear in graphs considered in Section \ref{gammasubsubsect} since any vertex labeled by a nodal annulus there is of type $(0,2)$. Components mapping into \spots arise from simplices for which the associated labeled graph contains a path with one of the vertices labeled by a nodal annulus and which terminates in a vertex labeled by a disk with a positive weight \spot.
\end{remark}

For each $n$ and $\nu \in \Nu$, let $K_{n,\nu}^{(m)}$ be the open subset of the universal curve $\cfam_{{\Sigma}}$ over $U_{\Sigma}$ defined by 
	\begin{equation*} 
		\begin{split}
			K_{n,\nu}^{(m)} = \Bigg\{ (z, \un{t}) \in R_{\hat{\nu}} \times \DD^{\Nu} \Bigg \lvert  |t_\nu| \in  & \left[ 0, \frac{1}{4n^2} \right],\mbox{ and }\\
														|z_{\hat{\nu}}(z) | \in & \left[ \frac{ |t_{\nu}| }{ \frac{1}{n-1}\left( 1+ \frac{1}{m} \right) }, \frac{ |t_{\nu}| }{ \frac{1}{n}\left( 1- \frac{1}{m} \right) } \right] \cup \left[\frac{1}{n}\left( 1- \frac{1}{m} \right), \frac{1}{n-1}\left( 1+ \frac{1}{m} \right)\right] \Bigg\}.
		\end{split}
	\end{equation*}
Set 
\[K_n^{(m)}:= \bigcup_{\nu \in \Nu} K_{n,\nu}^{(m)}.\]

Also for every $n$ and $1 \leq j \leq i$, let $L_{n,j}^{(m)} \subset \cfam_{\Sigma}$ be the open set defined by the union of images
\[ L_{n,j}^{(m)} := \bigcup_{\Sigma^\prime \in U_{\Sigma}} \sigma_{\Sigma^\prime,j}\left[\frac{1}{n}\left( 1 - \frac{1}{m}\right) < |z| < \frac{1}{n-1}\left( 1 + \frac{1}{m} \right) \right] \]

Define $\FF_{n, {\un{\Sigma}}}^{(m)}$ to be the simplicial subspace of $\GG^{n+1} \mdfrfull \coprod_{\GG^{n+1}  \mdfrfullunst} \GG^{n+1} \mdbarunst$ consisting of simplices $\un{x}$ such that the associated embedding $\gamma_{\un{x}}$ is disjoint from $K_n^{(m)} \cup L_n^{(m)}$. Set $\FF_{I, {\un{\Sigma}}}^{(m)}$ to be the intersection
\[ \FF_{I, {\un{\Sigma}}} := \bigcap_{n \in I} \FF_{n, {\un{\Sigma}}}.\]
Set $F_{n, {\un{\Sigma}}}^{(m)}$ and $F_{I, {\un{\Sigma}}}^{(m)}$ to be the open sets of $\left| \GG^{n+1} \mdfrfull \coprod_{\GG^{n+1}  \mdfrfullunst} \GG^{n+1} \mdbarunst \right|$ corresponding to these subcomplexes. We set 
\[ \WW_{n, {\un{\Sigma}}} := \bigcup_{m \in \NN} \FF_{n, {\un{\Sigma}}}^{(m)} \quad \mbox{ and } \quad \WW_{I, {\un{\Sigma}}}  = \bigcap_{n \in I} \WW_{n, {\un{\Sigma}}} \]
Finally, we define
\[W_{n, {\un{\Sigma}}} = \bigcup_{n \in \NN} F_{n, {\un{\Sigma}}}^{(m)}.\]
$W_{n, {\un{\Sigma}}}$ deformation retracts onto the geometric realization of the subcomplex $\WW_{n, {\un{\Sigma}}}$.

Note that the open subsets 
\[ F_{I, {\un{\Sigma}}}^{(m)} := \bigcap_{n \in I}  F_{n, {\un{\Sigma}}}^{(m)} \]
deformation retract onto the subcomplexes $\FF_{I, {\un{\Sigma}}}$ and provide a filtration of the subspace
\[ W_{I, {\un{\Sigma}}} = \bigcap_{n \in I} W_{n, {\un{\Sigma}}}\]
\begin{remark}\label{rmk: seq of nodes}
Note that this is analogous to the construction of open sets $A_n^{\nu}$ and $\WW_{n, {\un{\Sigma}}}$ in \ref{gammasubsubsect}, except that in this case instead of an annulus containing the node we have a union of two annuli, one on each side of the node, and both disjoint from the node. This is necessary since the open sets constructed using annuli $A_n^{\nu}$ as in \ref{gammasubsubsect} do not cover the entire geometric realization in the current context: for all $\un{x}$ which contains a sequence of nodal annuli as mentioned in Remark \ref{rmk: gamma seq nodes}, the image of $\gamma_{\un{x}}$ intersects annuli $A_n^{\nu}$ for any $n$. In order to construct a filtration which covers such points and moreover gives an open subset in the geometric realization, it is natural to work with the sets $K^\nu_n$ considered above.
\end{remark}



%
%

%

\begin{figure} 
	\centering
	\begin{subfigure}[b]{0.3\textwidth}
		\includegraphics[width=\textwidth]{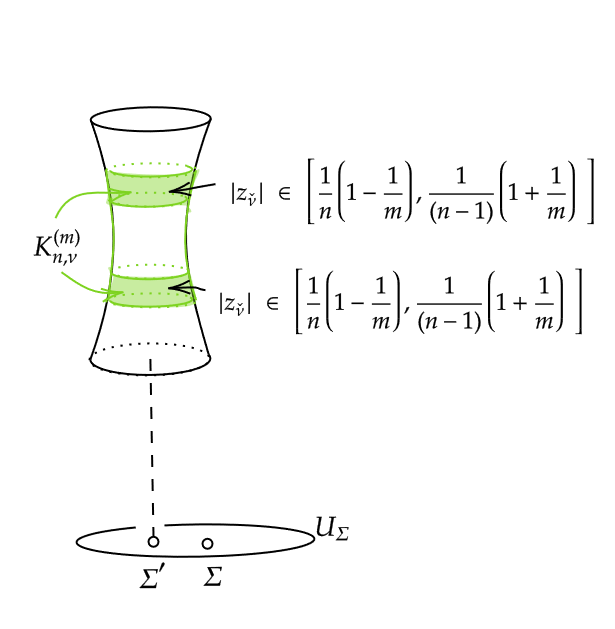}
		\caption{Open sets $W_{n, {\un{\Sigma}}}$}
		\label{filtrationfig3}
	\end{subfigure}
	\begin{subfigure}[b]{0.35\textwidth}
		\includegraphics[width=\textwidth]{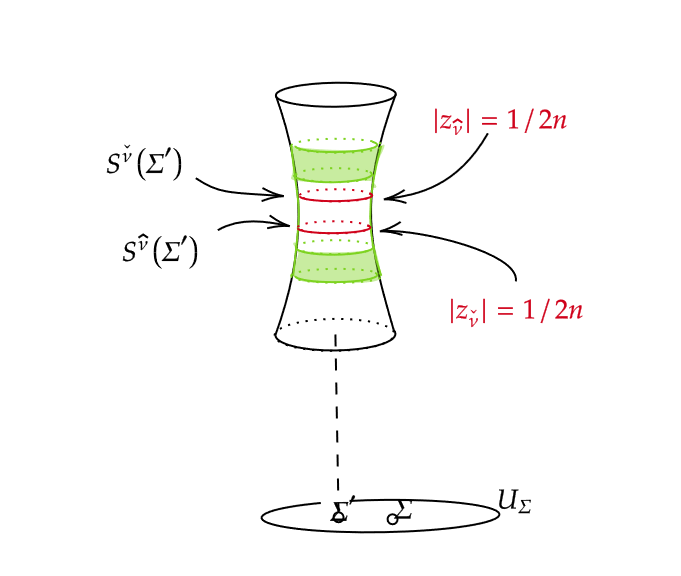}
		\caption{Circles $S_{\hat{\nu}}, S_{\check{\nu}}$}
		\label{filtrationfig4}
	\end{subfigure}
	\caption{}
\end{figure}

\subsection{Proof of  Condition \eqref{filtrantncondition}}\label{subsec: mbar prfcondfiltrtop}
Let us now turn to the fiberwise contraction.\\ 
Fix any $n_0 \in I$. Let ${\un{\Sigma}}^\prime \in U_{{\un{\Sigma}}}$ and let $S^{\hat{\nu}}({\un{\Sigma}}^\prime)$, and $S^{\check{\nu}}({\un{\Sigma}}^\prime)$ be circles defined by
\[ S^{\hat{\nu}}({\un{\Sigma}}^{\prime}) = \left\{(z,t_{\hat{\nu}}({\un{\Sigma}}^{\prime})) \Big \lvert |z_{\hat{\nu}}(z)| = \frac{1}{n_0}\left(1 + \frac{1}{2m} \right) \right\} , S^{\check{\nu}}({\un{\Sigma}}^{\prime}) = \left\{(z,t_{\hat{\nu}}({\un{\Sigma}}^\prime)) \Big \lvert |z_{\hat{\nu}}(z)| =\frac{|t_{\hat{\nu}}|({\un{\Sigma}}^\prime)}{\frac{1}{n_0}\left(1 + \frac{1}{2m} \right)}\right\} \subset {\un{\Sigma}}^{\prime}\]
Denote by $A_{\nu}({\un{\Sigma}}^\prime)$ the annulus bounded by $S^{\hat{\nu}}$ and $S^{\check{\nu}}$ in ${\un{\Sigma}}^\prime$. 
Similarly, let $S^{\hat{j}}(\Sigma^\prime)$ and $S^{\check{j}}(\Sigma^\prime)$ denote the circles 
\[ S^{\hat{j}}(\Sigma^\prime)=\sigma_{\Sigma^\prime,j}\left(\left\{|z| = \frac{1}{n_0}\left(1- \frac{1}{m} \right) \right\}\right) \mbox{ and }  S^{\check{j}}(\Sigma^\prime)=\sigma_{\Sigma^\prime,j}\left(\left\{|z| = \frac{1}{n_0-1}\left(1+ \frac{1}{m} \right) \right\} \right)\]
and $A_j(\Sigma^\prime)$ be the annulus bounded by these.

Having thus defined $A_{\nu}(\Sigma^\prime)$ and $A_j(\Sigma^\prime)$, rest of the argument for verifying Condition \eqref{filtrantncondition}, including the construction of cut maps, cut homotopies,  and the contraction proving Condition \eqref{filtrantncondition}, proceeds by combining the constructions in Sections \ref{prfcondfiltrtop} and \ref{subsec: fulltcftsec htpyeq} in a straightforward manner, with cuts made around both the nodes and weight $1$ \spots. The corresponding maps and homotopies can be defined using exactly the same expressions as in Section \ref{prfcondfiltrtop}, provided we adopt the definition of graphs $H_v$ constructed in \ref{hvgrph} as follows:
\subsubsection{Modification of graphs from Section \ref{hvgrph}}
Let $\un{x}$ be an element in $\mathbb{W}_{I,\un{\Sigma}}$ with underlying graph $G$, and let $\un{x}_v$ be the label of $v \in V(G)$ in $\un{x}$. Let $\un{\Sigma}_{\un{x}}$ and $\un{\Sigma}_{\un{x}_v}$ denote the \splitsurfaces\  obtained by gluing the inner-most labels of $\un{x}$ and $\un{x}_v$ respectively.\\
	Then, in this case we define $H_v$ to be the graph with: 
		\begin{itemize}
			\item vertices given by the connected components of $\un{\Sigma}_{\un{x}_v} \setminus \left( \bigcup_{\nu} \partial A_{\nu} (\un{\Sigma}_{\un{x}}) \cup \bigcup_{j} \partial A_{j} (\un{\Sigma}_{\un{x}})  \right)$, and
		\item edges given by connected components of $\bigcup_{\nu} \partial A_{\nu} (\un{\Sigma}_{\un{x}}) \cup \bigcup_{j} \partial A_{j} (\un{\Sigma}_{\un{x}})$. The edges corresponding to $\partial A_{\nu} (\un{\Sigma}_{\un{x}})$ are directed away from the annulus $A^{\nu}_{n,j} (\un{\Sigma}_{\un{x}})$ and those corresponding to $\partial A_{j} (\un{\Sigma}_{\un{x}})$ are directed away from the disk $A_j(\un{\Sigma}_{\un{x}})$. 
		\end{itemize}

\begin{remark}\label{rmk: mbar no output}
	Unlike the graphs $H_v$ constructed in \ref{hvgrph}, analogously defined graphs here will not satisfy the property that every vertex has an output (see Remark \ref{rmk: at least one output}). This was needed to ensure that the cut maps and homotopies constructed there land in appropriate targets. This itself was necessary since properad $\Mfr$ has no operations with $n_+=0$. However in the current case, properad $\mdfrfull$ has operations in all components and thus the constraint on graphs $H_v$ mentioned above is not relevant for the construction of cut maps and homotopies in this case.
\end{remark}
Having thus adopted the definition of graphs $H_v$, the rest of the argument for proving condition \ref{filtrantncondition} carries over from Section \ref{prfcondfiltrtop}, with the same formulae defining the maps and homotopies in this case. This concludes the proof of condition Condition \eqref{filtrantncondition}, and hence of Theorem \ref{mbarmainthm}

\section{Homotopy colimits of properads in stacks}\label{sec: infty properad}
In this section we outline a formalism for $\infty$-\reduced properads and explain how to view $\Mbar$ as an example of such a properad. 
Moreover, we explain how to interpret the earlier statements about homotopy pushouts of \reduced properads in the context of $\infty$-\reduced properads.

In this section we use the language of $\infty$-operads and $\infty$-categories. 
The main references for this part are \cite[Section 2.1]{lurie2016higher}, and \cite[Section 5.2]{lurie2009higher}. 
We shall also use the theory of homotopy type of stacks and singular simplicial complexes of stacks developed in \cite{noohi2012homotopy} and  \cite{coyne2016singular}.

\subsection{$\infty$-properads}
As in the case of operads, there are now a number of models available for up-to-homotopy properads. 
In \cite{hackney2015infinity}, a model for $\infty$-properads in terms of graphical sets (presheaves on a certain category of graphs) was constructed, generalizing the dendroidal sets approach to $\infty$-operads. 
A model category structure was constructed on the category of simplicial properads in \cite{hackney2018simplicial}, generalizing an analogous structure on the category of simplicial operads. 
Another approach is to use the fact that ordinary properads can be expressed as algebras over a colored operad. 
A model for $\infty$-properads can then be obtained by considering algebras over the $\infty$-operad associated with this operad.

We will follow this last approach to set up a category of $\infty$-\reduced properads. 
The reason behind this choice is that in the restricted setting of \reduced properads, the results of \cite{pavlov2018admissibility} can be used to prove Proposition \ref{thm: equiv of inf prpds}. 
It gives a convenient way for comparing the homotopy categories, and in particular the homotopy colimits, of ordinary \reduced properads with those of $\infty$-\reduced properads. 
Proposition \ref{thm: equiv of inf prpds} uses the fact that we are working with \reduced properads: the corresponding theorem for non-\reduced properads is known to be false (see Remark \ref{rmk: proprd rectification} below).

\subsubsection{$\infty$-properads}\label{subsubsec: infty proprd oprd}

As described in \cite{lurie2016higher}, we can associate to any colored operad $\oprdO$ a category $\oprdOten$ with a map $\oprdOten \to \Fin_*$ such that the associated map of nerves $N\oprdOten \to N\Fin_*$ defines an $\infty$-operad. We denote this $\infty$-operad simply as $\oprdO^\otimes$ and refer to it as the $\infty$-operad associated to $\oprdO$. For later reference, we note here that 
\begin{equation}\label{eq: tensor category 1}
	\begin{minipage}{0.92 \textwidth}
		\begin{itemize}
			\item objects of $\oprdOten$ are sequences $c_1,...,c_k$ of colors of $\oprdO$, and 
			\item morphisms in $\oprdOten$ from $\{c_i\}_{i \in [k]}$ to $\{d_j\}_{j \in [l]}$ are given by a map $\alpha \colon [k] \to [l]$ along with an element 
			\[ \phi_j \in \oprdO(\{c_i\}_{i \in \alpha^{-1}(j)}; d_j) \mbox{ for every } j\]
		\end{itemize}
	\end{minipage}
\end{equation}

Given a symmetric monoidal category $(\catc, \otimes)$, we have an associated colored operad $\oprdO_{\catc}$. We will denote the corresponding category ${\oprdOten_{\catc}}$ simply as $\catc^\otimes$. Here,
\begin{equation}\label{eq: tensor category 2}
	\begin{minipage}{0.92 \textwidth}
		\begin{itemize} 
			\item objects of $\catc^\otimes$ are given by sequences of $c_1,....,c_n$ of objects in $\catc$ and 
			\item morphisms in $\catc^\otimes$ between $\{c_i\}_{i \in [k]}$ and $\{d_j\}_{j \in [l]})$ are given by a maps $\alpha \colon \phi \colon [k] \to [l]$ along with an element 
			\[ \phi_j \in  \catc\left(\bigotimes_{i \in \alpha^{-1}(j)}c_i; d_j \right) \mbox{ for every } j\]
		\end{itemize}
	\end{minipage}
\end{equation}
Notice that the data of an $\oprdO$-algebra in $\catc$ is equivalent to an $\infty$-operad map $\oprdO^\otimes \to \catc^\otimes$.

More generally when $\catc$ is a symmetric monoidal $\infty$-category, there is an associated $\infty$-operad $\catc^{\otimes} \to N \Fin_*$.
Now, recall the set-valued, colored operad $\oprd$ defined in Section \ref{subsubsec: proprd oprd} above. Note that an algebra over $\oprd$ in $\Top$ is precisely a topological properad. Let $\oprdinfty$ denote the $\infty$-operad associated to $\oprd$. 
\begin{definition}
An $\infty$-\reduced properad in a symmetric monoidal $\infty$-category $\catc$ is defined to be a morphism of $\infty$-operads $\oprdinfty \to \catc^{\otimes}$, in other words a $\oprdinfty$-algebra in $\catc$.
\end{definition}
These algebras can be assembled into an $\infty$-category which we denote by $Alg_{\oprdinfty}(\catc)$.

\subsubsection{The comparison functor}
%

Recall that to any model category $A$ we can associate an $\infty$-category $N(A^c)[W^{-1}]$, where
\begin{enumerate}
\item $A^c$ denotes the full subcategory of cofibrant objects in $A$,
\item $N(A^c)$ denotes the $\infty$-category obtained by taking the nerve of the ordinary category $A^c$,
\item $W$ is the collection of weak equivalences in $A^c$, and 
\item $N(A^c)[W^{-1}]$ denotes the ($\infty$-categorical) localization of the $\infty$-category $N(A^c)$ at $W$.
\end{enumerate}
(see \cite[Definition 1.3.4.15 and Remark 1.3.4.16]{lurie2016higher}). This is an $\infty$-categorical enhancement of the homotopy category of $A$. In particular, homotopy colimits in $A$ correspond to the $\infty$-categorical colimits in $N(A^c)[W^{-1}]$.

We denote the $\infty$-categories associated with the model categories $\Top$ and $Alg_{\oprd}(\Top)$ by $\Top_\infty$ and $Alg_{\oprd}(\Top)_\infty$.\\
We have a comparison map
\begin{equation}\label{eq: comparison functor} Alg_{\oprd}(\Top)_\infty \to Alg_{\oprdinfty}(\Top_\infty) \end{equation}
between $\infty$-categories.

\begin{proposition}\label{thm: equiv of inf prpds}
The comparison functor \eqref{eq: comparison functor}
is an equivalence of $\infty$-categories. 
\end{proposition}

\begin{proof}
This is an analogue for $\infty$-\reduced properads of \cite[ Theorem 7.2.4]{chu2020rectification}, which itself is a corollary of \cite[Theorem 7.11]{pavlov2018admissibility}. As in \cite{chu2020rectification}, the main step is the verification of the symmetric flatness hypothesis in \cite[Theorem 7.11]{pavlov2018admissibility}, which in our case follows from Lemma \ref{lem: sigma cofibrancy}.
\end{proof}
\begin{remark}\label{rmk: proprd rectification}
In Proposition \ref{thm: equiv of inf prpds} we use the fact that we are working with input-output properads. The corresponding statement is not true for ordinary properads (see \cite[Theorem 7.2.5]{chu2020rectification}). However, even without the full strength of the proposition, it may still be possible to show that the homotopy colimits appearing in our theorems continue to be preserved in the context of ordinary $\infty$-properads. We do not pursue this here.
\end{remark}

\subsection{Topological stacks and their homotopy theory}
In this section we recall the definition of topological stacks and construct an $\infty$-category modeling their homotopy theory.
\subsubsection{Topological stacks}
Recall that a \emph{groupoid} is a category in which every morphism is an isomorphism. Groupoids form a (strict) $2$-category. An equivalence of groupoids is an equivalence of the underlying categories.\\
Denote by $\pshv(\Grpd)$ the (strict) $2$-category of presheaves of groupoids over topological spaces. By the Yoneda Lemma,
we have a fully faithful embedding
\[ \Top \to \pshv(\Grpd).\]
More precisely, for any topological space $T$ and a presheaf of groupoid $\calX$, we have an isomorphism of groupoids
\[ \calX(T) \simeq \pshv(\Grpd)(\calX,T).\]
Say that a map of presheaves $\calX \to \calY$ is an equivalence if for every topological space $T$, the corresponding maps $\calX(T) \to \calY(T)$ are equivalences of groupoids. Two presheaves are called equivalent if they are related by a zig-zag of equivalences.

\begin{definition}\label{def: stk}
A \emph{stack} over topological spaces is a presheaf of groupoids $\calX$ such that for any topological space $T$ and an open cover $\{U_{\alpha}\}$ of $T$,
\[ \calX(T) \to \varinjlim \left[ \prod_{\alpha} \calX(U_\alpha) \rrarrow \prod_{\alpha, \beta}\calX(U_{\alpha \beta}) \rrrarrow \prod_{\alpha, \beta, \gamma}\calX(U_{\alpha \beta \gamma})  \right]\]
is an equivalence of groupoids. Here $\varinjlim$ is the $2$-categorical limit in the category of groupoids.
\end{definition}
 Let $\stk$ denote the full subcategory spanned by stacks in $\pshv(\Grpd)$. Equivalences of stacks are precisely those of the underlying presheaves of groupoids.
 
\begin{remark}\label{rmk: cat fibrd in grpd}
Conventionally a stack is defined as a category fibered in groupoids which satisfies a descent condition. Let $\mathcal{G}rpd/\Top$ denote the $2$-category of categories fibered in groupoids over $\Top$. It was proved in \cite{hollander2008homotopy}, that there are model category structures on $\mathcal{G}rpd/\Top$ and $\pshv(\Grpd) $ such that the weak equivalences in both categories are given by the corresponding notions of equivalences. Moreover, there is an adjunction
\[ \Gamma \colon \mathcal{G}rpd/\Top \leftrightarrows \pshv(\Grpd) \colon p\]
which is a Quillen equivalence with respect to these model structures. Furthermore, this adjunction restricts to one between the category of categories fibered in groupoids over topological spaces which satisfy a descent condition and the category of stacks in the sense of Definition \ref{def: stk}. Thus both these approaches for defining stacks turn out to be equivalent.\\
For later reference, we note that the functor $\Gamma$ is given by 
\[ \Gamma \colon \mathcal{E} \mapsto (T \mapsto {\mathcal{G}rpd/\Top}_{Grpd}(\Top/T,\mathcal{E} ) ).\]
Here ,
	\begin{enumerate}
		\item $\Top/T$ given by the category of topological spaces with a map to $T$. Note that this is a category fibered in groupoids over $\Top$. 
		\item${\mathcal{G}rpd/\Top}_{Grpd}(\_,\_)$ denotes the groupoid of  $1$-morphisms in $\mathcal{G}rpd/\Top$.
	\end{enumerate}
\end{remark}

Recall that a topological groupoid $\mathbb{X}$ is a groupoid such that the set of objects and the set of morphisms are topological spaces, and moreover all the structure maps (source, target, identities, compositions, and inverses) are continuous. Let $\Top\Grpd$ be the category of topological groupoids.

Given a topological groupoid $\mathbb{X}$, we can define a presheaf given by
\[ T \mapsto [\mathbb{X}] (T) := {\Top \Grpd}_{Grpd}(T,\mathbb{X}) \]
where ${\Top \Grpd}_{Grpd}$ denotes the groupoid of $\Top\Grpd$-morphisms from $T$ to $\mathbb{X}$. This presheaf can be \emph{sheafified} to obtain a stack $[\mathbb{X}]$, known as the quotient stack of $\mathbb{X}$, by 
\[ [\mathbb{X}] \colon T \to \varinjlim_{U \underset{open}{\subset} T} \Top\Grpd_{Grpd}(U,\mathbb{X})\]
\begin{definition}
By a \emph{topological stack} we mean a stack which is equivalent to the quotient stack of some topological groupoid.
\end{definition}

\subsubsection{Homotopy type of a topological stack}\label{subsubsec: htpy type of top stk}
Following \cite[Section 5]{noohi2012homotopy} we say that a map of topological spaces $f \colon X \to Y$ is \emph{shrinkable} if it has a section $s \colon Y \to X$ such that $X$ admits a fiberwise relative deformation retract onto $s(Y)$. We say that a map $f \colon X \to Y$ is \emph{locally shrinkable} if $Y$ has an open cover $\{U_\alpha\}$ such that $f|_{f^{-1}(U_\alpha)}$ is shrinkable for each $\alpha$.\\
Say that a representable map of topological stacks $f \colon \calX \to \calY$ is locally shrinkable if for any topological space $T$ and a map $g \colon T \to \calY$, the pullback $g\times_{\calY} f \colon T \times_{\calY} \calX \to T$ is locally shrinkable.\\
We say that a representable map $\mathcal{X} \to \mathcal{Y}$ of topological stacks is a \emph{universal weak equivalence} if for any topological space $T$ with a map $T \to \calY$, the induced map $T \times_{\calY} \calX \to T$ is a weak equivalence. In particular, any locally shrinkable map is a universal weak equivalence.\\
Finally, we say that a stack $\calX$ admits an \emph{atlas} if there exists a representable map $U \to \calX$ from a topological space $U$, such that for any topological space $T$ and a map $T \to \calX$ the pullback $T \times_{\calX} U \to T$ admits local sections. It follows from the definition of topological stacks that every topological stack admits an atlas.

Note that if $V \to \calY$ is an atlas for $Y$, a representable map of stacks $\calX \to \calY$ is locally shrinkable (respectively, a universal weak equivalence) if and only if the induced map of topological spaces $T \times_{\calY} \calX \to \calX$ is locally shrinkable (respectively, a weak equivalence).

\begin{definition}
	We say that a topological space $X$ \emph{represents the weak homotopy type} of a stack $\calX$ if there exists a universal weak homotopy equivalence $X \to \calX$.
\end{definition}
Note that such a representative, if it exists, is unique up to weak homotopy equivalence. We now state a theorem, due to Noohi, guaranteeing the existence of such a space: 
\begin{theorem}[{\cite[Theorem 3.4]{coyne2016singular}}]\label{thm: class sp of stk}
Let $\calX$ be a topological stack. Then there exists a topological space $X$ along with a map $X \to \calX$ which is locally shrinkable and in particular a universal weak equivalence. 
\end{theorem}

From Theorem \ref{thm: class sp of stk} it follows that the weak homotopy type of every topological stack can be represented by a topological space. 

Given a map of stacks $f \colon \calX \to \calY$ and a topological space $Y$ representing the weak homotopy type of $\calY$ via a map $Y \to \calY$, the pullback $X:=Y \times_{\calY} \calX \to \calX$ represents the weak homotopy type of $X$ and moreover, we have an induced map $\widetilde{f} \colon X \to Y$. 
In Section \ref{subsubsec: infty cat of stacks} below, we shall construct a model for the homotopy theory of topological stacks obtained by inverting maps $f$ such that the induced maps $\widetilde{f}$ on representatives of the homotopy types are weak homotopy equivalences.

\subsubsection{Singular simplicial complexes of presheaves of groupoid}
Let $\sSet$ denote the category of simplicial sets.
We now define a functor 
\[\Sing \colon \pshv(\Grpd) \to \sSet.\]
For an $\calX \in \pshv(\Grpd)$, define $\calX_{\Delta}$ to be the simplicial groupoid 
\[ \calX_{\Delta} \colon n \mapsto \calX(\Delta^n).\]
Let $N(\calX_{\Delta})$ denote the bisimplicial set obtained by taking level-wise nerve of $\calX_{\Delta}$. Finally, define $\Sing(\calX)$ by 
\[ \Sing(\calX) := \mbox{Diag} (N(\calX_{\Delta})),\]
where $\mbox{Diag} \colon \bsSet \to \sSet$ is the diagonal functor defined by $\{\calY_{m,n}\}_{m\geq 0, n \geq 0} \mapsto \{\calY_{n,n}\}_{n\geq 0}$.

We note here a useful property of $\Sing$:
\begin{lemma}\label{prop: sing prod}
$\Sing$ preserves limits. In particular, $\Sing$ preserves products.
\null \hfill \qedsymbol
\end{lemma}
%
%
\subsubsection{$\infty$-category of stacks}\label{subsubsec: infty cat of stacks}
Given two stacks $\calX$ and $\calY$, we can consider the mapping stack $\Map(\calX,\calY)$. The underlying presheaf of groupoids is given by 
\[ T \mapsto {\stk}_{Grpd}(\calX \times T, \calY),\]
where ${\stk}_{Grpd}$ denotes the groupoid of stack morphisms. Using the functor $\Sing$ above we can associate with this stack a simplicial set $\Sing(\Map(\calX,\calY))$. Using Lemma \ref{prop: sing prod}, it follows that this provides a simplicial enrichment of the category of stacks. Denote this simplicial category by $\stk^{\Delta}$.

We can associate with this a fibrant simplicial category by taking the Kan replacements of the simplicial morphism spaces: Recall that the category $\sSet$ has a symmetric monoidal fibrant replacement functor given by 
\[ X \mapsto {X}^{\wedge}= \Map(\Delta_{\bullet},|X|)\]
Using this functor, we replace the $\sSet$-morphism spaces $\Sing(\Map(\calX,\calY))$ by their fibrant replacement $\Sing(\Map(\calX,\calY))^{\wedge}$. This gives us a fibrant simplicial enrichment of the category of stacks. Taking the homotopy coherent nerve of this category gives an $\infty$-category $\widetilde{\stk}_\infty$. As a model for the homotopy theory of topological stacks, we shall use a category $\stk_\infty$ which is obtained by further localization of $\widetilde{\stk}_\infty$ at the class of morphisms mentioned at the end of Section \ref{subsubsec: htpy type of top stk} above.

Let $\Top^{\Delta}$ denote the simplicially enriched category of spaces and $\widetilde{\Top}_\infty$ the corresponding $\infty$-category of spaces. Note that 
\[ \Top_\infty = \widetilde{\Top}_\infty[S^{-1}],\]
where $S$ denotes the class of weak homotopy  equivalences in $\widetilde{\Top}_\infty$ (see for example, \cite[Theorem 1.3.4.20]{lurie2016higher}).
For topological spaces $X,Y$, the mapping stack $\Map(X,Y)$ is the stack associated with the space of continuous maps $\Top(X,Y)$ with the usual compact-open topology. 
Thus, the functor 
\[\widetilde{\Top}_\infty \to \widetilde{\stk}_\infty\]
obtained by mapping a topological space to the stack associated to it, is fully faithful.

We now state Lemma \ref{lem: noohi 7.2} and Proposition \ref{lem: noohi 7.1} which are general category theoretic results.  These are essentially $\infty$-categorical refinements of some results from \cite[Section 7]{noohi2012homotopy}. We shall use $\infty$-categorical generalizations of the $1$-categorical arguments in loc. cit. to prove their $\infty$-categorical enhancements.

Recall that for an $\infty$-category $\catc$ and a collection of morphisms $\mathcal{W}$ in $\catc$, $\catc[W^{-1}]$ denotes the $\infty$-categorical localization of $\catc$ at $W$.

\begin{lemma}[{\cite[Lemma 7.5]{noohi2012homotopy}}]\label{lem: noohi 7.2}
Let $F \colon \mathcal{D} \leftrightarrows \mathcal{E} \colon G$ be an adjunction between $\infty$-categories. Let $S \subset \mathcal{D}$ and $T \subset \mathcal{E}$ be collection of morphisms such that $F(S) \subset T$ and $G(T) \subset S$. Then, there is an induced adjunction 
\[ \widetilde{F} \colon  \mathcal{D}[S^{-1}] \leftrightarrows \mathcal{E}[T^{-1}] \colon \widetilde{G}\]
Moreover, if $F$ is fully-faithful then so is $\widetilde{F}$.
\end{lemma}
\begin{proof}
The existence of the adjoint follows using \cite[Proposition 7.1.14]{cisinski2019higher}. If $F$ is fully-faithful it follows that the unit of the adjunction  $\mathbbm{1}_{\mathcal{E}} \Rightarrow G \circ F$ is an isomorphism of $\infty$-functors. In this case, the unit of the adjunction induced on localized categories is also an isomorphism and it follows that the functor $\widetilde{F}$ is fully-faithful as well.
\end{proof}

Let us now turn to Proposition \ref{lem: noohi 7.1}. The set up for the proposition is as follows: Let
$\catb \to \catc$ be a fully-faithful functor between $\infty$-categories and let $R$ be a class of morphisms in $\catb$ which contains all identities, is closed under compositions, and is closed under pullbacks (see \cite[Definition 7.2.14]{cisinski2019higher} for the precise meanings of these terms). Denote by $\widetilde{R}$ the class of morphisms $f \colon y \to x$ in $\catc$ satisfying the property that for any $p \colon t \to x$ with $t \in \catb$, the pullback $t \times_x y \to t$ of $f$ along $p$ is in $R$. 
\begin{proposition}[{\cite[Lemma 7.1, 7.2]{noohi2012homotopy}}]\label{lem: noohi 7.1}
In the setting described in the previous paragraph, suppose that for every $x \in \catc$, there exists an object $\theta(x)$ in $\catb$ along with a map $\theta(x) \to x$ lying in $\widetilde{R}$. Then, the functor $\catb \to \catc$ induces a functor $i \colon \catc[R^{-1}] \to \catb[R^{-1}]$ which is fully faithful and has a right adjoint $\catb[R^{-1}] \to \catc[R^{-1}]$ which naturally extends $\theta$. Further, the functors $i$ and $\theta$ induce an equivalence between $\catc[\widetilde{R}^{-1}]$ and $\catb[R^{-1}]$.
\end{proposition}
\begin{proof}
Let $\mathcal{H}$ denote the homotopy category of spaces and let $h\catb[R^{-1}]$ and $h\catc[R^{-1}]$ denote the $\mathcal{H}$-enriched homotopy categories of $\catb[R^{-1}]$ and $\catc[R^{-1}]$ respectively, as described in \cite[Section 1.1.5]{lurie2009higher}. Let 
$hi \colon h\catb[R^{-1}] \to h\catc[R^{-1}]$
 be the functor induced by inclusion $i \colon \catb \to \catc$. We shall prove that $hi$ is fully faithful and has a right adjoint 
\[h\theta \colon h\catc[R^{-1}] \to h\catb[R^{-1}]\]
given by $\theta(\_)$ on objects.  It will follows that $i$ is fully-faithful and using \cite[Proposition 5.2.2.12]{lurie2009higher} we get that $i$ has a right adjoint 
 \[\theta \colon  \catc[R^{-1}] \to \catb[R^{-1}],\] 
Using the proof of that proposition it follows that we may assume that $\theta$ is given by $\theta(\_)$ on objects. 
Finally, the fact that $i$ and $\theta$ induce an equivalence of categories $\catc[\widetilde{R}^{-1}]$ and $\catb[R^{-1}]$ follows using Lemma \ref{lem: noohi 7.2}.\\
Let us now turn to the construction of the functor $h\theta$. In order to carry out the construction we use the description of morphism spaces in localizations given by the theory of \emph{Calculus of Fractions} discussed in \cite[Section 7.2]{cisinski2019higher}.\\
Using Theorem 7.2.16 and Remark 7.2.10 from \cite{cisinski2019higher} it follows that the class $R$ of morphisms has a right calculus of fractions given by the maximal putative right calculus of fraction at each object and the homotopy types of the morphism spaces in the localization are given by the following homotopy equivalences of simplicial sets
\begin{align*}
\catc[R^{-1}](x,y) &\simeq \mbox{Span}^{R}_{\catc}(x,y)  \mbox{ and }\\
\catb[R^{-1}](x,y) &\simeq \mbox{Span}^{R}_{\catb}(x,y),
\end{align*}
Here $\mbox{Span}^{R}_{\catc}(x,y)$ (respectively, $\mbox{Span}^{R}_{\catb}(x,y)$) denotes the $\infty$-category of diagrams of the form $x \xleftarrow{r} \widetilde{x} \xrightarrow{f} y$ in $\catc$ (respectively, $\catb$) with $r$ in $R$ (see \cite[Remark 7.2.10]{cisinski2019higher} for more on the Span-categories).\\
Applying \cite[Lemma 7.2.15]{cisinski2019higher} to the map $\theta(y) \to y$ and using the fact that the pull-back of $\theta(y) \to y$ along any $x \to y$ with $x \in \catb$ lies in $R$, it follows that the map
\[ i_{x,y} \colon \mbox{Span}^{R}_{\catb}(x,\theta(y)) \to \mbox{Span}^{R}_{\catc}(x,y) \]
given by inclusion $\catb \to \catc$ followed by post-composition with the map $\theta(y) \to y$, has a right adjoint $g_{x,y}$. Moreover, on the level of objects $g_{x,y}$ is given by mapping a span $x \xleftarrow{r} z \xrightarrow{f} y$ in $\catc$ to the span $x \xleftarrow{r^\prime} z^\prime \xrightarrow{f^\prime} \theta(y)$ in $\catb$ where $f^\prime$ is obtained by considering the cartesian square
\[ 
	\begin{tikzcd}
		z^\prime \ar{r}{f^\prime}\ar{d}{s} &  \theta(y) \ar{d}\\
		z \ar{r}{f} & y
	\end{tikzcd}
\]
and $r^\prime= r \circ s$.
It follows that for any $x \in \catb$ and $y \in \catc$, the map 
\[ i_{x,y} \colon \catb[R^{-1}](x,\theta(y)) \to \catc[R^{-1}](x,y)\]
induces a homotopy equivalence of the simplicial sets underlying the $\infty$-categories, with the homotopy inverse given by the $g_{x,y}$. \\
We now describe the $\mathcal{H}$-enriched functor $h\theta \colon h\catc[R^{-1}] \to h\catb[R^{-1}]$ : On objects it is given by $\theta(\_)$, as mentioned above. On morphisms we define $h\theta_{x,y}$ by the composition
\[ h\catc[R^{-1}](x,y) \to h\catc[R^{-1}](\theta(x),y) \xrightarrow{g_{\theta(x),y}} h\catb[R^{-1}](\theta(x), \theta(y))\]
where the first arrow is induced by pre-composing with $\theta(x) \to x$. To verify compatibility with composition note that we have the following commutative diagram in the homotopy category of spaces $\mathcal{H}$:
\[
	\begin{tikzcd}
		h\catc[R^{-1}](x,y) \times h\catc[R^{-1}](y,z) \ar{r}\ar{d} 	&h\catc[R^{-1}](x,z) \ar{d} \\
	h\catc[R^{-1}](\theta(x),y) \times h\catc[R^{-1}](y,z) \ar{r} & h\catc[R^{-1}](\theta(x),z) \ar[equal]{dd} \\
		h\catb[R^{-1}](\theta(x),\theta(y))  \times h\catc[R^{-1}](y,z)  \ar{d}\ar{u}{i_{\theta(x),y}\times \mathbbm{1}}& \\
		 h\catb[R^{-1}](\theta(x), \theta(y)) \times h\catc[R^{-1}](\theta(y),z)  \ar{r} &   h\catc[R^{-1}](\theta(x),z) \\
		h\catb[R^{-1}](\theta(x), \theta(y)) \times h\catb[R^{-1}](\theta(y), \theta(z)) \ar{r}\ar{u}{\mathbbm{1} \times i_{\theta(y),z}} & h\catb[R^{-1}](\theta(x), \theta(z)) \ar{u}{i_{\theta(x),z}}
	\end{tikzcd}
\]
Here all the unlabeled vertical arrows are induced by suitable pre/post-compositions and the horizontal arrows are given by compositions in suitable categories. Notice that the maps $i_{\theta(x),y}\times \mathbbm{1}, \mathbbm{1} \times i_{\theta(y),z},$ and  $i_{\theta(x),z}$ are homotopy equivalences with inverses $g_{\theta(x),y}\times \mathbbm{1}, \mathbbm{1} \times g_{\theta(y),z},$ and  $g_{\theta(x),z}$ respectively. Thus inverting these maps and composing the vertical arrows we get the following commutative diagram :
\[
	\begin{tikzcd}
	h\catc[R^{-1}](x,y) \times h\catc[R^{-1}](y,z) \ar{r}\ar{d}{h\theta_{x,y} \times h\theta_{y,z} } 	&h\catc[R^{-1}](x,z) \ar{d}{h\theta_{x,z} } \\
	h\catb[R^{-1}](\theta(x), \theta(y)) \times h\catb[R^{-1}](\theta(y), \theta(z)) \ar{r} & h\catb[R^{-1}](\theta(x), \theta(z))
	\end{tikzcd}
\]
This completes the proof of compatibility of $h\theta$ with composition and thus of the fact that $h\theta$ defines a functor. It is clear from the construction of $h\theta$ that it is right adjoint to $hi$. That $hi$, and thus $i$, is fully-faithful follows from the adjunction using the fact that $\theta(y) \to y$ is an isomorphism in $\catb[R^{-1}]$ for any $y \in B$.
\end{proof}

Now, let $R$ denote the class of locally shrinkable morphisms in $\widetilde{\Top}_\infty$. We have the following 
\begin{proposition}[{\cite[Proposition 8.1]{noohi2012homotopy}}]\label{prop: noohi htpy 1}
	$i \colon \widetilde{\Top}_\infty \to \widetilde{\stk}_\infty$ induces a fully faithful functor $\widetilde{\Top}_\infty[R^{-1}] \to \widetilde{\stk}_\infty[R^{-1}]$ which has a right adjoint $\theta \colon \widetilde{\Top}_\infty \to \widetilde{\stk}_\infty$.
\end{proposition}
\begin{proof}
	This is an $\infty$-categorical restatement of \cite[Proposition 8.1]{noohi2012homotopy}. The proof follows by applying Lemma \ref{lem: noohi 7.1} to the inclusion $\widetilde{\Top}_\infty \to \widetilde{\stk}_\infty$.
\end{proof}

Finally, recall that $S$ denotes the class of weak homotopy equivalences in $\widetilde{\Top}_\infty$ and $\widetilde{\Top}_\infty [S^{-1}] \simeq \Top_\infty$. Denote by $S_{\stk}$ the class of maps in $\widetilde{\stk}_\infty$ whose image under $\theta$ lies in $\stk$ and by $\stk_\infty$ the localization $\stk[S_{\stk}^{-1}]$. Then, using Lemma \ref{lem: noohi 7.2} and Proposition \ref{prop: noohi htpy 1} we have:
\begin{corollary}[{\cite[Corollary 8.3]{noohi2012homotopy}}]\label{corr: noohi htpy 2}
	$i$ and $\theta$ as in Proposition \ref{prop: noohi htpy 1} induce a pair of adjoint functors
\[ i \colon \Top_\infty \leftrightarrows \stk_\infty \colon \theta\]
which are in fact equivalences of categories. In particular, $i \colon \Top_\infty \to \stk_\infty$ is fully faithful and preserves colimits.\\
\null \hfill \qedsymbol 
\end{corollary}

\subsection{$\infty$-\reduced Properads in $\stk_\infty$}
Both $\Top_\infty$ and $\stk_\infty$ admit finite products and the inclusion $\Top_\infty \hookrightarrow \stk_\infty$ preserves finite products. Let $\Top_\infty^\otimes$ and $\stk_\infty^\otimes$ be the symmetric monoidal categories associated with the $\infty$-categories $\Top_\infty$ and $\stk_\infty$ with the monoidal structure give by finite products. Using Corollary \ref{corr: noohi htpy 2} it follows that we have an equivalence of symmetric monoidal $\infty$-categories
\[ \Top_\infty^\otimes \hookrightarrow \stk_\infty^\otimes\]
Thus we have,
\begin{proposition}\label{prop: stk prpd colim}
The inclusion $\Top_\infty \hookrightarrow \stk_\infty$ induces a functor of operad algebras
\[ Alg_{\oprdinfty} (\Top_\infty^\otimes) \hookrightarrow Alg_{\oprdinfty} (\stk_\infty^\otimes)\]
which is an equivalence of categories and in particular preserves colimits.
\null \hfill \qedsymbol 
\end{proposition}

\subsection{$\Mbar$ as an $\infty$-\reduced properad in $\stk_\infty$}\label{subsec: mbar as stk proprd}
Recall that $\Mmodbar^{fr}_{g,n_-,n_+}$ is the moduli stack of genus $g$ stable nodal Riemann surfaces with $n_-$ input and $n_+$ output parametrized boundaries. As a category fibered in groupoids this stack is the category with the
\begin{itemize}
\item objects given by the families $\mathcal{C} \to T$ of nodal stable Riemann surfaces of genus $g$ with $n_-$ input and $n_+$ output boundary components, over topological spaces, and 
\item morphisms given by pullback diagrams of such families.
\end{itemize}
The functor $\Mmodbar^{fr}_{g,n_-,n_+} \to \Top$ is given by mapping a family $(\mathcal{C} \to T)$ to its base $T$.

The presheaf of groupoids underlying $\Mmodbar^{fr}_{g,n_-,n_+}$, obtained by applying functor $\Gamma$ described in Remark \ref{rmk: cat fibrd in grpd}, is thus given by
\[ T \mapsto {\mathcal{G}rpd/\Top}_{Grpd}(\Top/T,\Mmodbar^{fr}_{g,n_-,n_+})\]

\begin{notation}
Let $\Mmodbar^{fr}_{n_-,n_+}$ denote the stack $\coprod_{g \geq 0} \Mmodbar^{fr}_{g,n_-,n_+}$. 
\end{notation}

We now outline how to view $\Mbar$ as a $\oprdinfty$-algebra using these moduli spaces. We shall in fact show that $\Mbar$ can be realized as an algebra over $\oprd$, interpreted as a $2$-operad, in the (strict) $2$-category of stacks $\stk$. More concretely,  we will construct a lax $2$-functor of categories over $\Fin_*$
\[ \Mbar_2 \colon \oprdten \to \stk^\otimes.\]
Here
\begin{enumerate}
	\item $\stk^\otimes$ is the category over $\Fin_*$ associated to the strict $2$-category $\stk$ of stacks as in \eqref{eq: tensor category 2}, where the monoidal product is given by product of stacks. (The construction in \eqref{eq: tensor category 2} is described for $1$-categories, but generalizes to the case of strict $2$-categories in a straightforward manner).
	\item $\oprdinfty$ is the category over $\Fin_*$ associated with the colored operad $\oprd$ as in \eqref{eq: tensor category 1}.
\end{enumerate}
For an object $\un{n}^1, \un{n}^2,...,\un{n}^k$ in $\oprdten$ define
\[\Mbar_2( \un{n}^1, \un{n}^2,...,\un{n}^k) : = \{\Mbar^{fr}_{\un{n}^i}\}_{i=1}^k\]
Action of the functor $\Mbar_2$ on morphisms is defined as follows: as described in \ref{subsubsec: infty proprd oprd}, the morphisms in the category $\oprdten$ are given by a sequence of vertex-ordered \rda-graphs.  We describe how $\Mbar_2$ acts on the morphisms of type $\oprdten(\un{n}^1, \un{n}^2,...,\un{n}^k;\un{n}_+)$, in other words on morphisms given by a single vertex-ordered \rda-graphs. The definition extends to morphisms described by disconnected graphs in a straightforward manner.\\
Let $G$ be a (connected) vertex-ordered \rda-graph describing a morphism form $\un{n}^1, \un{n}^2,...,\un{n}^k$ to $\un{n}_+$. We define a map
\[ \mu_G(\Mbar_2) \colon\Mbar_2( \un{n}^1, \un{n}^2,...,\un{n}^k)  = \prod_{i}\Mbar^{fr}_{\un{n}^i} \to \Mbar_2( \un{n}_+) = \Mbar^{fr}_{\un{n}^1}\]
Let $T$ be any test space, we will describe the functor 
\[\mu_G(\Mbar_2)(T) \colon \prod \Mbar^{fr}_{\un{n}^i}(T) \to \Mbar_2( \un{n}_+)(T)\]
Here,
\begin{enumerate}
\item $\prod_{i}\Mbar^{fr}_{\un{n}^i}(T)$ is the groupoid of tuples $\{\catc^i \to T\}_{i=1}^n$ such that the curves in $\catc^i$ have input-output profile $\un{n}^i$
\item $\Mbar_2( \un{n}_+)(T)$ is the groupoid of families $\catc_{\un{n}_+} \to T$ with input-output profile $\un{n}_+$
\end{enumerate}
The groupoid map $\mu_G(\Mbar_2)(T)$ is induced by mapping a tuple of families over $T$ as in $(1)$ to the family obtained by gluing them as prescribed by the edges of $G$ (followed by stabilization).\\
Now let $G$ and $H$ be two graphs describing a pair of composible morphisms in $\oprdten$. Given a test space $T$ and a tuple of families of curves $\{\catc_v \to T\}_{v \in V(G)}$ indexed by the vertices of $G$, we note that the two families over $T$ obtained by 
\begin{itemize}
\item first gluing the families along edges of $G$ and then gluing the resulting families along $H$, and 
\item gluing the families along the composed graph $H \circ G$
\end{itemize}
are canonically isomorphic. Thus there exist canonical isomorphisms 
\[\mu_H(\Mbar_2)(T) \circ \mu_G(\Mbar_2)(T) \Rightarrow \mu_{H \circ G}(T).\]
These isomorphisms are natural in $T$ and hence it follows that there exists a \emph{unique} $2$-isomorphism $\alpha \colon \Mbar_2 (H) \circ \Mbar_2(G) \Rightarrow \Mbar_2(H \circ G)$.\\
Thus we have a lax $2$-functor 
\[\Mbar_2 \colon \oprdten \to \stk^\otimes \]
as desired. \\
From the construction of $\Mbar_2$ it is clear that the induced map of $\infty$-categories over $\Fin_*$ preserves inert morphisms and thus defines a map of $\infty$-operads.\\
Note that $\stk$ is a subcategory of the $\infty$-category $\stk_{\infty}$ and the inclusion $\stk \hookrightarrow \stk_\infty$ induces a map of $\infty$-operads $\stk^\otimes \to \stk_{\infty}^\otimes$. Composing $\Mbar_2$ with this map then gives a map the desired map of $\infty$-operads:
\[\Mbar \colon \oprdten \to \stk^\otimes\]
exhibiting $\Mbar$ as a $\oprdinfty$-algebra.\\

%


\appendix

\section{Proofs from Section \ref{sec: model category str and cofib resolutions}}\label{apndx: Cofibrant Resolutions via The Bar Construction}
The aim of this appendix is to fill in technical details which were left out in Section \ref{subsec: Cofibrant Resolutions via The Bar Construction}. 

In Section \ref{subsec: appendix tensor cotensor} we explain in detail the fact that the category of topological \reduced properads is tensored and cotensored over the category of topological spaces. We also prove Lemma \ref{complem} which describes the compatibility of the tensor and cotensor operations with the model structure on \reduced properads. In Section \ref{subsec: appendix geometric realization}, we recall the definition of geometric realization from Section \ref{subsec: Cofibrant Resolutions via The Bar Construction}, and provide its alternate description in terms of \emph{latching spaces} which was alluded to there. In Section \ref{subsec: appendix cofib of bar}, we use this description to prove Proposition \ref{prop: cofibrancy of bar construction} and Corollary \ref{corr: cofibrancy of bar construction} from Section \ref{subsec: Cofibrant Resolutions via The Bar Construction}. In Section \ref{subsec: appendix geometric realizations iso} we provide the proof of Proposition \ref{prop: geometric realizations iso} asserting that the geometric realizations of simplicial \reduced properads in the categories of topological \reduced sequences and \reduced properads coincide. Finally, in Section \ref{subsec: app compareproposition}, we provide the details of the proof of Proposition \ref{compareproposition}.


\subsection{Tensor and Cotensor over topological spaces}\label{subsec: appendix tensor cotensor}	
The category of \reduced properads is tensored and cotensored over topological spaces. This means that, along with the bifunctor 
\[  [\ ,\ ] \colon \Galg \times \Galg^{op} \to \Top \]
providing an enrichment of the category of \reduced properads over topological spaces, there are two additional bifunctors 
\begin{align*}
\odot &\colon \Galg \times \Top \to \Galg, \mbox{  and }\\ 
(\_)^{\_} &\colon \Galg \times \Top^{op} \to \Galg 
\end{align*}
such that there are isomorphisms
\[ {\Galg}(P \odot Z, Q) \simeq {\Top} ( Z, [P,Q]) \simeq {\Galg}(P,Q^Z) ,\]
natural in all the variables involved. These bifunctors are constructed as follows:

The enrichment of topological \reduced properads over topological spaces is given by endowing the properad hom-spaces with the natural topology obtained by identifying them as subspaces of the hom-spaces of topological \reduced sequences. In other words, $[P,Q]$ is the space of \reduced properad maps form $P$ to $Q$ with the product compact-open topology. Equivalently, if $[\ ,\ ]^{\redTopSeq}$ denotes the enrichment of topological \reduced sequences over topological spaces, $[ P,Q ]$ is defined as the equalizer
\[ [ P ,Q ]  \to [ P , Q ]^{\redTopSeq} \rightrightarrows [\GG P , Q ]^{\redTopSeq} ,\]
where the first arrow is induced by pullback along the properad structure map $\GG P \to P$ and the second map is the composition $[P,Q]^{\redTopSeq} \to [\GG P, \GG Q]^{\redTopSeq} \to  [\GG P, Q]^{\redTopSeq}$

The operation $\odot$ can be constructed as 
\[ P \odot Z := coeq \Big[ \GG(\GG P \times Z) \rightrightarrows \GG(P \times Z) \Big] .\]
Here, for $X$ a topological \reduced sequence and $Z$ a topological space, $X \times Z$ denotes the topological \reduced-sequence given by component-wise product $\{ X(n_-,n_+) \times Z\}_{n_-,n_+}$. The coequalizer can be taken in either topological \reduced properads or topological \reduced sequences, since they both coincide in this case. The first arrow is induced from the properad structure map $\GG P \to P$, whereas the second arrow is induced from $\GG P \times Z \to \GG( P \times Z)$ using the universal property of free properads.

The operation $(\_)^{\_}$ can be constructed as follows: The underlying sequence of $P^Z$ is given by $P^Z=\{{\Top}(Z,P(n_-,n_+))\}_{n_-,n_+}$ and the properad structure is induced from that of $P$.\\

These operations are compatible with the model structure in the following sense:
\begin{lemma}[\cite{rezk1996spaces}]\label{complem}
	The following equivalent conditions are satisfied:
	\begin{enumerate}
		\item \label{lem1} If $i \colon Y \to Z$ is a cofibration of topological spaces and $j \colon P \to Q$ is a cofibration of \reduced properads, then \[P \odot Z \bigsqcup_{P \odot Y}^{\Galg} Q \odot Y \to Q \odot Z\] is a cofibration of \reduced properads. Moreover this map is a weak equivalence whenever $i$ or $j$ is.
		\item  \label{lem2} If $i \colon Y \to Z$ is a cofibration of topological spaces  and $j \colon P \to Q$ is a fibration of \reduced properads, then \[P^Z \to P^Y \times_{Q^Y} Q^Z\] is a fibration of properads. Moreover this map is a weak equivalence whenever $i$ or $j$ is.
		\item If $j \colon P \to Q$ is a cofibration and $k \colon R \to S$ a fibration of \reduced properads, then \[[Q,R] 	\to [P,R] \times_{[P,S]} [Q,S]\] is a fibration. Again, this map is a weak equivalence whenever $j$ or $k$ is.
	\end{enumerate}
\end{lemma}
	\begin{proof}
	We explain the proof of (\ref{lem2}) and the implication (\ref{lem2}) $\Rightarrow$ (\ref{lem1}), since this is the only part we shall use.

	(\ref{lem2}): Note that we need to show $P^Z \to P^Y \times_{Q^Y} Q^Z$ is a fibration of underlying topological \reduced sequences. The fibered product $P^Y \times_{Q^Y} Q^Z$ coincides with the fibered product of the underlying topological \reduced sequences. Let $V \to W$ be an acyclic cofibration in topological \reduced sequences. We need to show that any square
\[	\begin{tikzcd}
		V \ar{r}\ar{d} & P^Z \ar{d}\\
		W \ar{r}& P^Y \times_{Q^Y} Q^Z
	\end{tikzcd}
\]
	admits a lift $W \to P^Z$ (in topological \reduced sequences). This is equivalent to the statement that any square
\[
	\begin{tikzcd}
		W \times Y \sqcup_{V \times Y} V \times Z \ar{r}\ar{d} &P \ar{d} \\
		W \times Z \ar{r} &Q
	\end{tikzcd}
\]
	admits a lift $W \times Z \to P$ (in topological \reduced sequences). But this follows from the fact that
	\begin{itemize}
		\item $P \to Q$ is a fibration of \reduced properads, and hence of topological \reduced sequences as well
		\item $Y \to Z$ is a cofibration of topological spaces and $V \to W$ is an acyclic cofibration of topological \reduced sequences, and hence $W \times Y \sqcup_{V \times Y} V \times Z \to W \times Z$ is an acyclic cofibration in topological \reduced sequences.
	\end{itemize}
	The conclusion when $i$ or $j$ is a weak equivalence follows by a similar argument, starting with a cofibration $V \to W$ which is not necessarily a weak equivalence.  
		This completes the proof of (\ref{lem2}).
		
	(\ref{lem2}) $\Rightarrow$ (\ref{lem1}): To show that $ P \odot Z \bigsqcup_{P \odot Y}^{\Galg} Q \odot Y \to Q \odot Z$ is a cofibration of topological \reduced properads, we show that for any acyclic fibration $E \to F$ in \reduced properads the square
\[
	\begin{tikzcd}
			P \odot Z \bigsqcup_{P \odot Y}^{\Galg}  Q \odot Y \ar{r}\ar{d} & E \ar{d}\\
			Q \odot Z \ar{r} & F
	\end{tikzcd}
\]
admits a lift $Q \odot Z \to E$ (in \reduced properads). This is equivalent to 
\[
	\begin{tikzcd}
	P \ar{r}\ar{d} & E^Z \ar{d}\\
	Q \ar{r} &F^Y \times_{F^Y} F^Z
	\end{tikzcd}
\]
admitting a lift $Q \to E^Z$ (in \reduced properads). But this is a consequence of (\ref{lem2}) combined with the fact that $P \to Q$ is a properad cofibration.\\
Again, the conclusion when $i$ or $j$ is a weak equivalence follows by a similar argument, starting with a fibration $E \to F$ which is not necessarily a weak equivalence.  
	\end{proof}
	

\subsection{Geometric realization}\label{subsec: appendix geometric realization}
Now, using this tensor product on the category of \reduced properads over topological spaces, we can define the geometric realization of a simplicial \reduced properad $\{P_\bullet\}_{\bullet \geq 0}$ by the usual formula:
\begin{equation}\label{eqn: appendix geomrelform}
|P_\bullet|^{\Galg} = \int_{\Delta}^{\Galg} P_n \odot \Delta^n := coeq^{\Galg} \Big[ \coprod_{\phi \colon [n] \to [m] \in \Delta}^{\Galg} P_m \odot \Delta^n\rightrightarrows \coprod_{[n] \in \Delta}^{\Galg} P_n \odot \Delta^n .\Big]
\end{equation}
Here $\Delta$ denotes the simplex category, and $\Delta^n$ is the standard $n$-simplex.	In the coequalizer, the first arrow is induced from the simplicial structure map $P^*(\phi) \colon P_m \to P_n \in \Galg$ corresponding to $\phi \colon [n] \to [m]$ and the second arrow is induced from the map $\Delta(\phi) \colon \Delta^n \to \Delta^m \in \Top$.

There is an alternate description of $|\ \ |^{\Galg}$ as an iterated pushout in terms of the so called `latching spaces': \\
Define the $n$th \emph{latching space} of $P_\bullet$ to be the coequalizer
\[ L_nP_\bullet := coeq^{\Galg} \Big[ \coprod_{0 \leq i < j \leq n}^{\Galg} P_{n-1} \rightrightarrows \coprod_{0 \leq i \leq n}^{\Galg} P_n \Big] .\]
The two arrows are defined as follows: on the $(i,j)$th summand, the first arrow is induced by the degeneracy map $s_i \colon P_{n-1} \to P_n$ into the $j$th summand and the second arrow is induced from $s_{j-1} \colon P_{n-1} \to P_n$ into the $i$th summand. Note that there is a canonical map
\[ L_{n}P_\bullet \to P_{n+1} .\]
$L_n P_\bullet$ can be thought of intuitively as ``the space of degenerate simplices in $P_{n+1}$''.\\
Define $|\ \ |^{\Galg}_{(n)}$ inductively as:
	\begin{itemize}
	\item $|P_\bullet|^{\Galg}_{(0)} =  P_0$, and
	\item $|P_\bullet|^{\Galg}_{(n+1)}$ is defined to be the pushout
	\begin{equation}\label{dig: appendix latchingdig}
		\begin{tikzcd}
			L_nP_\bullet \odot \Delta^{n+1} \bigsqcup_{L_nP_\bullet \odot \partial \Delta^{n+1}}^{\Galg} P_{n+1} \odot \partial \Delta^{n+1} \ar{r}\ar{d} & P_{n+1} \odot \Delta^{n+1} \ar{d}\\
			{|P_\bullet|}^{\Galg}_{(n)} \ar{r} & {|P_\bullet|}^{\Galg}_{(n+1)}\\
		\end{tikzcd}
	\end{equation}
	\end{itemize}

Then we have $|P_\bullet|^{\Galg} = \varinjlim_n |P_\bullet|^{\Galg}_{(n)}$.\\
Using this description we can now reformulate the cofibrancy condition on $|P_\bullet|^{\Galg}$ in terms of the pushouts \eqref{dig: appendix latchingdig}:
\begin{lemma}\label{lem: cofib in latching}
	If $L_nP_\bullet \to P_{n+1}$ is a cofibration of \reduced properads for every $n$, then $|P_\bullet|^{\Galg}$ is a cofibrant \reduced properad.
\end{lemma}
\begin{proof}
	It suffice to prove that the upper horizontal arrow in the diagram \ref{dig: appendix latchingdig} is a cofibration for all $n$. This follows from the compatibility statements from Lemma \ref{complem} since $\partial \Delta^{n+1} \hookrightarrow \Delta^{n+1} $ is a cofibration of topological spaces and $P \to Q$ is a cofibration of \reduced properads.
\end{proof}

We have a relative version of Lemma \ref{lem: cofib in latching}:
\begin{lemma}\label{lem: rel cofib in latching}
	If $P_{\bullet} \to Q_{\bullet}$ is a map of simplicial \reduced properads such that $L_nQ_\bullet \bigsqcup_{L_nP_\bullet}^{\Galg} P_{n+1} \to Q_{n+1}$ is a cofibration of properads for every $n$, then $|P_\bullet|^{\Galg} \to |Q_\bullet|^{\Galg}$ is a cofibration of \reduced properads.
\end{lemma}
\begin{proof}
Using the filtrations on $|P_\bullet|$ and $|Q_\bullet|$ described by the diagram \eqref{dig: appendix latchingdig}, it suffices to prove that 
\begin{equation}\label{eq: rel bar pushout 1} 
	{|P_\bullet|}^{\Galg}_{(n+1)} \bigsqcup_{{|P_\bullet|}^{\Galg}_{(n)}}^{\Galg} {|Q_\bullet|}^{\Galg}_{(n)} \to {|Q_\bullet|}^{\Galg}_{(n+1)} 
\end{equation}
is a cofibration of \reduced properads. 
For convenience, set
\begin{align*}
	L^{\Delta}_nP_\bullet &:= L_nP_\bullet \odot \Delta^{n+1} \bigsqcup_{L_nP_\bullet \odot \partial \Delta^{n+1}}^{\Galg} P_{n+1} \odot \partial \Delta^{n+1} \mbox{, and }\\
	L^{\Delta}_nQ_\bullet &:= L_nQ_\bullet \odot \Delta^{n+1} \bigsqcup_{L_nQ_\bullet \odot \partial \Delta^{n+1}}^{\Galg} Q_{n+1} \odot \partial \Delta^{n+1}.
\end{align*} 
Then, considering the map of diagrams \eqref{dig: appendix latchingdig} induced by $P_\bullet \to Q_\bullet$, it can be proved that \eqref{eq: rel bar pushout 1} is equivalent to showing that 
\begin{equation}\label{eq: rel bar pushout 2}
	 L^{\Delta}_nQ_\bullet \bigsqcup_{L^{\Delta}_nP_\bullet}^{\Galg} P_{n+1} \odot \Delta^{n+1} \to Q_{n+1} \odot \Delta^{n+1} 
\end{equation}
is a cofibration of \reduced properads.
To prove \eqref{eq: rel bar pushout 2} observe that 
\[ L^{\Delta}_nQ_\bullet \bigsqcup_{L^{\Delta}_nP_\bullet}^{\Galg} P_{n+1} \odot \Delta^{n+1} = \left(L_n Q_\bullet \bigsqcup_{L_n P_\bullet}^{\Galg} P_{n+1}\right) \odot \Delta^{n+1} \coprod_{(L_n Q_\bullet \bigsqcup_{L_n P_\bullet}^{\Galg} P_{n+1}) \odot \partial \Delta^{n+1}}^{\Galg}  Q_{n} \odot  \partial \Delta^{n+1}\]
\eqref{eq: rel bar pushout 2} now follows by Lemma \ref{complem} from the hypothesis that $L_nQ_\bullet \sqcup_{L_nP_\bullet}^{\Galg} P_{n+1} \to Q_{n+1}$ is a cofibration.
\end{proof}


\subsection{Cofibrancy of the bar construction}\label{subsec: appendix cofib of bar}
We now provide the proof of Proposition \ref{prop: cofibrancy of bar construction}. Recall that for an \reduced properad $P$ its bar construction is a simplicial properad which we denote by $B_{\bullet}(\GG,\GG,P)$. 

Let us start by recalling the statement of Proposition \ref{prop: cofibrancy of bar construction}:
\begin{proposition}\label{prop: appendix cofibrancy of bar construction}
If the topological \reduced sequence underlying a \reduced properad $P$ is cofibrant, then $|B(\GG,\GG,P)|^{\Galg}$ is cofibrant in the category of \reduced properads.
\end{proposition}
\begin{proof}
The latching objects for the bar construction are given by 
\[ L_n B(\GG,\GG,P) = coeq^{\Galg} \Big[ \coprod_{0 \leq i < j \leq n}^{\Galg}\GG^{n}P \rightrightarrows \coprod_{0 \leq i \leq n}^{\Galg} \GG^{n+1}P \Big] .\]
The free $\GG$-algebra functor, being the left adjoint in the pair $\GG \colon \redTopSeq \leftrightarrows \Galg \colon$ Forget, commutes with coproudcts and coequalizer. Thus,
\[ L_nB(\GG,\GG,P) = \GG \Big( coeq \Big[ \coprod_{0 \leq i < j \leq n}\GG^{n-1}P \rightrightarrows \coprod_{0 \leq i \leq n}\GG^{n}P \Big] \Big) .\]
Set 
\[ K_nP := coeq \Big[ \coprod_{0 \leq i < j \leq n}\GG^{n-1}P \rightrightarrows \coprod_{0 \leq i \leq n}\GG^{n}P \Big] .\]
Then, to show that $L_nB(\GG,\GG,X) \to B_{n+1}(\GG,\GG,P) = \GG^{n+2}P$ is a cofibration of  \reduced properads, it suffices to show that $K_nP \to \GG^{n+1}P$ is a cofibration of topological \reduced sequences. But, it is not difficult to see that this map identifies $K_nP$ with the subspace of $\GG^{n+1}P$ given by $\bigcup_{0 \leq i \leq n} s_i(\GG^n P)$, where $s_i$ are the simplicial degeneracy maps.\\
Moreover, note that $\GG^{n+1}P$ is a disjoint union of spaces indexed by $(n+1)$-nested graphs and $s_i(\GG^nP)$ is the union corresponding to a subset of this indexing set. Since the topological \reduced sequence underlying $P$ is cofibrant, it follows that the inclusion $\bigcup_{0 \leq i \leq n}s_i(\GG^n P) \hookrightarrow \GG^{n+1}P$ is a cofibration of topological \reduced sequences. \end{proof}

A similar argument using Lemma \ref{lem: rel cofib in latching} gives the relative version of Proposition \ref{prop: appendix cofibrancy of bar construction}. We omit the details:
\begin{proposition}\label{prop: appendix rel cofibrancy of bar construction}
Let $P, Q$ be properads which are cofibrant as topological \reduced sequences and let $P \to Q$ be a map of \reduced properads such that the underlying map of topological \reduced sequences is a cofibration. Then, $|B(\GG,\GG,P)|^{\Galg} \to |B(\GG,\GG,Q)|^{\Galg}$ is a cofibration of properads.\\
\null \hfill \qedsymbol
\end{proposition}

\subsection{Comparison of geometric realizations in topological \reduced sequences and in topological \reduced properads}\label{subsec: appendix geometric realizations iso}\ \\
The aim of this section is to provide the proof of Proposition \ref{prop: geometric realizations iso}. In this section we follow the discussion in \cite[Section 7]{mandell2019operads}.\\
Let us start by recalling some notation: Let $P_{\bullet}$ be a simplicial topological \reduced properad. Denote by $|P_{\bullet}|^{\Galg}$ and $|P_{\bullet}|^{\redTopSeq}$ the geometric realizations of $P$ in the category of properads and the category of topological \reduced sequences, respectively. As a consequence of the fact that the geometric realization of simplicial spaces is a monoidal functor, it follows that $|P_{\bullet}|^{\redTopSeq}$ carries a natural properad structure induced from the structure maps of the simplicial properad $P_{\bullet}$.\\

Proposition \ref{prop: geometric realizations iso} asserts that the \reduced properads given by geometric realizations of $P_{\bullet}$ in the category of \reduced properads and in the category of topological \reduced sequences are isomorphic.

\begin{proof}[Proof of Proposition \ref{prop: geometric realizations iso}]
	Recall that 
	\[P_n \odot \Delta^n = coeq \Big[ \GG(\GG P_n \times \Delta^n) \rightrightarrows \GG(P_n \times \Delta^n) \Big]\]
	(where the coequalizers taken in \reduced properads and \reduced sequences coincide).\\
	 The maps $P_n \times \Delta^n \to |P_\bullet|^{\redTopSeq}$ induce properad maps $P_n \odot \Delta^n \to |P_\bullet|^{\redTopSeq}$. Moreover, these maps are compatible with the simplicial face and degeneracy maps, and hence we get a $\GG$-algebra map
\[ coeq \Big[ \coprod_{m,n}^{\Galg} P_m \odot \Delta^n \rightrightarrows \coprod_{n}^{\Galg} P_n \odot \Delta^n \Big] = |P_\bullet|^{\Galg} \to |P_\bullet|^{\redTopSeq} .\]
Now, notice that for properad $P$, 
\begin{equation}\label{eq: geom rel prop 1} |\GG P_\bullet|^{\redTopSeq} \simeq \GG|P_\bullet|^{\redTopSeq}.\end{equation}
This holds since $|\ \  |^{\redTopSeq}$, being a left adjoint ($|\ \ |^{\redTopSeq}\colon \redTopSeq \leftrightarrows \redTopSeq \colon (\_)^{\Delta^\bullet}$), commutes with colimits and as noted above it also commutes with finite products.

Furthermore, for any $X \in \redTopSeq$, we have 
\begin{equation}\label{eq: geom rel prop 2}|\GG X_\bullet|^{\Galg} \simeq \GG |X_\bullet|^{\redTopSeq}.\end{equation}
To see this, first note that for a topological \reduced sequence $X$
\begin{align*}
 \GG X_n \odot \Delta^n = coeq \Big[ \GG (\GG(\GG X_n) \times \Delta^n ) \rightrightarrows \GG(\GG X_n \times \Delta^n) \Big] \simeq \GG(X_n \times \Delta^n) .
\end{align*}
This implies that
\begin{align*}
\coprod_{m,n}^{\Galg} \GG X_n \odot \Delta^m &\simeq \GG \Big( \coprod_{m,n}X_n \times \Delta^m \Big), \mbox{ and } \\
 \coprod_{n}^{\Galg} \GG X_n \odot \Delta^n &\simeq \GG \Big( \coprod_{n} X_n \times \Delta^m \Big) .
\end{align*}
Since $\GG$ commutes with colimits, this gives us
\begin{equation*}\begin{split}
 coeq^{\Galg} \Big[ \coprod & \GG X_n \odot \Delta^m \rightrightarrows  \coprod \GG X_n \odot \Delta^n\Big] \simeq\\
 & \GG \Big( coeq \Big[ \coprod X_n \times \Delta^m \rightrightarrows  \coprod X_n \times \Delta^n \Big] \Big) .
\end{split}
\end{equation*}

Thus, it follows that $|\GG X_\bullet|^{\Galg} \simeq \GG |X_\bullet|^{\redTopSeq}$ for any topological \reduced sequence $X$.
Finally, for any simplicial properad $P_\bullet $
\begin{equation}\label{eq: geom rel prop 3}\GG(\GG P ) \rightrightarrows \GG P \to P \end{equation}
is a coequalizer diagram (in \reduced properads as well as in topological \reduced sequences). Since $|\ \ |^{\Galg}$ is a left adjoint ($|\ \ |^{\Galg} \colon s\Galg \leftrightarrows \Galg \colon (\_)^{\Delta^\bullet}$), it commutes with the coequalizer. Thus, using \eqref{eq: geom rel prop 3}, we have the following commutative diagram in the category of \reduced properads:
\[
	\begin{tikzcd}
		{|\GG(\GG P) |}^{\Galg} \ar[shift left]{r}\ar[shift right]{r} \ar{d} & {|\GG P |}^{\Galg} \ar{r}\ar{d} & {|P_\bullet|}^{\Galg} \ar[dashed]{d}\\
		\GG \GG {|P_\bullet|}^{\redTopSeq}\ar[shift left]{r}\ar[shift right]{r} & {\GG|P_\bullet|}^{\redTopSeq} \ar{r} &  {|P_\bullet|}^{\redTopSeq}
	\end{tikzcd}
\]
where both the rows are coequalizer diagrams mentioned  above. Further, using \eqref{eq: geom rel prop 1} and \eqref{eq: geom rel prop 2} it follows that the first two vertical maps are isomorphisms. It thus follows that the rightmost vertical arrow is also an isomorphism, proving the proposition.
\end{proof}


\subsection{Proof of Proposition \ref{compareproposition}}\label{subsec: app compareproposition}
Let us first recall the statement of Proposition \ref{compareproposition}:
\begin{proposition}\label{prop: appendix comparepropositiondig}
Let 
\begin{equation}\label{dig: appendix comparepropositiondig}
	\begin{tikzcd}
	P \ar{d} & R \ar{l}\ar{r}\ar{d} & Q \ar{d} \\
	P^{\prime}  & R^{\prime} \ar{l}\ar{r} & Q^{\prime} 
	\end{tikzcd}
\end{equation}
be a map of pushout diagrams of \reduced properads. If
	\begin{enumerate}
		\item each vertical arrow is a \hur weak-equivalence, and
		\item $R \to P$ , $R^{\prime} \to P^{\prime}$ are \hur cofibrations of topological \reduced sequences
	\end{enumerate}
	Then 
	\[ |B_\bullet P| \coprod^{\Galg}_{|B_\bullet R|} |B_\bullet Q| \to |B_\bullet P^{\prime}| \coprod^{\Galg}_{|B_\bullet R^{\prime}|}|B_\bullet P^{\prime}|\]
	is a \hur weak-equivalence.
\end{proposition}
Note that here we continue to follow the convention mentioned in Notation \ref{notn: geom real}, namely $|\_|$ denotes the common geometric realization in topological \reduced  properads and topological \reduced sequences.

\subsubsection{\hur model structure on topological \reduced sequences}\label{subsec: app hur model str}
Before beginning the proof of Proposition \ref{compareproposition}, we start with a short discussion of the \hur model structure on topological \reduced sequences, alluded to in Remark \ref{rmk: hur weq cofib}:

There exists a model structure on the category $\Top$, called the \emph{\hur or Str\o{}m model structure}, in which the weak-equivalences, fibrations, and cofibrations are given, respectively, by homotopy equivalences, \hur fibrations, and \hur cofibrations (see \cite[Section 17.1]{may2011more}). There is an induced model structure on $\redTopSeq$ with weak-equivalences, fibrations, and cofibrations defined component-wise. We will refer to this as the \hur model structure on $\redTopSeq$.\\
With these definitions, the \hur weak equivalences and  cofibrations defined in Notation \ref{notn: hur weq cofib} correspond precisely to the weak-equivalences and cofibrations in the \hur model structure on topological \reduced sequence.

	\subsubsection{Proof of Proposition \ref{compareproposition}}
		Note that,
	\begin{align}
	\label{barpushout}	|B_\bullet P| \coprod_{|B_\bullet R|}^{\Galg}  |B_\bullet Q|  & \simeq \Big|B_\bullet P \coprod_{B_\bullet R}^{\Galg} B_\bullet Q\Big| \simeq  \Big|\GG^{\bullet+1}  P  \coprod_{\GG^{\bullet+1}  R}^{\Galg} \GG^{\bullet+1} Q \Big| \\
	\notag & \simeq  \Big| \GG \big(\GG^{\bullet}  P \coprod_{\GG^{\bullet}  R} \GG^{\bullet} Q \big) \Big|
	\end{align}
	Here the first equality holds since $|\ \ |$ is a left adjoint in the pair $| \ \ | \colon \leftrightarrows \Galg \colon (\_)^{\Delta^\bullet}$ and hence commutes with colimits. Similarly the last equality holds since $\GG$ is a left adjoint in the pair $\GG \colon \redTopSeq \leftrightarrows  \Galg \colon Forget$ and hence commutes with colimits.
	Similarly, 
	\[ |B_\bullet P^{\prime}| \coprod_{|B_\bullet P^{\prime}|}^{\Galg} |B_\bullet Q^{\prime}|  \simeq  \Big| \GG \big(\GG^{\bullet}  P^{\prime} \coprod_{\GG^{\bullet}  R^{\prime}} \GG^{\bullet} Q^{\prime} \big) \Big| \]
	For the rest of the proof, we shall consider the category of topological \reduced sequences equipped with the \hur model structure (see Section \ref{subsec: app hur model str} above). On the category of simplicial topological \reduced sequences we consider the Reedy model structure induced from the \hur model structure on topological \reduced sequences (see \cite{riehl2013theory} for details on Reedy model structures).
	
	Consider the simplicial topological \reduced sequence 
	\[\GG \big( \GG^{\bullet}  P \coprod_{\GG^{\bullet}  R} \GG^{\bullet} Q \big).\]
	Its $(n+1)$-simplices are given by a disjoint union of spaces indexed by certain $(n+1)$-nested \rda-graphs. As in the proof of Proposition \ref{prop: appendix cofibrancy of bar construction}, it can be seen that the latching map
	\begin{equation}\label{eq: latching eq hur} L_n \left( \GG \big( \GG^{\bullet}  P \coprod_{\GG^{\bullet}  R} \GG^{\bullet} Q \big) \right) \to \GG \big( \GG^{n+1}  P \coprod_{\GG^{n+1}  R} \GG^{n+1} Q \big) 
	\end{equation}
	identifies the latching space with the disjoint union of spaces corresponding to a subset of this indexing set. Using the fact that all topological \reduced sequences are cofibrant in the \hur model structure, it then follows that map \eqref{eq: latching eq hur} is a \hur cofibration as well. We conclude that 
	\[\GG \big( \GG^{\bullet}  P \coprod_{\GG^{\bullet}  R} \GG^{\bullet} Q \big)\]
	is a Reedy cofibrant simplicial topological \reduced sequence. Similarly, it follows that 	
	\[\GG \big(\GG^{\bullet}  P^{\prime} \coprod_{\GG^{\bullet}  R^{\prime}} \GG^{\bullet} Q^{\prime} \big)\]
	is also Reedy cofibrant.  Applying \cite[Corollary 10.6]{riehl2013theory}, it suffice to show that the map
	\[   \GG \big(\GG^{\bullet}  P \coprod_{\GG^{\bullet}  R} \GG^{\bullet} Q \big) \to  \GG \big(\GG^{\bullet}  P^{\prime} \coprod_{\GG^{\bullet}  R^{\prime}} \GG^{\bullet} Q^{\prime} \big)  \]
	is a Reedy weak-equivalence i.e. a level-wise \hur weak-equivalence. It follows from the definition of $\GG$ that it preserves \hur weak-equivalences and thus it suffices to verify that 
	\[ \GG^{\bullet}  P \coprod_{\GG^{\bullet}  R} \GG^{\bullet} Q  \to \GG^{\bullet}  P^{\prime} \coprod_{\GG^{\bullet}  R^{\prime}} \GG^{\bullet} Q^{\prime} \]
	is a \hur weak-equivalence.\\
	Applying $\GG^{\bullet}$ to the diagram \eqref{dig: appendix comparepropositiondig}, we get a diagram of topological \reduced sequences 
	\begin{equation}\label{app: comparepropositiondig 2}
		\begin{tikzcd}
		\GG^{\bullet}P \ar{d} & \GG^{\bullet}R \ar{l}\ar{r}\ar{d} & \GG^{\bullet}Q \ar{d} \\
		\GG^{\bullet}P^{\prime}  & \GG^{\bullet}R^{\prime} \ar{l}\ar{r} & \GG^{\bullet}Q^{\prime} 
		\end{tikzcd}
	\end{equation}
	satisfying the following conditions:
	\begin{itemize}
		\item the vertical maps are \hur weak equivalences, and
		\item $\GG^{\bullet}R \to \GG^{\bullet}P$ , $\GG^{\bullet}R^{\prime} \to \GG^{\bullet}P^{\prime}$ are \hur cofibrations of topological \reduced sequences.
	\end{itemize}
The \hur model structure on $\redTopSeq$ is left proper (see for example \cite[Proposition 15.4.2]{may2011more}) and hence the pushout of any map along a cofibration coincides with its homotopy pushout . In particular, the pushouts of both the rows in \eqref{app: comparepropositiondig 2} coincide with their homotopy pushouts. Since the vertical arrows in \eqref{app: comparepropositiondig 2} are \hur weak-equivalences, it then follows that the two pushouts are \hur weak-equivalent. This completes the proof of the proposition.\\
\null \hfill \qedsymbol

\printbibliography
\end{document}